\documentclass{tac}
\pdfoutput=1

\newif\ifcoloronly
\newif\ifbwonly
\newif\ifbwcompromise

\coloronlytrue     
\bwonlyfalse       
\bwcompromisefalse 

\usepackage{versions}
\ifcoloronly\includeversion{coloronly}\else\excludeversion{coloronly}\fi
\ifbwonly\includeversion{bwonly}\else\excludeversion{bwonly}\fi
\ifbwcompromise\includeversion{bwcompromise}\else\excludeversion{bwcompromise}\fi

\usepackage{amssymb,amsmath,stmaryrd,mathrsfs}
\usepackage[all,2cell]{xy}
\UseAllTwocells
\usepackage{enumitem}
\usepackage{xcolor}
\definecolor{darkgreen}{rgb}{0,0.45,0}
\usepackage[pagebackref,colorlinks,citecolor=darkgreen,linkcolor=darkgreen]{hyperref}
\usepackage[noadjust]{cite}
\usepackage{subfigure}

\usepackage{afterpage}
\usepackage{pdflscape}
\usepackage{array}
\usepackage{mathtools}
\usepackage{tikz}
\usepackage{xspace}

\makeatletter
\newif\ifhyperref
\@ifpackageloaded{hyperref}{\hyperreftrue}{\hyperreffalse}
\let\your@state\state
\def\state#1{\my@state#1}
\def\my@state#1.{\gdef\currthmtype{#1}\your@state{#1.}}
\let\your@staterm\staterm
\def\staterm#1{\my@staterm#1}
\def\my@staterm#1.{\gdef\currthmtype{#1}\your@staterm{#1.}}
\def\currthmtype{}
\ifhyperref
\def\autoref#1{\ref*{label@name@#1}~\ref{#1}}
\else
\def\autoref#1{\ref{label@name@#1}~\ref{#1}}
\fi
\AtBeginDocument{%
  \let\old@label\label%
  \def\label#1{%
    {\let\your@currentlabel\@currentlabel%
      \edef\@currentlabel{\currthmtype}%
      \old@label{label@name@#1}}%
    \old@label{#1}}
}

\newtheorem{thm}{Theorem}[section]

\newtheorem{cor}{Corollary}
\newtheorem{prop}{Proposition}
\newtheorem{lem}{Lemma}
\newtheoremrm{defn}{Definition}
\newtheoremrm{notn}{Notation}
\newtheoremrm{rmk}{Remark}
\newtheoremrm{eg}{Example}
\newtheoremrm{egs}{Examples}

\newcommand{\sB}{\ensuremath{\mathscr{B}}}
\newcommand{\sC}{\ensuremath{\mathscr{C}}}
\newcommand{\bbZ}{\ensuremath{\mathbb{Z}}}
\newcommand{\bC}{\ensuremath{\mathbf{C}}}

\newcommand{\bS}{\ensuremath{\mathbf{S}}}
\newcommand{\bT}{\ensuremath{\mathbf{T}}}
\newcommand{\fa}{\ensuremath{\mathfrak{a}}}
\newcommand{\fl}{\ensuremath{\mathfrak{l}}}

\newcommand{\fr}{\ensuremath{\mathfrak{r}}}

\newcommand{\ten}{\ensuremath{\otimes}}
\newcommand{\Id}{\ensuremath{\operatorname{Id}}}
\newcommand{\id}{\ensuremath{\operatorname{id}}}
\newcommand{\Ho}{\ensuremath{\operatorname{Ho}}}

\newcommand{\op}{\ensuremath{^{\mathit{op}}}}

\newdir{ >}{{}*!/-10pt/@{>}}
\newcommand{\iso}{\cong}
\newcommand{\eqv}{\simeq}
\newcommand{\too}[1][]{\ensuremath{\overset{#1}{\longrightarrow}}}
\newcommand{\ot}{\ensuremath{\leftarrow}}
\newcommand{\maps}{\colon}
\newcommand{\rdual}[1]{{{#1}^{\bigstar}}}
\newcommand{\tr}{\ensuremath{\operatorname{tr}}}

\def\calBi#1#2{\ensuremath{\mathscr{#1}\!/\!_{\mathbf{#2}}}}
\newcommand{\ep}{\ensuremath{\varepsilon}}

\newcommand{\Set}{\ensuremath{\mathbf{Set}}}
\newcommand{\Ab}{\ensuremath{\mathbf{Ab}}}
\newcommand{\Top}{\ensuremath{\mathbf{Top}}}
\newcommand{\Sp}{\ensuremath{\mathbf{Sp}}}
\newcommand{\bCh}[1]{\ensuremath{\mathbf{Ch}_{#1}}}
\newcommand{\bEx}[1]{\ensuremath{\mathbf{Sp}_{#1}}}
\newcommand{\bMod}[1]{\ensuremath{\mathbf{Mod}}_{#1}}
\newcommand{\Cat}{\ensuremath{\mathcal{C}\mathit{at}}}
\newcommand{\SymMonCat}{\ensuremath{\mathcal{S}\mathit{ymMonCat}}}

\let\xto\xrightarrow
\def\slashedarrowfill@#1#2#3#4#5{%
  $\m@th\thickmuskip0mu\medmuskip\thickmuskip\thinmuskip\thickmuskip
   \relax#5#1\mkern-7mu%
   \cleaders\hbox{$#5\mkern-2mu#2\mkern-2mu$}\hfill
   \mathclap{#3}\mathclap{#2}%
   \cleaders\hbox{$#5\mkern-2mu#2\mkern-2mu$}\hfill
   \mkern-7mu#4$%
}
\def\rightslashedarrowfill@{%
  \slashedarrowfill@\relbar\relbar\mapstochar\rightarrow}
\newcommand\xslashedrightarrow[2][]{%
  \ext@arrow 0055{\rightslashedarrowfill@}{#1}{#2}}
\def\hto{\xslashedrightarrow{}}

\newcommand{\sh}[1]{{\ensuremath{\hspace{1mm}\makebox[-1mm]{$\langle$}%
      \makebox[0mm]{$\langle$}\hspace{1mm}{#1}\makebox[1mm]{$\rangle$}%
      \makebox[0mm]{$\rangle$}}}}
\newcommand{\bigsh}[1]{{\ensuremath{\hspace{1mm}%
      \makebox[-1mm]{$\big\langle$}\makebox[0mm]{$\big\langle$}%
      \hspace{1mm}{#1}\makebox[1mm]{$\big\rangle$}%
      \makebox[0mm]{$\big\rangle$}}}}

\renewcommand{\theenumi}{(\roman{enumi})}

\let\c@equation\c@subsection
\numberwithin{equation}{section}

\DeclareRobustCommand\widecheck[1]{{\mathpalette\@widecheck{#1}}}
\def\@widecheck#1#2{%
  \setbox\z@\hbox{\m@th$#1#2$}%
  \setbox\tw@\hbox{\m@th$#1%
    \widehat{%
      \vrule\@width\z@\@height\ht\z@
      \vrule\@height\z@\@width\wd\z@}$}%
  \dp\tw@-\ht\z@
  \@tempdima\ht\z@ \advance\@tempdima2\ht\tw@ \divide\@tempdima\thr@@
  \setbox\tw@\hbox{%
    \raise\@tempdima\hbox{\scalebox{1}[-1]{\lower\@tempdima\box
        \tw@}}}%
  {\ooalign{\box\tw@ \cr \box\z@}}}

\usetikzlibrary{scopes,calc,arrows,patterns}
\tikzset{ed/.style={auto,inner sep=0pt,font=\scriptsize}} 
\tikzset{>=stealth'}
\tikzset{vert/.style={draw,circle,inner sep=1pt,fill=white}}
\tikzset{vert2/.style={draw,circle,inner sep=2pt,fill=white}}

\colorlet{myblue}{blue!40!white}
\colorlet{myred}{red!35!white}
\colorlet{mygreen}{green!30!white}
\colorlet{myyellow}{yellow!10!white}
\tikzset{bluefill/.style={fill=myblue}}
\tikzset{redfill/.style={fill=myred}}
\tikzset{greenfill/.style={fill=mygreen}}
\tikzset{yellowfill/.style={fill=myyellow}}
\tikzset{dotsF/.style={pattern=north east lines,pattern color=black!60!white}}
\tikzset{dotsG/.style={pattern=north west lines,pattern color=black!60!white}}
\colorlet{mydarkred}{red!80!black}
\colorlet{mydarkblue}{blue!70!black}

\begin{bwcompromise}
\gdef\setinnerouter{\tikzset{inner/.style={mydarkblue,densely dotted}}
  \tikzset{outer/.style={mydarkred,densely dashed}}
  \tikzset{outervert/.style={mydarkred,solid,vert}}}
\gdef\innertext{blue\xspace}
\gdef\outertext{red\xspace}
\end{bwcompromise}
\begin{bwonly}
\gdef\setinnerouter{\tikzset{inner/.style={densely dotted}}
  \tikzset{outer/.style={densely dashed}}
  \tikzset{outervert/.style={solid,vert}}}
\gdef\innertext{\relax}
\gdef\outertext{\relax}
\end{bwonly}
\begin{coloronly}
\gdef\setinnerouter{\tikzset{inner/.style={blue}}
  \tikzset{outer/.style={red}}
  \tikzset{outervert/.style={red,vert}}}
\gdef\innertext{blue\xspace}
\gdef\outertext{red\xspace}
\end{coloronly}
\setinnerouter
\tikzset{innervert/.style={inner,circle,fill,inner sep=1pt}}
\tikzset{topvert/.style={innervert}}

\tikzset{innerunit/.style={inner,regular polygon,regular polygon sides=3,%
    shape border rotate=180,fill,inner sep=1pt}}
\tikzset{innercounit/.style={inner,regular polygon,regular polygon sides=3,%
    fill,inner sep=1pt}}
\tikzset{inneriso/.style={inner,circle,fill,inner sep=1pt}}
\tikzset{innerbc/.style={inner,diamond,fill,inner sep=1.3pt}}

\usetikzlibrary{decorations.pathmorphing}
\usetikzlibrary{decorations}
\tikzset{transf/.style={decorate,decoration={zigzag,amplitude=1pt,segment length=3pt}}}

\pgfdeclarelayer{background}
\pgfdeclarelayer{foreground}
\pgfsetlayers{background,main,foreground}

\def\bgcylinder#1#2#3#4#5#6{
  \def\cylempty{}\def\cylfrontcolor{#5}\def\cylbackcolor{#6}
  \begin{pgfonlayer}{background}
    \ifx\cylfrontcolor\cylempty\draw\else\fill[fill=my#5]\fi
    (#1) coordinate (dl)
    -- ++(0,#2) node[coordinate] (ul) {} 
    arc (-180:0:#3 and #4) coordinate (ur)
    -- ++(0,-#2) node[coordinate] (dr) {}
    arc (0:-180:#3 and #4);
    \ifx\cylbackcolor\cylempty\draw\else\fill[fill=my#6!80!black]\fi
    ($(ul)!.5!(ur)$) node[coordinate] (top) {} ellipse (#3 and #4);
    \path (dl) arc (-180:-90:#3 and #4) node[coordinate] (bot) {};
    \clip (dl) -- (ul) arc (-180:0:#3 and #4) -- (dr) arc (0:-180:#3 and #4);
  \end{pgfonlayer}
  \begin{pgfonlayer}{foreground}
    \clip (dl) -- (ul) arc (-180:0:#3 and #4) -- (dr) arc (0:-180:#3 and #4);
  \end{pgfonlayer}
  \clip (dl) -- (ul) arc (-180:0:#3 and #4) -- (dr) arc (0:-180:#3 and #4);
  \path (ul) ++(-0.1,0.1) coordinate (ul');
  \path (ur) ++(0.1,0.1) coordinate (ur');
  \path (dl) ++(-0.1,-0.1) coordinate (dl');
  \path (dr) ++(0.1,-0.1) coordinate (dr');
  \path (top) ++(0,0.1) coordinate (top');
  \path (bot) ++(-0,-0.1) coordinate (bot');
}

\usetikzlibrary{shapes.geometric}
\tikzset{fiber/.style={draw,rectangle,inner sep=2pt,font=\scriptsize}}
\tikzset{pb/.style={draw,regular polygon,regular polygon sides=3,inner sep=0pt,shape border rotate=180,font=\scriptsize}}
\tikzset{pbe/.style={draw,regular polygon,regular polygon sides=3,inner sep=2pt,shape border rotate=180}} 
\tikzset{pbsm/.style={draw,regular polygon,regular polygon sides=3,inner sep=-.5pt,shape border rotate=180,font=\scriptsize}} 
\tikzset{pbflat/.style={draw,regular polygon,regular polygon sides=3,shape border rotate=270,inner sep=1pt,font=\scriptsize}}
\tikzset{pbeflat/.style={draw,regular polygon,regular polygon sides=3,shape border rotate=270,inner sep=2pt}}
\tikzset{pbsmflat/.style={draw,regular polygon,regular polygon sides=3,shape border rotate=270,inner sep=0pt,font=\scriptsize}}
\tikzset{pf/.style={draw,regular polygon,regular polygon sides=3,inner sep=0pt,font=\scriptsize}}
\tikzset{pfe/.style={draw,regular polygon,regular polygon sides=3,inner sep=2pt}} 
\tikzset{pfsm/.style={draw,regular polygon,regular polygon sides=3,inner sep=-.5pt,font=\scriptsize}} 
\tikzset{pfflat/.style={draw,regular polygon,regular polygon sides=3,shape border rotate=90,inner sep=1pt,font=\scriptsize}}
\tikzset{pfeflat/.style={draw,regular polygon,regular polygon sides=3,shape border rotate=90,inner sep=2pt}}
\tikzset{pfsmflat/.style={draw,regular polygon,regular polygon sides=3,shape border rotate=90,inner sep=0pt,font=\scriptsize}}

\tikzset{backwards/.style={z={(-.3,-.7)},x={(-1,0)},y={(0,-1.5)}}}
\tikzset{strings/.style={scale=.75, z={(0,-.4)},x={(-1,0)},y={(0,-2)}}}

\title{Duality and traces for indexed monoidal categories}

\author{Kate Ponto and Michael Shulman}
\eaddress{kate.ponto@uky.edu\CR mshulman@ucsd.edu}

\copyrightyear{2012}


\thanks{Both authors were supported by National Science Foundation
  postdoctoral fellowships during the writing of this paper.

  The final published version of this article, which is essentially
  identical to this one, is available from the journal's web site at
  \url{http://www.tac.mta.ca/tac}.

\ifcoloronly
    This version of this paper contains diagrams that use color
    to distinguish different features, and is intended for reading on
    a computer screen or printing on a color printer.  An alternative
    version which uses dots and dashes for this purpose, and intended
    for printing on a black and white printer, is available from the
    journal web site and from the authors' web sites.
\fi%
\ifbwonly
    This version of this paper contains diagrams that use dots
    and dashes to distinguish different features, and is intended for
    printing on a black and white printer.  An alternative version
    which uses color for this purpose, and intended for reading on a
    computer screen or printing on a color printer, is available from
    the journal web site and from the authors' web sites.
\fi%
\ifbwcompromise
    This version of this paper contains diagrams that use both
    color and dashes to distinguish different features.  Alternative
    versions which use only color (optimal for reading on a computer
    screen or printing on a color printer) or only dashes (optimal for
    printing on a black and white printer) are available from the
    authors' web sites.
\fi%
}

\date{Version of \today}

\makeatother

\begin{document}
\maketitle

\begin{abstract}  
  By the Lefschetz fixed point theorem, if an endomorphism of a
  topological space is fixed-point-free, then its Lefschetz number
  vanishes.  This necessary condition is not usually sufficient,
  however; for that we need a refinement of the Lefschetz number called
  the Reidemeister trace.  Abstractly, the Lefschetz number is a trace
  in a symmetric monoidal category, while the Reidemeister trace is a
  trace in a bicategory; in this paper we relate these contexts using
  indexed symmetric monoidal categories.

  In particular, we will show that for any symmetric monoidal category
  with an associated indexed symmetric monoidal category, there is an
  associated bicategory which produces refinements of trace analogous to
  the Reidemeister trace.  This bicategory also produces a new notion of
  trace for parametrized spaces with dualizable fibers, which refines
  the obvious ``fiberwise'' traces by incorporating the action of the
  fundamental group of the base space.  We also advance the basic theory
  of indexed monoidal categories, including introducing a string diagram
  calculus which makes calculations much more tractable.  This abstract
  framework lays the foundation for generalizations of these ideas to
  other contexts.
\end{abstract}

\tableofcontents

\section{Introduction}
\label{sec:introduction}

It is well-known that in any symmetric monoidal category, there are
useful intrinsic notions of \emph{duality} and \emph{trace}; see, for
example,~\cite{dp:duality,kl:cpt,jsv:traced-moncat}.  One of the
original motivating examples is that traces in the stable homotopy
category compute fixed-point indices, an observation which leads
directly to the Lefschetz fixed point theorem.  The search for
generalizations and converses of the Lefschetz fixed point theorem has
led to various generalizations and refinements of the fixed-point index,
such as the Reidemeister trace~\cite{b:lefschetz,j:nielsen}.  In this
paper we present an abstract framework for constructing refinements of
traces in symmetric monoidal categories, which produces the Reidemeister
trace as a particular example.  Other examples, which we will not
discuss here, include equivariant~\cite{equiv},
relative~\cite{kate:rel}, and fiberwise~\cite{kate:traces}
generalizations of the Lefschetz number and Reidemeister trace.

These refinements arise from generalizations of duality and trace to
\emph{indexed} symmetric monoidal categories.  An indexed symmetric
monoidal category is a family of symmetric monoidal categories $\sC^A$,
one for each object $A$ of a cartesian monoidal base category \bS,
equipped with base change functors induced by the morphisms of \bS.  The
simplest case is when \bS\ is the category of sets and $\sC^A$ the
category of $A$-indexed families of objects of some fixed symmetric
monoidal category, such as abelian groups or chain complexes.  We can
also allow \bS\ to consist of groupoids or higher groupoids.  The
prototypical example in homotopy theory, which will give rise to the
Reidemeister trace, is when \bS\ is the category of topological spaces
and $\sC^A$ is the homotopy category of parametrized spectra over $A$,
as in~\cite{maysig:pht}.  More generally, \bS\ could be any category of
spaces, or of schemes, or a topos, and $\sC^A$ a category of sheaves of
abelian groups, modules, spaces, or parametrized spectra on $A$, or a
derived category thereof.

In any such context we can, of course, consider duality and trace in the
individual symmetric monoidal categories $\sC^A$.
Following~\cite{maysig:pht} we call these notions \emph{fiberwise},
since in examples they are usually equivalent to duality and trace
acting on each fiber or stalk separately.  Thus, if $M$ is fiberwise
dualizable, any endomorphism $f\colon M\rightarrow M$ has a fiberwise
trace, which is an endomorphism of the unit object $I_A$ of $\sC^A$.  In
examples, this trace essentially calculates the ordinary traces of the
induced endomorphisms $f_x \colon M_x \to M_x$ of the fibers $M_x$, for
all $x\in A$.  Our first theorem gives a refinement of this fiberwise
symmetric monoidal trace.

\begin{thm}\label{thm:intro-fiberwise}
  If $M\in\sC^A$ is fiberwise dualizable and $f\colon M\rightarrow M$ is
  any endomorphism, then the symmetric monoidal trace of $f$ factors as
  a composite
  \[ I_A \too (\pi_A)^* \sh{A} \xto{\tr(\widehat{f})} I_A. \]
\end{thm}

(In fact, we prove a stronger theorem about ``partial traces'' whose
domains and codomains are ``twisted'' by an additional object.)

The first morphism in this factorization is defined purely in terms of
$A$, while the second morphism contains the information about $M$ and
$f$.  Thus, this theorem refines $\tr(f)$ by lifting its domain to a
potentially larger one.  In examples, the refined trace
$\tr(\widehat{f})$ also includes the traces of the composites of the
fiberwise endomorphisms $f_x$ with the action of ``loops'' in $A$ on the
fibers of $M$ (in cases where $A$ is something which contains loops,
like a groupoid or a topological space).

Often, however, we are in a somewhat different situation: we want to
extract trace-like information (such as fixed point invariants) from an
endomorphism $f\colon A \to A$ in some \emph{cartesian} monoidal
category \bS, such as sets, groupoids, or spaces.  Since there are no
nontrivial dualities in a cartesian monoidal category, the standard
approach is to choose a non-cartesian monoidal category \bC\ (such as
abelian groups, chain complexes, or spectra), apply a functor
$\Sigma\colon \bS \to \bC$ (such as the free abelian group or suspension
spectrum), and then consider the symmetric monoidal trace of $\Sigma(f)$
in \bC.  The trace obtained from the free abelian group functor in this
way simply counts the number of fixed points of a set-endofunction,
while that obtained from the suspension spectrum functor is exactly the
classical fixed-point index.

Now, it turns out that in most examples where this is done, there is
actually an indexed symmetric monoidal category over \bS, such that $\bC
= \sC^\star$ is the category indexed by the terminal object of \bS.
Moreover, the functor $\Sigma\colon \bS\to \sC^\star$ is definable in
terms of this structure: $\Sigma(A)$ is the ``pushforward'' to
$\sC^\star$ of the unit object of $\sC^A$.  For example, the free
abelian group functor arises from ``set-indexed families of abelian
groups'', and the suspension spectrum functor arises from spectra
parametrized over spaces.

Our second theorem gives a refinement of the trace of $\Sigma(f)$ in
this situation.  We need a slightly stronger hypothesis, however: rather
than requiring $\Sigma(A)$ to be dualizable in $\sC^\star$ (the
requirement in order for $\Sigma(f)$ to have a trace), we need to assume
that $I_A$ is \emph{totally dualizable} in $\sC^A$.  Total duality is a
new type of duality, not reducible to symmetric monoidal duality, which
can be defined in any good indexed symmetric monoidal category; it was
first noticed in the case of parametrized spectra by Costenoble and
Waner~\cite{cw:eohc}, and studied further in~\cite{maysig:pht}.  Total
dualizability of $I_A$ implies ordinary dualizability of $\Sigma(A)$ in
$\sC^\star$, and in examples it seems to be not much stronger than this
(if at all).

\begin{thm}\label{thm:intro-total}
  If $I_A$ is totally dualizable and $f\colon A\rightarrow A$ is an
  endomorphism in \bS, then the symmetric monoidal trace of $\Sigma(f)$
  factors as a composite
  \[ I_\star \xto{\tr(\widecheck{f})} \sh{A_f} \too I_\star. \]
\end{thm}

Here, it is the \emph{first} morphism in the factorization which should
be regarded as a refined trace of $f$, while the second morphism forgets
the additional information contained therein (although in this case, the
intermediate object $\sh{A_f}$ also depends on $f$).  In examples, the
object $\sh{A_f}$ is ``generated by'' the ``fixed-point classes'' of
$f$, and the refined trace $\tr(\widecheck{f})$ separates out the
contributions to the fixed-point index of $f$ depending on which
fixed-point class they belong to.  In particular, in the case of spectra
parametrized over spaces, we obtain the Reidemeister trace, and
\autoref{thm:intro-total} says that the Reidemeister trace refines the
fixed-point index.

As with \autoref{thm:intro-fiberwise}, \autoref{thm:intro-total} is a
special case of a stronger theorem about ``partial traces'' and
arbitrary totally dualizable objects.  We can also exploit this
generalization to describe a further connection to the classical
symmetric monoidal theory.  Perhaps the most important ``partial'' or
``twisted'' trace is the trace of the composite
\[ \Sigma(A) \xto{\Sigma(f)} \Sigma(A) \xto{\Sigma(\Delta_A)} \Sigma(A)
\otimes \Sigma(A) \] which is a map $I_\star \to \Sigma(A)$ called the
\textbf{transfer of $f$} (or the ``trace of $f$ with respect to
$\Delta$'').  The transfer of $f$ is also a refinement of
$\tr(\Sigma(f))$, in that when composed with the augmentation
\[\Sigma(\pi_A)\colon \Sigma(A) \to \Sigma(\star) = I_\star\]
it reproduces the symmetric monoidal trace of $f$.  The generalized
version of \autoref{thm:intro-total} then implies that
$\tr(\widecheck{f})$ refines not only $\tr(\Sigma(f))$, but also the
trace of $f$ with respect to $\Delta$.

\begin{thm}\label{thm:intro-total2}
  In the situation of \autoref{thm:intro-total}, the transfer of $f$
  factors as a composite
  \[ I_\star \xto{\tr(\widecheck{f})} \sh{A_f} \too \Sigma(A). \]
\end{thm}

Although they may seem different on the surface, Theorems
\ref{thm:intro-fiberwise} and
\ref{thm:intro-total}--\ref{thm:intro-total2} actually apply in
\emph{formally dual} situations (although their conclusions and proofs
are not dual).  This observation is originally due to May and
Sigurdsson~\cite{maysig:pht}, who constructed a \emph{bicategory} out of
parametrized spectra, and identified fiberwise and Costenoble-Waner
duality as cases of the general notion of duality in this bicategory
(which comes in dual flavors, since composition in a bicategory is not
symmetric).  In~\cite{shulman:frbi}, the second author generalized their
construction to any indexed symmetric monoidal category.  Part of
proving the above theorems, therefore, will be to extend the results
of~\cite{maysig:pht} on duality to the case of a general indexed
symmetric monoidal category.

The consideration of \emph{traces} for bicategorical duality, on the
other hand, requires an additional structure called a \emph{shadow},
introduced by the first author in~\cite{kate:traces} and studied further
in~\cite{PS2}.  Thus, another necessary preliminary will be to prove
that the bicategory arising from any indexed symmetric monoidal category
has a canonical shadow.  This done, we will be able to identify the
morphisms $\tr(\widehat{f})$ and $\tr(\widecheck{f})$ as traces relative
to the bicategorical incarnations of fiberwise and total duality,
respectively, and deduce Theorems \ref{thm:intro-fiberwise},
\ref{thm:intro-total}, and \ref{thm:intro-total2}.

Finally, there is one further ingredient in this paper: the use of a
\emph{string diagram} calculus for indexed monoidal categories.  In
general, string diagram calculus is a ``Poincar\'e dual'' way of drawing
composition in categorical structures, which tends to make complicated
computations much more visually evident.  In particular, many basic
equalities, such as the naturality of tensor products, are realized by
simple isotopies.  String diagrams for monoidal categories and
bicategories are described
in~\cite{penrose:negdimten,js:geom-tenscalc-i,js:geom-tenscalc-ii,js:pdta,jsv:traced-moncat,selinger:graphical,street:ldtop-hocat},
and a generalization for bicategories with shadows was given
in~\cite{PS2}.

Several of the results of this paper require fairly involved
computations which are made much more tractable by an appropriate string
diagram calculus.  However, for clarity of exposition, we postpone the
introduction of string diagrams (along with the proofs which depend on
them, of course) to the very end of the paper.  Thus, the reader can
avoid string diagrams altogether by stopping after \S\ref{sec:total}.

The organization of this paper is as follows.  In
\S\S\ref{sec:shad-from-mono}--\ref{sec:indexed-coproducts} we introduce
indexed symmetric monoidal categories.  Then in \S\ref{sec:smc-traces}
we recall the classical notion of trace in symmetric monoidal
categories, and construct the functor $\Sigma$ referred to above.  In
\S\ref{sec:shadows} we describe the construction of a bicategory with a
shadow from an indexed symmetric monoidal category.

In \S\ref{sec:fiberwise} we state the ``fiberwise'' duality and trace
theorems---\autoref{thm:intro-fiberwise} and its
generalizations---although the proofs are postponed until
\S\ref{sec:fiberwise-proofs}.  Next, \S\ref{sec:base-change} is devoted
to the consideration of ``base change objects'', an additional bit of
structure that exists in bicategories arising from indexed monoidal
categories.  This is then used in \S\ref{sec:total}, where we state and
prove the ``total'' duality and trace theorems, \ref{thm:intro-total}
and \ref{thm:intro-total2} and their generalizations.  The proofs in
this case are actually easier than the fiberwise case, although some
calculations still have to be postponed to \S\ref{sec:tdproofs}.

Finally, in \S\S\ref{sec:string-diagrams}--\ref{sec:colored-strings} we
introduce our string diagram calculus, and in
\S\S\ref{sec:fiberwise-proofs}--\ref{sec:tdproofs} we apply it to the
postponed calculations.

Throughout the paper, we make use of three primary examples:
\begin{enumerate}
\item $\bS=$ sets, $\sC^A=$ $A$-indexed families of abelian groups.
\item $\bS=$ topological spaces, $\sC^A=$ spectra parametrized over $A$.
\item $\bS=$ groupoids, $\sC^A=$ $A$-indexed diagrams of chain complexes.
\end{enumerate}
The first is a ``toy'' example, which is easy to compute with, but which
is so degenerate that it fails to display all of the interesting
phenomena.  The second is the motivating example which matters most in
applications (together with various generalizations), but it is
technically fairly complicated, so we omit all proofs relating to it;
many of them can be found in~\cite{kate:higher}.  (In particular, no
knowledge of spectra is necessary to read this paper.)  The third is an
intermediate example which is easier to understand, but which still
displays most of the interesting phenomena.

We would like to thank David Corfield, Todd Trimble, and Daniel
Sch\"appi for helpful suggestions about string diagrams and monoidal
bicategories.  The second-named author gratefully acknowledges the
hospitality of the University of Kentucky.

\section{Indexed monoidal categories}
\label{sec:shad-from-mono}

We begin with the following standard definition (cf.\ for
instance~\cite{ps:indexed-aft} or~\cite[B1.2]{ptj:elephant1}).

\begin{defn}
  Let \bS\ be a category.  An \textbf{\bS-indexed category} \sC\ is a
  pseudofunctor from $\bS\op$ to \Cat.
\end{defn}

In elementary terms, this means we have a category for each object
$A\in\bS$, which we write as $\sC^A$, and a functor for each morphism
$f\maps A\to B$ in \bS, which we write as $f^*\maps \sC^B\to\sC^A$, such
that composition and identities are preserved up to coherent natural
isomorphism.  We sometimes refer to \bS\ as the \emph{base category} and
the $\sC^A$ as the \emph{fiber categories}, and we refer to the functors
$f^*$ as \emph{reindexing functors}.

\begin{eg}\label{eg:ring-ind-mod}
  If $\bS=\mathbf{Ring}$ is the category of rings, there is an
  \bS-indexed category sending a ring $A$ to $\bMod{A}$, with $f^*$
  given by restriction of scalars.  There are similar examples using
  chain complexes, DGAs, or ring spectra.
\end{eg}

\begin{eg}\label{eg:families}
  For any category \bC, there is a \Set-indexed category sending each
  set $A$ to the category $\bC^A$ of $A$-indexed families of objects
  of \bC.
\end{eg}

\begin{eg}\label{eg:spaces-over-spaces}
  If \bS\ has finite limits, then there is an \bS-indexed category
  sending each object $A$ to its slice category $\bS/A$, with the
  functors $f^*$ given by pullback.

  In the important special case when \bS\ is the category \Top\ of
  topological spaces, we also have a ``homotopy'' version $A\mapsto
  \Ho(\Top/A)$, in which we invert the weak homotopy equivalences in
  each fiber category.  See~\cite{maysig:pht} for an extensive formal
  development of this example.  Similar examples can be constructed
  whenever \bS\ has a ``homotopy theory'' (e.g.\ when it is a Quillen
  model category).
\end{eg}

``Derived'' examples such as the last one are of particular interest to us.
Here are two more such examples.

\begin{eg}\label{eg:spectra-over-spaces}
  If \Top\ denotes a suitably nice category of topological spaces, there
  is a \Top-indexed category defined by $A\mapsto \bEx{A}$, where
  $\bEx{A}$ denotes a point-set--level category of spectra parametrized
  over $A$, such as the parametrized orthogonal spectra used
  in~\cite{maysig:pht}.  We can similarly invert the weak equivalences
  of spectra to obtain a derived \Top-indexed category $A\mapsto
  \Ho(\bEx{A})$.
\end{eg}

\begin{eg}\label{eg:derived-families}
  If the category \bC\ in \autoref{eg:families} has a homotopy theory,
  then we can pass to homotopy categories in that example to obtain
  another \Set-indexed category $\Ho(\bC^A)$.  However, since
  $\Ho(\bC^A) \simeq (\Ho(\bC))^A$ for any set $A$, we gain no extra
  generality thereby.
\end{eg}

We will see in later sections that qualitatively different phenomena
arise in ``derived'' situations, when the base category \bS\ also has a
``homotopy theory'' which is taken into account somehow in the fibers.
This is the case in Examples~\ref{eg:spaces-over-spaces}
and~\ref{eg:spectra-over-spaces}, but not in
\autoref{eg:derived-families}.  Since
Examples~\ref{eg:spaces-over-spaces} and~\ref{eg:spectra-over-spaces}
are somewhat technically complicated, it will be useful to have a
simpler example which exhibits some of the same derived phenomena.  We
can obtain such examples by replacing spaces by groupoids (which, up to
homotopy, can be identified with homotopy 1-types).

\begin{eg}\label{eg:gpds}
  If $\bS=\mathbf{Gpd}$ is the category of groupoids and \bC\ is any
  category, we can define $\sC^A = \bC^A$ to be the category of functors
  from the groupoid $A$ to \bC.  For simplicity, we usually take \bC\ to
  be abelian groups.  This is an enlargement of \autoref{eg:families},
  if we consider sets as discrete groupoids.  On the other hand, if $A$
  is a one-object groupoid, i.e. a group $G$, then $\Ab^A$ can be
  identified with the category of modules over the group ring $\bbZ[G]$.
\end{eg}

\begin{eg}\label{eg:hogpds}
  We have ``derived phenomena'' in \autoref{eg:gpds} even when \bC\ does
  not have its own homotopy theory, because the construction
  automatically ``sees'' the homotopy theory of groupoids.  (For
  instance, we have $\bC^A\simeq \bC^B$ whenever $A\simeq B$ are
  \emph{equivalent} groupoids, not only when they are isomorphic).
  However, we can also ``derive it further'' if \bC\ \emph{does} have a
  homotopy theory, obtaining a $\mathbf{Gpd}$-indexed category with
  $\sC^A = \Ho(\bC^A)$, where we invert the objectwise weak equivalences
  in $\bC^A$.  To maximize concreteness and familiarity, we may take
  $\bC=\bCh{\mathbb{Z}}$ to be the category of (unbounded) chain
  complexes, with its usual homotopy theory.

  Note that in contrast to \autoref{eg:derived-families}, when $A$ is
  not discrete, $\Ho(\bC^A)$ will generally be different from
  $(\Ho(\bC))^A$.  Thus, this example is genuinely different from
  \autoref{eg:gpds}.
\end{eg}

Now in most of these examples, each category $\sC^A$ has a monoidal
structure, which is respected by the transition functors.  Namely:
\begin{itemize}
\item If \bC\ is a monoidal category (such as \Ab\ or $\bCh{\mathbb{Z}}$
  with tensor product), then so is $\bC^A$ for any set or groupoid $A$;
\item When \bS\ has finite limits, each slice category $\bS/A$ is
  cartesian monoidal;
\item Each category $\bEx{A}$ has a parametrized smash product; and
\item All of these monoidal structures descend to homotopy categories.
\end{itemize}
This ``fiberwise'' monoidal structure is simply described by the
following definition.

\begin{defn}
  For a category \bS, an \textbf{\bS-indexed symmetric monoidal
    category} is a pseudofunctor from $\bS\op$ to the 2-category
  \SymMonCat\ of symmetric monoidal categories, strong symmetric
  monoidal functors, and monoidal natural transformations.
\end{defn}

In elementary terms, this means an indexed category such that each
$\sC^A$ is a symmetric monoidal category, each reindexing functor is
strong symmetric monoidal, and each constraint isomorphism is a monoidal
transformation.  We write $\ten_A$ and $I_A$ for the tensor product and
unit object of the monoidal category $\sC^A$, and refer to them as the
\emph{fiberwise} monoidal structure.

However, for many purposes it is more effective to express the monoidal
structure differently.  If \bS\ has finite products, including a
terminal object $\star$, then in any \bS-indexed monoidal category \sC\
we can define an \emph{external product} functor
\begin{align*}
  \sC^A \times \sC^B &\to \sC^{A\times B}\\
  (M,N) &\mapsto M\boxtimes N \coloneqq (\pi_B^* M) \ten_{A\times B} (\pi_A^* N).
\end{align*}
(Here $\pi_A\colon A\times B\to B$ and $\pi_B\colon A\times B\to A$ are
the projections from the cartesian product.)  This external product is
coherently associative and unital, in a suitable sense
(see~\cite{shulman:frbi}), with unit object $U = I_\star$.  Moreover, we
can recover the fiberwise monoidal structures (or \emph{internal
  products}) from the external products, via the isomorphisms
\begin{align}
  M \ten_A N &\cong \Delta_A^* (M \boxtimes N)\label{eq:intprod}\\
  I_A &\cong \pi_A^* U.\label{eq:intunit}
\end{align}
(Here $\Delta_A\colon A\to A\times A$ is the diagonal of the cartesian
product.)  These hold for any $A\in\bS$ and $M,N\in\sC^A$, by
monoidality of the functors $f^*$.  In the same way we can construct
natural isomorphisms
\begin{equation}\label{eq:extmon}
  f^*M \boxtimes g^*N \cong (f\times g)^*(M\boxtimes N)
\end{equation}
for any $M,N,f,g$ for which this makes sense.

\begin{eg}
  If \bC\ is a symmetric monoidal category, then in the \Set-indexed
  symmetric monoidal category $A\mapsto \bC^A$ of \autoref{eg:families},
  the fiberwise product of $M = (M_a)$ and $N= (N_a)$ in $\bC^A$ is
  defined by
  \[ (M\otimes_A N)_a = M_a \otimes N_a \] whereas the external product
  of $M = (M_a)\in\bC^A$ and $N= (N_b)\in\bC^B$ is defined by
  \[ (M\boxtimes N)_{(a,b)} = M_a \otimes N_b. \] The case of
  \autoref{eg:gpds} is similar.  Note that if $A$ is one-object
  groupoid, hence a group $G$, and we identify $\Ab^A$ with the category
  of $\bbZ[G]$-modules, then its fiberwise monoidal structure is the
  tensor product over \bbZ, made into a $\bbZ[G]$-module using the fact
  that $\bbZ[G]$ is a bialgebra.  It is \emph{not} a tensor product over
  $\bbZ[G]$, which wouldn't make sense anyway since $\bbZ[G]$ is not
  commutative.
\end{eg}

\begin{eg}
  If \bS\ has finite limits, then in the \bS-indexed symmetric monoidal
  category $A\mapsto \bS/A$ of \autoref{eg:spaces-over-spaces}, the
  fiberwise product of $M,N \in \bS/A$ is their fiber product
  (pullback):
  \[ M \otimes_A N = M\times_A N \] whereas the external product of
  $M\in \bS/A$ and $N\in \bS/B$ is just their cartesian product, with
  the induced projection to $A\times B$:
  \[ M \boxtimes N = M\times N. \] The case of
  \autoref{eg:spectra-over-spaces} is similar, with cartesian products
  replaced by smash products.  Both the fiberwise and external products
  of parametrized spectra are used in~\cite{maysig:pht}.
\end{eg}

\begin{rmk}
  When the base category \bS\ is monoidal but not cartesian monoidal, we
  can still define ``an external product'' for \sC\ to consist of
  functors $\sC^A \times \sC^B \to \sC^{A\otimes B}$ which are
  coherently associative and unital.  For instance,
  \autoref{eg:ring-ind-mod} has an external product, but not an internal
  one.  However, our interest here is solely in the case when \bS\ is
  cartesian monoidal.
\end{rmk}

\section{Indexed coproducts}
\label{sec:indexed-coproducts}

In order to construct a bicategory from an indexed symmetric monoidal
category, and thereby prove our refinements of the symmetric monoidal
trace, we will need one further piece of structure.

\begin{defn}\label{def:indcoprod}
  Let \bS\ be a cartesian monoidal category, and \sC\ an \bS-indexed
  category.  We say that \sC\ has \textbf{\bS-indexed coproducts} if
  \begin{enumerate}
  \item Each reindexing functor $f^*$ has a left adjoint $f_!$, and
  \item For any pullback square
    \[\xymatrix@-.5pc{A \ar[r]^f\ar[d]_h & B \ar[d]^g\\
      C\ar[r]_k & D}\]
    in \bS, the composite
    \[f_! h^* \too f_!h^*k^*k_! \too[\iso] f_!f^*g^*k_! \too g^*k_!\]
    is an isomorphism (the \emph{Beck-Chevalley condition}).
  \end{enumerate}
  If \sC\ is symmetric monoidal, we say that \textbf{$\ten$ preserves
    indexed coproducts} (in each variable separately), or that the
  \textbf{projection formula holds}, if
  \begin{enumerate}[resume]
  \item for any $f\maps A\to B$ in \bS\ and any $M\in\sC^B$,
    $N\in\sC^A$, the canonical map
    \[f_!(f^*M \otimes N) \to f_!(f^*M \otimes f^* f_! N)
    \cong f_! f^* (M \otimes f_! N) \to M \otimes f_!N
    \]
    is an isomorphism.\footnote{If $\sC$ is not symmetric, we must
      assert also the analogous condition on the other
      side.}\label{item:prescoprod}
  \end{enumerate}
\end{defn}

We will only use the Beck-Chevalley condition for a few types of
pullback squares.  The basic such squares are shown in
Figure~\ref{fig:pullbacks}.  (These are all actually pullback squares in
any cartesian monoidal category, whether or not it has pullbacks in
general.  See
also~\cite{lawvere:comprehension,seely:hd-beck,trimble:predlogic}.)  We
will also consider the transposes of Figures~\ref{fig:frobpb}
and~\ref{fig:slidpb}, as well as squares obtained from one of these by
taking a cartesian product with a fixed object.

For ease of reference, we have assigned a name to each of these
Beck-Chevalley conditions.  We call Figure~\ref{fig:reindpb}
``commutativity with reindexing'' because the condition $(\id\times
g)_!(f\times\id)^* \cong (f\times\id)^*(\id\times g)_!$ says that the
indexed-coproduct functor $g_!$ commutes with the reindexing functor
$f^*$.  We call Figure~\ref{fig:frobpb} the ``Frobenius axiom'' because
the condition $\Delta_! \Delta^* \cong
(\Delta\times\id)^*(\id\times\Delta)_!$ has the same form as one of the
axioms of a Frobenius algebra.  The name of Figure~\ref{fig:slidpb} will
make more sense once we introduce string diagrams; see
\S\ref{sec:string-diagrams}.  Finally, we call Figure~\ref{fig:mdpb}
``monic diagonals'' because the fact that that square is a pullback says
precisely that $\Delta_A$ is a monomorphism.

\begin{figure}
  \centering
  \begin{tabular}{c@{\hspace{1cm}}c}
    \subfigure[Commutativity with reindexing]{\label{fig:reindpb}
      $\xymatrix@C=4pc{A\times C \ar[r]^{f\times \id_C}\ar[d]_{\id_A\times g}
        & B\times C \ar[d]^{\id_B\times g}\\
        A\times D\ar[r]_{f\times \id_D} & B\times D}$}
    &
    \subfigure[The Frobenius axiom]{\label{fig:frobpb}
      $\xymatrix@C=3pc{A \ar[r]^{\Delta}\ar[d]_{\Delta}
        & A\times A \ar[d]^{\id_A\times \Delta}\\
        A\times A\ar[r]_-{\Delta\times \id_A} & A\times A\times A}$}
    \\
    \subfigure[Sliding and splitting]{\label{fig:slidpb}
      $\xymatrix@C=3pc{A \ar[r]^-{(\id_A,f)}\ar[d]_f
        & A\times B \ar[d]^{f\times \id_B}\\
        B\ar[r]_-{\Delta_B} & B\times B}$}
    &
    \subfigure[Monic diagonals]{\label{fig:mdpb}
      $\xymatrix@C=3pc{A \ar[r]^{\id_A}\ar[d]_{\id_A} & A \ar[d]^{\Delta}\\
        A\ar[r]_-{\Delta} & A\times A}$}
  \end{tabular}
  \caption{Pullback squares involving products}
  \label{fig:pullbacks}
\end{figure}

We have stated the definition of ``$\ten$ preserves indexed coproducts''
in a form which may be most familiar, but we care most about the
following equivalent statement.

\begin{lem}
  If \sC\ is an indexed symmetric monoidal category with \bS-indexed
  coproducts, then these are preserved by $\ten$ if and only if for any
  $f\maps A\to B$ and $g\maps C\to D$ in \bS\ and any $M\in\sC^A$,
  $N\in\sC^C$, the composite
  \begin{align*}
    (f\times g)_!(M\boxtimes N) &\too (f\times g)_!(f^* f_! M\boxtimes g^* g_! N)\\
    &\too[\cong] (f\times g)_!(f\times g)^*(f_! M\boxtimes g_! N)\\
    &\too f_! M\boxtimes g_! N
  \end{align*}
  is an isomorphism.
\end{lem}
\begin{proof}[(Sketch)]
  We will show that each of the two families of isomorphisms
  \[ f_!(f^*M \otimes N) \cong (M\otimes f_!N)
  \qquad\text{and}\qquad
  (f\times g)_!(M\boxtimes N) \cong (f_!M \boxtimes g_!N)
  \]
  can be constructed from the other.  We omit the proof that the
  constructed isomorphism in each case is in fact the canonical morphism
  exhibited above and in \autoref{def:indcoprod}; in each case this can
  be shown by a diagram chase, or by invoking the technology of
  ``mates'' from~\cite{ks:r2cats}.

  First, we observe that to prove the second isomorphism above, it
  suffices to consider the case when $f$ or $g$ is an identity, since
  then we can conclude
  \[ (f\times g)_!(M\boxtimes N) \cong
  (f\times \id)_!(\id\times g)_!(M\boxtimes N) \cong
  (f\times \id)_! (M \boxtimes g_!N) \cong (f_!M \boxtimes g_!N).
  \]
  Now assuming that $\otimes$ preserves indexed coproducts, for any
  $f\maps A\to B$, $M\in\sC^A$, and $N\in\sC^C$, we have
  \begin{align*}
    (f\times \id_C)_!(M\boxtimes N)
    &= (f\times \id_C)_! (\pi_C^* M \otimes (f\times \id_C)^* \pi_B^* N)\\
    &\cong (f\times \id_C)_! \pi_C^*M \otimes \pi_A^* N\\
    &\cong \pi_C^* f_! M \otimes \pi_A^* N\\
    &= f_!M \boxtimes N
  \end{align*}
  using the assumption and the ``commutativity with reindexing''
  Beck-Chevalley condition.  Conversely, assuming the desired condition,
  then for any $f\maps A\to B$, $M\in\sC^A$, and $N\in\sC^B$, we have
  \begin{align*}
    f_! (M\otimes f^*N)
    &= f_! \Delta^* (M \boxtimes f^*N)\\
    &\cong f_! (\id_A,f)^* (M\boxtimes N)\\
    &\cong \Delta^* (f\times \id_B)_!(M\boxtimes N)\\
    &\cong \Delta^* (f_!M \boxtimes N)\\
    &= f_!M \otimes N
  \end{align*}
  using~\eqref{eq:extmon}, the assumption, and the ``sliding and
  splitting'' Beck-Chevalley condition.
\end{proof}

\begin{rmk}\label{rmk:frobenius}
  In the case when the fiber categories are cartesian monoidal, the
  projection formula is also called the ``Frobenius condition''.  As
  remarked above, however, we reserve the adjective ``Frobenius'' for
  the Beck-Chevalley condition associated to Figure~\ref{fig:frobpb}.
  The two are closely related, however:
  \begin{enumerate}
  \item Trimble has shown that the ``Frobenius'' Beck-Chevalley
    condition follows from either
    \begin{enumerate}
    \item The ``commutativity with reindexing'' Beck-Chevalley condition
      and the projection formula (see~\cite{nlab:frobenius}), or
    \item The Beck-Chevalley conditions ``commutativity with reindexing'' and ``sliding and
      splitting'' 
      (see~\cite{trimble:predlogic}).
    \end{enumerate}
  \item Walters has shown that in a context different than ours (a
    ``cartesian bicategory''), the ``Frobenius'' and ``sliding and
    splitting'' Beck-Chevalley conditions imply the projection formula
    (as a special case of the ``modular law'');
    see~\cite{walters:modular-law}.
  \end{enumerate}
\end{rmk}

We now consider some examples of indexed categories with indexed coproducts.

\begin{eg}
  If \bC\ is a cocomplete symmetric monoidal category with ordinary
  coproducts that are preserved on both sides by $\ten$ (such as if it
  is closed), then the \Set-indexed category $A\mapsto \bC^A$ has
  indexed coproducts preserved by $\ten$, given by taking ordinary
  coproducts along the fibers of a set-function.
\end{eg}

\begin{eg}
  In the \bS-indexed category $A\mapsto \bS/A$, the left adjoints $f_!$
  are given by composition with $f$, and the compatibility conditions
  are elementary lemmas about pullback squares.
\end{eg}

In derived situations, the adjunctions $f_! \dashv f^*$ are usually
unproblematic to obtain, as is the projection formula, but the
Beck-Chevalley conditions are somewhat more subtle.  In general, when
\bS\ has a homotopy theory that is ``seen'' by $A\mapsto \sC^A$, we can
usually only expect the Beck-Chevalley condition to hold for
\emph{homotopy} pullback squares in \bS\ (that is, commutative squares
for which the canonical map from the vertex to ``the'' homotopy pullback
is a suitable sort of equivalence).  The pullback squares in
Figures~\ref{fig:reindpb},~\ref{fig:frobpb}, and~\ref{fig:slidpb} are
always homotopy pullback squares, but the one in Figure~\ref{fig:mdpb}
(monic diagonals) is not.

Motivated by these examples, we define an \bS-indexed category to have
\textbf{indexed homotopy coproducts} if the reindexing functors have
left adjoints which satisfy the Beck-Chevalley condition for all
homotopy pullback squares, including particularly those in
Figures~\ref{fig:reindpb},~\ref{fig:frobpb}, and~\ref{fig:slidpb} (and
its dual), and all squares obtained from them by taking cartesian
products with a fixed object.  In order to ensure our results remain
valid in derived contexts, we will never assume more than this, although
we will see that some theorems reduce to a slightly simpler form if the
Beck-Chevalley condition for Figure~\ref{fig:mdpb} also holds.  We will
refer to this latter case by saying that \emph{diagonals are monic}.
(Of course, diagonals are always monic in the base category \bS; the
question is whether that monicity is ``visible'' to the fiber categories
$\sC^A$.)

\begin{eg}
  In the examples $A\mapsto \Top/A$ and $A\mapsto \bEx{A}$ of
  Examples~\ref{eg:spaces-over-spaces} and~\ref{eg:spectra-over-spaces},
  the adjunctions $f_!\dashv f^*$ are Quillen adjunctions, and thus
  descend to homotopy categories.  The Beck-Chevalley condition is
  proven in~\cite{maysig:pht} for pullbacks \emph{of fibrations} (see
  also~\cite{shulman:dblderived}).  This includes ``commutativity with
  reindexing'' (Figure~\ref{fig:reindpb}) if $B$ or $D$ is a terminal
  object, but never the Frobenius axiom (Figure~\ref{fig:frobpb}).

  For this reason, in~\cite{maysig:pht} and~\cite{shulman:frbi}, the
  presence of closed structure was invoked to avoid the Frobenius axiom.
  In hindsight, the reason this works is that closedness of the fibers
  and the reindexing functors implies the projection formula, by a
  standard argument---and together with ``commutativity of reindexing''
  this actually implies the Frobenius axiom (see
  \autoref{rmk:frobenius}).

  However, there is a more direct argument: Figures~\ref{fig:frobpb}
  and~\ref{fig:slidpb} are \emph{homotopy} pullback squares, and it is
  straightforward to show that if the Beck-Chevalley condition holds for
  pullbacks of fibrations, it also holds for all homotopy pullback
  squares.  Thus, these two examples have indexed homotopy coproducts;
  the projection formula is also proven in~\cite{maysig:pht}.
\end{eg}

\begin{eg}\label{eg:gpds-bc}
  In \autoref{eg:gpds} of diagrams indexed over groupoids, the left
  adjoints $f_!$ are given by left Kan extension.  If \bC\ has a
  homotopy theory as in \autoref{eg:hogpds}, then these adjunctions are
  also Quillen, and therefore descend to homotopy categories.  In both
  cases, the Beck-Chevalley condition for homotopy pullback squares
  (which is to say, pseudo-pullbacks of groupoids) is straightforward to
  verify, as is the projection formula.

  Note that this would not be true if instead of groupoids we allowed
  \bS\ to consist of arbitrary small \emph{categories}.  In that case,
  the Beck-Chevalley condition would hold only for \emph{comma squares}
  (squares containing a natural transformation which exhibits the
  upper-left vertex as (equivalent to) the comma category of the
  functors on the right and the bottom).  In general, none of the
  squares in Figure~\ref{fig:pullbacks} are comma squares.
\end{eg}

Note that diagonals are not generally monic in \autoref{eg:gpds-bc},
even when \bC\ has no homotopy theory itself.  As we remarked in
\S\ref{sec:shad-from-mono}, for ``derived phenomena'' it suffices that
the homotopy theory of \bS\ is ``visible'' to the indexed category.

Here are a couple additional examples of interest.

\begin{eg}\label{eg:factsys}
  If \bS\ is a category with pullbacks and a factorization system
  $(\mathcal{E},\mathcal{M})$, then $\mathcal{M}$ is automatically
  stable under pullback and composition, and so $A \mapsto
  \mathcal{M}/A$ defines an \bS-indexed symmetric monoidal subcategory
  of $A\mapsto \bS/A$.  The $(\mathcal{E},\mathcal{M})$-factorizations
  give left adjoints $f_!$ for this indexed category, and it has indexed
  coproducts preserved by $\times$ if $\mathcal{E}$ is stable under
  pullback.

  For instance, $(\mathcal{E},\mathcal{M})$ could be the (regular epi,
  mono) factorization system on a regular category, such as \Set\ or
  another topos.  In this case we can also write this indexed category
  as $A\mapsto \mathrm{Sub}(A)$, where $\mathrm{Sub}(A)$ is the poset of
  subobjects of $A$.
\end{eg}

\begin{eg}\label{eg:hyperdoctrine}
  The \emph{hyperdoctrines} of
  Lawvere~\cite{lawvere:adjointness,lawvere:comprehension} are, in
  particular, indexed cartesian monoidal categories with indexed
  coproducts preserved by $\times$.  The Beck-Chevalley conditions for
  the squares in Figures~\ref{fig:reindpb} and~\ref{fig:slidpb} appear
  in~\cite[p.~8--9]{lawvere:comprehension}.  In this case we think of
  \bS\ as a category of types (or contexts) and terms in a type theory,
  and the fiber categories as categories of propositions and proofs in a
  given context.
\end{eg}

From now on, \sC\ will always denote an \bS-indexed symmetric monoidal
category which has indexed homotopy coproducts preserved by $\ten$.

\begin{rmk}
  As mentioned in the introduction, in
  \S\S\ref{sec:string-diagrams}--\ref{sec:colored-strings} we will
  introduce a string diagram calculus for reasoning about indexed
  monoidal categories.  This calculus greatly simplifies the
  constructions to be performed in subsequent sections, and is necessary
  (as a practical matter, though not a mathematical one) for the proofs
  of our main theorems in
  \S\S\ref{sec:fiberwise-proofs}--\ref{sec:tdproofs}.  Some readers may
  find it helpful to have the string diagram calculus in mind all
  through the paper, and we encourage such readers to jump forward and
  read \S\S\ref{sec:string-diagrams}--\ref{sec:colored-strings} now.
\end{rmk}

\section{Symmetric monoidal traces}
\label{sec:smc-traces}

As a preliminary to refining the symmetric monoidal trace, in this
section we recall the definition of the latter and explain how it
interacts with the indexed situation.  Recall that an object $M$ of a
symmetric monoidal category is said to be \textbf{dualizable} if there
exists an object $\rdual{M}$, its \textbf{dual}, and evaluation and
coevaluation morphisms $\eta \colon I \rightarrow M \ten \rdual{M}$ and
$\epsilon\colon \rdual{M} \ten M \rightarrow I$ satisfying the usual
triangle identities.  In this case, the \textbf{trace} of a morphism $f
\colon Q \ten M \rightarrow M \ten P$ is defined to be the composite:
\[\xymatrix{Q\ar[r]^-{\id \ten \eta}&
Q \ten M \ten \rdual{M}\ar[r]^-{ f \ten \id}&
M \ten P \ten \rdual{M}\ar[r]^\cong&
\rdual{M}\ten M \ten P\ar[r]^-{\epsilon\ten \id}&
P
}\]
The most basic case is when $Q$ and $P$ are the unit object $I$, so that
the trace of an endomorphism $f\colon M\to M$ is a morphism
$\tr(f)\colon I\to I$.  References
include~\cite{dp:duality,jsv:traced-moncat,kl:cpt}.

\begin{eg}
  Any finitely generated free abelian group is dualizable in \Ab, and
  the trace of an endomorphism is the usual trace of its matrix with
  respect to any basis.  (To be precise, it is the endomorphism
  $\bbZ\to\bbZ$ determined by \emph{multiplying} by the usual numerical
  trace.)  Similarly, in $\bCh{\bbZ}$, any bounded chain complex of
  finitely generated free abelian groups is dualizable, and the trace of
  an endomorphism is (multiplying by) its \emph{Lefschetz number} (the
  alternating sum of its degreewise traces).
\end{eg}

\begin{eg}
  A ``homotopical'' version of this takes place in the stable homotopy
  category of spectra.  Any topological space gives rise to a spectrum
  called its \emph{suspension spectrum}, and if the space was a
  finite-dimensional manifold, then its suspension spectrum is
  dualizable.  Its dual is a desuspension of the stable normal bundle of
  an embedding of the manifold into a suitably-high-dimensional
  Euclidean space; the evaluation and coevaluation are Thom collapse maps.

  Thus, any endomorphism of such a manifold gives rise to an
  endomorphism of its suspension spectrum.  It turns out that the trace
  of this endomorphism can be identified with the \emph{fixed-point
    index} of the original map; see~\cite{dp:duality}.  The fixed-point
  index is a classical invariant which sums the index of the induced
  vector field in the neighborhood of each fixed point of the endomap.
\end{eg}

It is easy to check that in a \emph{cartesian} monoidal category, the
only dualizable object is the terminal object.  However, as remarked in
the introduction, we often want to extract trace-like information from
endomorphisms of objects in a cartesian monoidal category.  Thus, to get
out of the cartesian situation, we apply a monoidal functor landing in a
non-cartesian monoidal category.

If the cartesian monoidal category \bS\ is the base category of some
indexed symmetric monoidal category with indexed coproducts preserved by
$\ten$, then there is a canonical such functor $\Sigma\colon
\bS\to\sC^\star$, defined as follows.  We send an object $A$ to
\[ \Sigma (A) = (\pi_A)_! I_A \iso (\pi_A)_! (\pi_A)^* I_\star \]
and a morphism $\phi\colon A \to B$ to the composite
\begin{equation}
  \Sigma(A) = (\pi_A)_! (\pi_A)^* I_\star \cong
  (\pi_B)_! \phi_! \phi^* (\pi_B)^* I_\star \too
  (\pi_B)_! (\pi_B)^* I_\star = \Sigma(B).\label{eq:sigma-phi}
\end{equation}
Moreover, this functor $\Sigma$ is strong symmetric monoidal:
\begin{align*}
  \Sigma(A\times B) &= (\pi_{A\times B})_! \pi_{A\times B}^* I_\star\\
  &\iso (\pi_{A\times B})_! \pi_{A\times B}^* (I_\star \ten I_\star)\\
  &\iso (\pi_{A\times B})_! (\pi_{A\times B}^*I_\star \ten \pi_{A\times B}^*I_\star)\\
  &\iso (\pi_A\times \pi_B)_! (\pi_{B}^* \pi_A^*I_\star \ten \pi_{A}^*\pi_B^*I_\star)\\
  &\iso (\pi_A)_!\pi_A^* I_\star \ten (\pi_B)_! \pi_B^* I_\star\\
  &= \Sigma A \ten \Sigma B
\end{align*}
(This is especially obvious in string diagram notation; see
Figure~\ref{fig:sigma} on page~\pageref{fig:sigma}.  We leave the
verification of the coherence of these isomorphisms to the reader.)  If
$\Sigma(A)$ is dualizable, for some $A\in \bS$, then we can ask about
the trace of $\Sigma(f)$.

\begin{eg}
  For a cocomplete symmetric monoidal category \bC\ and the \Set-indexed
  category $A\mapsto \bC^A$, we have $\bC^\star \cong \bC$, and the
  functor $\Sigma$ takes a set $A$ to the \emph{copower} $A \cdot I$ of
  the unit object $I\in\bC$ by the set $A$ (that is, the coproduct of
  $A$ copies of $I$).  If \bC\ is \Ab, then $\Sigma(A)$ is the free
  abelian group on $A$, and similarly for modules, chain complexes, and
  so on.

  In the case of abelian groups, if $A$ is finite, then the abelian
  group $\Sigma(A) = \bbZ[A]$ is dualizable.  And if $f\colon A\to A$ is
  an endofunction, then the trace of $\Sigma(f)$ is just the number of
  fixed points of $f$.
\end{eg}

\begin{eg}
  For the \Top-indexed monoidal category $A\mapsto \Ho(\bEx{A})$, the
  category $\sC^\star$ is the stable homotopy category of spectra.  For
  a space $A$, we have $\Sigma(A) = \Sigma^\infty(A_+)$, the suspension
  spectrum of $A$ with a disjoint basepoint.  This should be regarded as
  a homotopical version of the ``free abelian group'' functor.

  We have already remarked that if $A$ is a finite-dimensional closed
  smooth manifold, then $\Sigma(A)$ is dualizable, and if $f\colon A\to
  A$ is an endomorphism, then the trace of $\Sigma(f)$ is the
  fixed-point index of $f$.
\end{eg}

\begin{eg}
  In the case of  the underived groupoid-indexed category $A\mapsto \Ab^A$ 
  from 
  \autoref{eg:gpds}, we again have $\Ab^\star \cong \Ab$, and
  $\Sigma(A)$ is just the free abelian group on the set of connected
  components of $A$.  If $A$ has finitely many connected components,
  then $\Sigma(A)$ is dualizable, and the trace of $\Sigma(f)$ is the
  number of ``fixed components,'' i.e. the number of isomorphism classes
  of objects $x$ such that $f(x)\cong x$.
\end{eg}

\begin{eg}\label{eg:gpd-smtrace}
  By contrast, for the derived version $A\mapsto \Ho(\bCh{\bbZ}^A)$, the
  fiber over $\star$ is $\Ho(\bCh{\bbZ})$, and $\Sigma(A)$ is the
  homotopy colimit of the constant $A$-diagram of shape $\bbZ$.  This
  can be represented concretely by the complex of chains on the nerve of
  $A$.

  For an example of duality and trace, suppose that $A$ is a
  \emph{finitely generated free groupoid}; that is, a groupoid freely
  generated by some finite directed graph.  Thus, it has finitely many
  objects, and the group of automorphisms of any object is a finitely
  generated free group.  Then $\Sigma(A)$ is equivalent to the following
  chain complex concentrated in degrees 1 and 0:
  \[ \bbZ[A_1] \to \bbZ[A_0].
  \]
  Here $A_0$ is the (finite) set of objects of $A$ and $A_1$ a (finite)
  set of \emph{generating} morphisms, and the differential sends each
  generator $\gamma$ to $t(\gamma) - s(\gamma)$, the difference between
  its source and target objects.  Since this complex is bounded and
  finitely generated free, it is dualizable in $\Ho(\bCh{\bbZ})$; its
  dual can be identified with
  \[ \bbZ[A_0] \to \bbZ[A_1]
  \]
  in degrees 0 and $-1$, where now the differential sends an object $x$
  to the sum of all generating morphisms having target $x$, minus the
  sum of all generators having source $x$.

  Now suppose $f\colon A\to A$ is an endomorphism of $A$.  Therefore, it
  takes each object to another object, and each generating morphism to a
  composite of other generators and inverses of generators.  The induced
  endomorphism $\Sigma(f)$ corresponds, in the above representation, to
  the endomorphism
  \[\xymatrix{ \bbZ[A_1] \ar[r] \ar[d] & \bbZ[A_0] \ar[d] \\
    \bbZ[A_1] \ar[r] & \bbZ[A_0]}
  \]
  which acts as $f$ on objects, and sends each generating morphism to
  the \emph{sum} of the generators occurring in its image, counted with
  multiplicity (where inverses of generators contribute negatively).
  Therefore, the trace $\tr(f)$ is the sum of
  \begin{enumerate}
  \item the number of objects (literally) fixed by $f$, and
  \item for each generating morphism $\gamma$, the number of occurrences
    of $\gamma^{-1}$ in the image $f(\gamma)$, minus the number of
    occurrences of $\gamma$ in $f(\gamma)$.\label{item:gst2}
  \end{enumerate}
  (The perhaps-surprising signs in~\ref{item:gst2} come from the sign in
  the symmetry isomorphism for the tensor product of chain complexes.)
  As it must be, the result is invariant under equivalence of groupoids
  (though not manifestly so from the above description).

  Note that a finitely generated free groupoid is essentially the same
  thing as a finite 1-dimensional CW complex, up to homotopy.  It is
  straightforward to check that under this equivalence, the trace
  calculated above agrees with the topological fixed-point index.
\end{eg}

As mentioned at the beginning of this section, we can also consider
traces of more general morphisms of the form $Q \ten M \to M\ten P$.
Probably the most important symmetric monoidal trace of this form is the
trace of
\[ \Sigma(A) \xto{\Sigma(f)} \Sigma(A) \xto{\Sigma(\Delta_A)}
\Sigma(A) \otimes \Sigma(A)
\]
for some endomorphism $f\colon A \to A$ in \bS, which we call the
\textbf{transfer of $f$}.  Note that the transfer is a morphism $I\to
\Sigma(A)$.

\begin{eg}\label{eg:fam-smtransfer}
  Considering the \Set-indexed category $A\mapsto \Ab^A$, the transfer
  of an endomorphism $f\colon A\to A$ of a finite set is a morphism
  $\bbZ \to \bbZ[A]$.  This is equivalent to a single element of
  $\bbZ[A]$, which turns out to be just the formal sum of all the fixed
  points of $f$.
\end{eg}

\begin{eg}\label{eg:gpd-smtransfer}
  Considering the derived groupoid-indexed category $A\mapsto
  \Ho(\bCh{\bbZ}^A)$, if $A$ is a finitely generated free groupoid, then
  using the small chain complex representing $\Sigma(A)$ from
  \autoref{eg:gpd-smtrace}, the diagonal $\Sigma(A) \to \Sigma(A) \ten
  \Sigma(A)$ can be represented by the morphism
  \[ \xymatrix{0 \ar[r] \ar[d] &
    \bbZ[A_1] \ar[r] \ar[d] &
    \bbZ[A_0] \ar[d] \\
    \bbZ[A_1] \ten \bbZ[A_1] \ar[r] &
    \Big(\bbZ[A_1] \ten \bbZ[A_0]\Big) \oplus
    \Big(\bbZ[A_0] \ten \bbZ[A_1]\Big) \ar[r] &
    \bbZ[A_0] \ten \bbZ[A_0] }
  \]
  which sends each object $x$ to $x\ten x$, and each generating morphism
  $\gamma$ to
  \[\gamma \ten s(\gamma) + t(\gamma)\ten \gamma\]
  (where $s(\gamma)$ and $t(\gamma)$ are the source and target objects
  of $\gamma$, respectively).  The transfer of an endomorphism $f$ is
  then a morphism $\bbZ \to \Sigma(A)$ in $\Ho(\bCh{\bbZ}^A)$, i.e.\ an
  element of $H_0(\Sigma(A))$, which is the free abelian group on the
  set of connected components of $A$.  The coefficient of each component
  in this trace is 0 if that component is not mapped to itself, and
  otherwise it is the trace, as in \autoref{eg:gpd-smtrace}, of $f$
  restricted to that component.
\end{eg}

\begin{eg}
  If $M$ is a closed smooth manifold, the transfer of $f\colon
  M\rightarrow M$ is an element of the $0^{\mathrm{th}}$ stable homotopy
  group of $M_+$.  This latter group can be identified with the free
  abelian group on the set of connected components of $M$, and as in the
  previous example, the coefficient of each component in the trace is
  the sum of the indices of the fixed points in that component.
\end{eg}

Thus, in general, we expect that the transfer of $f$ separates out the
contributions to $\tr(f)$ based on in which connected component of $A$
they lie.  (In the case of a plain set, of course, each element is its
own connected component.)

In particular, the transfer is itself a refinement of the ordinary trace
of $\Sigma(f)$.  Note that the unique map $A\to \star$ to the terminal
object induces an augmentation $\Sigma(A) \to \Sigma(\star) = I$;
standard facts about the functoriality of traces then imply that the
composite
\[ I \xto{\tr(\Sigma(\Delta \circ f))} \Sigma(A) \too I
\]
is simply the trace of $\Sigma(f)$.  In the above examples, the
augmentation simply adds up all the coefficients, and this equality is
then obvious.

\section{Shadows from indexed monoidal categories}
\label{sec:shadows}
\label{sec:bicats}

We now move on to our new refinements of the symmetric monoidal trace,
which we will obtain by constructing a \emph{bicategory} out of an
indexed monoidal category and making use of the general notion of
bicategorical trace introduced in~\cite{kate:traces,PS2}.  Bicategorical
traces require some extra structure on the bicategory, however, so we
begin by recalling that.

The notion of \emph{duality} in a symmetric monoidal category, which we
recalled in the previous section, generalizes easily to bicategories.  A
1-cell $M \colon R \hto S$ in a bicategory is \textbf{right dualizable}
if there is a 1-cell $\rdual{M}\colon S \hto R$ called its \textbf{right
  dual}, and evaluation and coevaluation 2-cells $\eta \colon U_R
\rightarrow M \odot \rdual{M}$ and $\epsilon\colon \rdual{M} \odot M
\rightarrow U_S$ satisfying the usual triangle identities.  (Here
$\odot$ denotes the bicategory composition and $U_R$ is the unit 1-cell
associated to the 0-cell $R$.)  Dual pairs in a bicategory are often
also called \emph{adjoints}, since in the bicategory \Cat\ of
categories, functors, and natural transformations they are precisely
adjoint functors.

The definition of \emph{trace} in a symmetric monoidal category requires
the symmetry isomorphism, which is not present in a bicategory.  In
fact, it wouldn't even make sense to ask for it, since in general
$M\odot N$ and $N\odot M$ live in different hom-categories.  However, as
described in~\cite{kate:traces,PS2}, if we impose an extra structure on
a bicategory, we can define an analogous notion of trace.

Specifically, we define a \textbf{shadow functor} on a bicategory \sB\
to consist of functors
\[\bigsh{-}\maps \sB(R,R) \to \bT\]
for each object $R$ of \sB\ and some fixed category \bT, equipped
with a natural isomorphism
\[\theta\maps \sh{M\odot N}\too[\iso] \sh{N\odot M}\]
for $M\maps R\hto S$ and $N\maps S\hto R$, such that the following
diagrams commute whenever they make sense:
\[\xymatrix{\bigsh{(M\odot N)\odot P} \ar[r]^\theta \ar[d]_{\sh{\fa}} &
  \bigsh{P \odot (M\odot N)} \ar[r]^{\sh{\fa}} &
  \bigsh{(P\odot M) \odot N}\\
  \bigsh{M\odot (N\odot P)} \ar[r]^\theta & \bigsh{(N\odot P)
    \odot M} \ar[r]^{\sh{\fa}} & \bigsh{N\odot (P\odot
    M)}\ar[u]_\theta }\]
\[\xymatrix{\bigsh{M\odot U_R} \ar[r]^\theta \ar[dr]_{\sh{\fr}} &
  \bigsh{U_R\odot M} \ar[d]^{\sh{\fl}} \ar[r]^\theta &
  \bigsh{M\odot U_R} \ar[dl]^{\sh{\fr}}\\
  &\bigsh{M}}\]

If $\sB$ is equipped with a shadow functor and $M$ is a right dualizable
1-cell in $\sB$, then the \textbf{trace} of a 2-cell $f \colon Q \odot M
\rightarrow M \odot P$ is defined to be the composite:
\[\xymatrix{\sh{Q}\ar[r]^-{\sh{ \id \odot \eta}}&
\sh{Q \odot M \odot \rdual{M} }\ar[r]^-{\sh{ f \odot \id}}&
\sh{M \odot P \odot \rdual{M} }\ar[r]^\theta&
\sh{ \rdual{M}\odot M \odot P}\ar[r]^-{\sh{\epsilon\odot \id}}&
\sh{P}
}\]
The most basic case is when $Q$ and $P$ are unit 1-cells, so the shadows
of such units are particularly important; we write $\sh{A} = \sh{U_A}$
and call it \emph{the shadow of $A$}.

The following example generally provides the best intuition.

\begin{eg}\label{eg:bimodules}
  There is a bicategory whose 0-cells are rings (not necessarily
  commutative), whose 1-cells are bimodules, and whose 2-cells are
  bimodule homomorphisms.  Composition of 1-cells is done with the
  tensor product of bimodules, and the right dualizable bimodules
  $\mathbb{Z}\hto R$ are precisely the finitely generated projective
  right $R$-modules.  The shadow of an $R$-$R$-bimodule $M$ is the
  abelian group $M / (r \cdot m = m \cdot r)$.  The bicategorical trace
  specializes to the \emph{Hattori-Stallings trace}
  from~\cite{stallings,hattori}, which takes values in $\sh{R}$.

  There are similar bicategories consisting of chain complexes of
  bimodules over rings, or DGAs.  In this case the shadow is essentially
  Hochschild homology, and the bicategorical trace is the alternating
  sum of the levelwise Hattori-Stallings traces, just as in the
  symmetric monoidal category of chain complexes.
\end{eg}

Now we turn to the task of constructing a bicategory from an indexed
symmetric monoidal category.  For motivation, let \bS\ be a category
with finite limits and consider the indexed symmetric monoidal category
$A\mapsto \bS/A$, whose fiberwise monoidal structures are given by
pullback.  There is also a bicategory whose composition operation is
given by pullback in \bS: its objects are those of \bS\ and its 1-cells
are \emph{spans} $A \leftarrow M \to B$ in \bS.  Thus, let us consider
how we might derive the structure of this bicategory from that of
$A\mapsto \bS/A$.

The category of 1-cells from $A$ to $B$ is isomorphic to $\bS/(A\times
B)$, which is just the fiber category over $A\times B$, so all we need
to do is construct the units and composition.  The central observation
is that if $A\ot M \to B$ and $B\ot N \to C$ are spans, then their
composite $M\times_B N$ is isomorphic to $\Delta_B^*(M\boxtimes N)$,
where $\boxtimes$ is the external (cartesian) product of our indexed
category.  Of course, $\Delta_B^*(M\boxtimes N)$ is actually an object
of $\bS/(A\times B\times C)$, so to make it into an object of
$\bS/(A\times C)$, we need to forget the map to $B$ by pushing forward
along the projection $\pi_B\colon A\times B\times C \to A\times C$; thus
we have $M\times_B N \cong (\pi_B)_! \Delta_B^*(M\boxtimes N)$.
Similarly, the unit span $A\ot A \to A$ can be described (somewhat
perversely) as $(\Delta_A)_! I_A$, where $I_A$ is the unit object of the
symmetric monoidal category $\bS/A$ (namely, the identity $A\to A$).

These considerations motivate the following theorem.  Minus the
statement about the shadow, this theorem was first observed
by~\cite{maysig:pht} in a particular case, and then generalized
in~\cite{shulman:frbi}.

\begin{thm}\label{thm:mf-bi}
  Let \bS\ be a cartesian monoidal category, and let \sC\ be an
  \bS-indexed symmetric monoidal category with indexed homotopy
  coproducts preserved by $\ten$.  Then there is a bicategory $\calBi
  CS$, whose 0-cells are the objects of \bS, with
  \[\calBi CS(A,B) = \sC^{A\times B},
  \]
  and with composition and units defined by
  \begin{align*}
    M\odot N &= (\id_A\times\pi_B\times\id_C)_!
    (\id_A\times \Delta_B\times\id_C)^*(M \boxtimes N) \quad\text{and}\\
    U_A &= (\Delta_A)_! \pi_A^* (U)
  \end{align*}
  Moreover, $\calBi CS$ has a shadow with values in $\sC^\star$, defined by
  \[ \sh{M} = (\pi_A)_!(\Delta_A)^* M.\]
\end{thm}

Note that we can equivalently write the composition and units in terms
of the internal monoidal structure, as $M\odot N = (\pi_B)_! (\pi_C^* M
\ten \pi_A^* N)$ and $U_A = (\Delta_A)_! I_A$.

\begin{proof}
  The associativity isomorphism is the composite
  \begin{multline*}
    \scriptstyle
    {\textstyle(}\id_A\times \pi_C\times \id_D{\textstyle )}_!
    {\textstyle(}\id_A\times \Delta_C\times \id_D{\textstyle )}^*
    \big(\big ({\textstyle(}\id_A\times \pi_B\times \id_C{\textstyle )}_!
    {\textstyle(}\id_A\times \Delta_B\times \id_C{\textstyle )}^*
    {\textstyle(}M\boxtimes N{\textstyle )}\big)\boxtimes P\big)\\
    \scriptstyle
    \too[\cong] {\textstyle(}\id_A\times \pi_B\times \pi_C\times \id_D{\textstyle )}_!
    {\textstyle(}\id_A\times \Delta_B\times \Delta_C\times \id_D{\textstyle )}^*
    {\textstyle(}M\boxtimes N\boxtimes P{\textstyle )}\\
    \scriptstyle
    \too[\cong] {\textstyle(}\id_A\times \pi_B\times \id_D{\textstyle )}_!
    {\textstyle(}\id_A\times \Delta_B\times \id_D{\textstyle )}^*
    \big(M\boxtimes \big({\textstyle(}\id_B\times \pi_C\times \id_D{\textstyle )}_!
    {\textstyle(}\id_B\times \Delta_C\times \id_D{\textstyle )}^*
    {\textstyle(}N\boxtimes P{\textstyle )}\big)\big).
  \end{multline*}
  This uses the ``commutativity with reindexing'' Beck-Chevalley
  condition from Figure~\ref{fig:reindpb}, along with
  pseudofunctoriality isomorphisms for the functors $f^*$ and $f_!$.
  The unit isomorphism is
  \begin{multline*}
    \big(\id_A\times\pi_B\times\id_B\big)_!
    \big(\id_A\times\Delta_B\times\id_B\big)^*
    \Big(M\boxtimes (\Delta_B)_!\pi_B^*(U)\Big)\\
    \too[\cong] (\id_A\times\pi_B\times \id_B)_!
    (\id_A\times \Delta_B)_!(\id_A\times\Delta_B)^*
    (\id_A\times\id_B\times \pi_B)^*(M\boxtimes U)\\
    \too[\cong] M
  \end{multline*}
  This uses the ``Frobenius'' isomorphism from Figure~\ref{fig:frobpb},
  together with pseudofunctoriality for the equality $(\pi \times
  \id)\Delta = \id$.  More details, including proofs of the coherence
  axioms, can be found in~\cite{shulman:frbi}.

  To define the isomorphism $\sh{M\odot N}\cong \sh{N\odot M}$, first
  note that for $P\in \sC^A$ and $Q\in \sC^B$ there is an isomorphism
  \begin{equation}
    P\boxtimes Q\rightarrow \gamma^*(Q\boxtimes P),\label{eq:symep}
  \end{equation}
  where $\gamma\colon A\times B \cong B\times A$ is the symmetry
  isomorphism of \bS.  The isomorphism~\eqref{eq:symep} is defined by
  \begin{align*}
    P\boxtimes Q&\coloneqq (\id_A\times \pi_B)^*P
    \otimes_{A\times B}(\pi_A\times \id_B)^*Q\\
    &\xto{\gamma} (\pi_A\times \id_B)^*Q
    \otimes_{A\times B}(\id_A\times \pi_B)^*P\\
    &\cong \gamma^*(\id_B\times \pi_A)^*Q
    \otimes_{A\times B}\gamma^*(\pi_B\times \id_A)^*P\\
    &\cong \gamma^*((\id_B\times \pi_A)^*Q
    \otimes_{B\times A}(\pi_B\times \id_A)^*P)\\
    &\cong \gamma^*(Q\boxtimes P) 
  \end{align*}
  where the arrow labeled $\gamma$ denotes the symmetry isomorphism of
  $\sC^{A\times B}$.  Given this, we can define the isomorphism $\theta$
  to be the composite
  \begin{align*}
    (\pi_A)_!\Delta^*_A&((\id_A\times \pi_B\times \id_A)_!
    (\id_A\times \Delta_B\times \id_A)^*(M\boxtimes N))
    \\&\cong 
    (\pi_A)_!\Delta_A^*(\id_A\times \id_A\times \pi_B)_!\gamma_!
    (\id_A\times \Delta_B\times\id_A)^*(M\boxtimes N)\\
    &\cong (\pi_A)_!(\id_A\times \pi_B)_!(\Delta_A\times \id_B)^*
    \gamma_!(\id_A\times \Delta_B\times \id_A)^*(M\boxtimes N)\\
    &\cong (\pi_A)_!(\id_A\times \pi_B)_!(\Delta_A\times \id_B)^*
    (\id_A\times \id_A\times  \Delta_B)^*(\id_A\times \gamma
    )_!(M\boxtimes N)\\
    &\cong (\pi_A\times \pi_B)_!(\Delta_A\times  \Delta_B)^*
    (\id_A\times \gamma )_!(M\boxtimes N)\\
    &\xto{\eqref{eq:symep}}  (\pi_A\times \pi_B)_!
    (\Delta_A\times  \Delta_B)^*(\id_A\times \gamma )_!
    \gamma^*(N\boxtimes M)\\
    &\cong (\pi_B)_!(\pi_A\times \id_B)_!(\id_A\times  \Delta_B)^*
    (\Delta_A\times \id_B\times \id_B)^*\gamma^*(N\boxtimes M)\\
    &\cong (\pi_B)_!(\pi_A\times \id_B)_!(\id_A\times  \Delta_B)^*
    \gamma^*(\id_B\times \Delta_A\times \id_B)^*(N\boxtimes M)\\
    &\cong (\pi_B)_!(\Delta_B)^*(\pi_A\times \id_B\times \id_B)_!
    \gamma^*(\id_B\times \Delta_A\times \id_B)^*(N\boxtimes M)\\
    &\cong (\pi_B)_!(\Delta_B)^*(\id_B\times \pi_A\times \id_B)_!
    (\id_B\times \Delta_A\times \id_B)^*(N\boxtimes M)
  \end{align*}
  where $\gamma$ denotes various symmetry isomorphisms in $\sC$.  The
  axioms of a shadow functor can be proven by adapting the methods
  of~\cite{shulman:frbi}.
\end{proof}

\begin{rmk}
  We will see in \S\ref{sec:string-diagrams} that the isomorphisms in
  the proof of \autoref{thm:mf-bi} are \emph{dramatically} simplified by
  the use of string diagram notation.  Figure~\ref{fig:bicatops} on
  page~\pageref{fig:bicatops} shows the operations of $\calBi CS$;
  Figure~\ref{fig:bicat-constr} on page~\pageref{fig:bicat-constr} shows
  the associativity and unit isomorphisms; and Figure~\ref{fig:shadow}
  on page~\pageref{fig:shadow} shows the shadow isomorphism.
\end{rmk}

\begin{rmk}\label{rmk:ismc-bicat-surface}
  The bicategory constructed from an indexed symmetric monoidal category
  as in \autoref{thm:mf-bi} has additional structure: it is a symmetric
  monoidal bicategory in which each object is its own dual.
  (In~\cite{shulman:frbi} it is shown to be a ``fibrant'' symmetric
  monoidal double category, and in~\cite{csmb} it is shown how this
  structure gives rise to a symmetric monoidal bicategory.)  The shadow
  defined above can also be constructed from this additional structure,
  as suggested in~\cite{PS2}, but for our purposes it is easier to
  construct it directly.
\end{rmk}

We now consider some examples of \autoref{thm:mf-bi}.

\begin{eg}\label{eg:monfib-span}
  If \bS\ has finite limits, then from the indexed symmetric monoidal
  category $A\mapsto \bS/A$, \autoref{thm:mf-bi} produces the bicategory
  of spans in \bS.  The shadow of an endospan $A \leftarrow M \to A$ is
  the pullback $(\Delta_A)^* M$, regarded as an object of $\bS =
  \bS/\star$.
\end{eg}

\begin{eg}\label{eg:monfib-mat}
  If \bC\ is a symmetric monoidal category with ordinary coproducts
  preserved by $\ten$, then from the \Set-indexed symmetric monoidal
  category $A\mapsto \bC^A$, this theorem  produces the bicategory
  of \bC-valued \emph{matrices}.  Its objects are sets, of course, and
  its 1-cells $A\hto B$ are $(A\times B)$-indexed families of objects of
  \bC, while its composition is by ``matrix multiplication.''

  The shadow of an $A$-by-$A$ matrix $(M_{a_1,a_2})$ is its ``trace''
  $\coprod_a M_{a,a}$.  In particular, the shadow of a set $A$ is
  $\coprod_{a\in A} I$, which is isomorphic to the copower $A\cdot I$.
\end{eg}

\begin{eg}\label{eg:monfib-rel}
  If \bS\ is a regular category, then from the \bS-indexed monoidal
  category $A\mapsto \mathrm{Sub}(A)$, \autoref{thm:mf-bi} produces the
  bicategory of \emph{relations} in \bS.  This bicategory reflects all
  the logical structure of subobjects in \bS; see for
  instance~\cite{freydscedrov:allegories}.  The shadow of a relation
  $R\hookrightarrow A\times A$ is again the pullback $\Delta^* R$, which
  we can interpret as ``the object of all $a\in A$ such that $R(a,a)$''.
\end{eg}

\begin{eg}\label{eg:monfib-prof}
  From \autoref{eg:gpds}, \autoref{thm:mf-bi} produces a bicategory
  whose objects are groupoids and whose 1-cells $A\hto B$ form the
  diagram category $\bC^{A\times B}$.  Since every groupoid is
  canonically isomorphic to its opposite, we can equivalently regard
  this category as $\bC^{A\op\times B}$, whose objects are variously
  called (\bC-valued) ``profunctors'', ``distributors'', ``bimodules'',
  or ``relators'' from $A$ to $B$.  When we do this, the composition of
  1-cells produced by \autoref{thm:mf-bi} becomes identified with the
  usual \emph{tensor product of functors} construction with which we
  compose profunctors.  The unit object is the profunctor defined by
  $U_A(a_1,a_2) = \hom_A(a_1,a_2)\cdot I$, where $I$ is the unit object
  of \bC.

  Similarly, the shadow of a profunctor $M\colon A\hto A$ is the coend
  \[ \int^{a\in A} M(a,a). \]
  In particular, the shadow of the unit $U_A$ is
  \[ \int^{a\in A} \hom_A(a,a)\cdot I \cong
  \left(\int^{a\in A} \hom_A(a,a)\right)\cdot I.
  \]
  The set $\int^{a\in A} \hom_A(a,a)$ is the quotient of the set
  $\coprod_a \hom_A(a,a)$ of all automorphisms of objects of $A$ by the
  ``conjugacy'' relation $\gamma \sim \alpha^{-1}\gamma\alpha$.  Note
  that this set decomposes into a disjoint union, over all connected
  components of $A$, of the set of conjugacy classes of the isotropy
  group\footnote{The isotropy group of a connected component of $A$ is
    just the group of automorphisms of any object in that component.  It
    is well-defined up to conjugacy, so its set of conjugacy classes is
    well-defined up to isomorphism.} of that connected component.

  This is the usual bicategory of groupoids and \bC-valued profunctors,
  with its usual shadow (as discussed in~\cite{PS2}).  Of course, this
  bicategory sits in a larger bicategory of \emph{categories} and
  profunctors, but for the reasons given in \autoref{eg:gpds-bc}, we
  cannot produce the latter using \autoref{thm:mf-bi}.

  An important special case occurs when $\bC=\Ab$ is the category of
  abelian groups, and we restrict to groups (that is, one-object
  groupoids).  In that case, an \Ab-valued profunctor $G\hto H$ can be
  identified with a bimodule between the group rings $\bbZ[G]$ and
  $\bbZ[H]$.  The above formula for the shadow of a group, in the
  bicategory of profunctors, gives us the free abelian group on its set
  of conjugacy classes:
  \[\sh{G} = \bbZ[\mathrm{Conj}(G)]
  \]
  which is also the shadow $\sh{\bbZ[G]}$ of its group ring in the
  bicategory of bimodules.  Thus, the sub-bicategory of \Ab-valued
  profunctors between groups can be identified with the sub-bicategory
  of \autoref{eg:bimodules} determined by group rings.  (Of course, we
  can also replace $\bbZ$ by any commutative ring.)
\end{eg}

Note that for this example, it is important that the proof of
\autoref{thm:mf-bi} only used the Beck-Chevalley condition for
\emph{homotopy} pullback squares.  The same is true of the next example.

\begin{eg}\label{eg:monfib-ex}
  The indexed monoidal categories $A\mapsto \bEx{A}$ and $A\mapsto
  \Ho(\bEx{A})$ of parametrized spectra give rise to point-set--level
  and derived bicategories.  The latter bicategory is the one studied
  in~\cite{maysig:pht}.
\end{eg}

In general, shadows in \autoref{eg:monfib-ex} do not have a simple
description, but in the important case of the shadow of unit 1-cells, we
can say more.  In fact, there is a general way to compute the shadow of
a unit 1-cell which works in most examples.  The definition gives us:
\[ \sh{A} = \sh{U_A} = (\pi_A)_!(\Delta_A)^* U_A =
(\pi_A)_!(\Delta_A)^* (\Delta_A)_! I_A.
\]
Recalling the functor $\Sigma(A) = (\pi_A)_! I_A$ from
\S\ref{sec:smc-traces}, we see that there is a canonical morphism
$\Sigma(A) \to \sh{A}$ induced by the unit of the adjunction $\Delta_!
\dashv \Delta^*$.  If diagonals are monic, this is an isomorphism, so
that we can identify $\sh{A}$ with $\Sigma(A)$.  This is the case in
Examples~\ref{eg:monfib-span}, \ref{eg:monfib-mat},
and~\ref{eg:monfib-rel}.

On the other hand, if diagonals are not monic, as in
Examples~\ref{eg:monfib-prof} and~\ref{eg:monfib-ex}, then $\sh{A}$ can
be noticeably different from $\Sigma(A)$.  However, we can compute
$\sh{A}$ if we can find a different square
\[ \xymatrix{ LA \ar[r]^-p \ar[d]_q & A \ar[d]^{\Delta} \\
  A \ar[r]_-{\Delta} & A\times A,}
\]
for some object $LA$, which \emph{is} a homotopy pullback (and hence
will usually satisfy the Beck-Chevalley condition, although it is not
one of the squares listed in Figure~\ref{fig:pullbacks}).  For if this
is the case, we have
\[ \sh{A} = (\pi_A)_!\Delta_A^* (\Delta_A)_! I_A \cong
(\pi_A)_! p_! q^* I_A \cong (\pi_{LA})_! I_{LA} = \Sigma(LA).
\]
Moreover, in this case the map $\Sigma(A) \to \sh{A} = \Sigma(LA)$ is
simply the image by $\Sigma$ of the comparison map $A\to LA$ (which is
induced by the homotopy pullback property of $LA$).

\begin{eg}\label{eg:free-loop-space}
  In \Top, the homotopy pullback of $\Delta_A$ with itself is the
  \emph{free loop space} $LA$: its points are continuous maps $S^1 \to
  A$ (no basepoints involved).  Therefore, in the bicategory of
  parametrized spectra, the shadow $\sh{A}$ is $\Sigma(LA) =
  \Sigma_+^\infty(LA)$, the suspension spectrum of the free loop space
  of $A$.  The map $A\to LA$ sends each point $x\in A$ to the constant
  loop at $x$.
\end{eg}

\begin{eg}
  We can compute the shadow in \autoref{eg:monfib-prof} in this way too.
  When $A$ is a groupoid, the homotopy pullback (or pseudo-pullback) of
  $\Delta_A$ with itself is a groupoid $LA$ whose objects are pairs
  $(x,\gamma)$, where $x$ is an object of $A$ and
  $\gamma\in\hom_A(x,x)$, and whose morphisms are ``conjugations''.  The
  map $A\to LA$ sends each object $x$ to $(x,\id_x)$.

  Now recall that in the underived case of \Ab-valued profunctors,
  $\Sigma(B)$ is the free abelian group on the set of connected
  components of $B$.  Therefore, we recover the fact that $\sh{A}$ is
  the free abelian group on the set of conjugacy classes of
  automorphisms in $A$.
\end{eg}

\begin{eg}\label{eg:hochschild}
  The derived version of the profunctors example is perhaps easier to
  see from the free-loop-space perspective.  Now the shadow of $A$ is
  the homotopy quotient of the $LA$-diagram constant at $I$.  When $\bC$
  is chain complexes, this just gives the complex of chains on the nerve
  of $LA$.

  The nerve of $LA$, however, is isomorphic to what is called the
  \emph{cyclic nerve} of $A$.  This is a simplicial set $ZA$ whose
  $n$-simplices are ``composable loops'' of $n+1$ morphisms in $A$:
  \[ x_0 \xto{\alpha_0} x_1 \xto{\alpha_1} \cdots
  \xto{\alpha_{n-1}} x_n \xto{\alpha_n} x_0
  \]
  (note that the starting and ending objects are the same).  The face
  maps of $ZA$ compose pairs of adjacent morphisms, with the final face
  map composing around the loop:
  \[ (\alpha_0\,,\,\dots\,,\,\alpha_{n-1},\,\alpha_n)
  \mapsto (\alpha_n\alpha_0\,,\,\dots\,,\,\alpha_{n-1}).
  \]
  The isomorphism $N(LA) \cong ZA$ is as follows: given an $n$-simplex
  \[ (x_0,\gamma_0) \xto{\alpha_1} (x_1,\gamma_1)
  \xto{\alpha_2} \cdots \xto{\alpha_n} (x_n,\gamma_n)
  \]
  in $N(LA)$, where by definition $\alpha_{i+1}^{-1} \gamma_i
  \alpha_{i+1} = \gamma_{i+1}$, we send it to the $n$-simplex
  \[ (\alpha_0, \alpha_1, \dots, \alpha_n) \]
  in $ZA$, where by definition
  \[ \alpha_0 \coloneqq \alpha_n^{-1} \cdots \alpha_0^{-1} \gamma_0.
  \]
  It is straightforward to check that this is an isomorphism of
  simplicial sets (in fact, an isomorphism of ``cyclic sets'').  See,
  for instance,~\cite{bressler:cyclic}.

  The reason for passing across this isomorphism is that when $A$ is a
  group $G$, the chains on $ZG$ form exactly the \emph{Hochschild
    complex} of the group ring $\bbZ[G]$ (as a bimodule over itself).
  More generally, when $A$ is a skeletal groupoid, the chains on $ZA$
  are the direct sum of the Hochschild complexes of $\bbZ[G]$, as $G$
  runs over the isotropy groups of connected components of $A$.  If $A$
  is not skeletal, then the chains on $ZA$ are homotopy equivalent to
  those on its skeleton.  Thus, for any groupoid $A$, we have a
  quasi-isomorphism
  \[ \sh{A} \;\simeq \bigoplus_{x\in\pi_0(A)}
  \mathit{HH}_*\Big(\bbZ[\hom_A(x,x)]\Big).
  \]

  Note that when using $\sh{G}$ as the target of a map into or out of
  \bbZ, the difference between the derived and underived cases is
  negligible.  This is because maps into and out of \bbZ\ are
  determined, up to homotopy, by the 0th homology of a chain complex,
  and the 0th Hochschild homology of $\bbZ[G]$ is exactly its underived
  shadow: the free abelian group on its set of conjugacy classes.  These
  two cases cover most of our examples, so higher Hochschild homology
  will not make much of an appearance hereafter.
\end{eg}

\section{Fiberwise duality}
\label{sec:fiberwise}

At this point we have all the machinery necessary to prove the
``fiberwise'' comparison \autoref{thm:intro-fiberwise} and its
generalizations, by considering duality and trace in the bicategory
$\calBi CS$.  From now on, we will mostly restrict attention to the
three examples mentioned in the introduction:
\begin{enumerate}
\item The \Set-indexed category $A\mapsto \Ab^A$, where $f^*$ is given
  by reindexing and $f_!$ is given by coproducts over the fibers of $f$.
  In this case $\calBi CS$ is the bicategory of \Ab-valued matrices.
\item The \Top-indexed category $A\mapsto \Ho(\bEx{A})$ of
  parametrized spectra, where $f^*$ is given by pullback and $f_!$ by
  pushout along $f$.  In this case $\calBi CS$ is the bicategory of
  parametrized spectra from~\cite{maysig:pht}.
\item The $\mathbf{Gpd}$-indexed category $A\mapsto \Ho(\bCh{\bbZ}^A)$,
  where $f^*$ is given by reindexing and $f_!$ by homotopy left Kan
  extension.  In this case $\calBi CS$ is equivalent to the derived
  bicategory of $\bCh{\bbZ}$-valued profunctors over groupoids.
\end{enumerate}
Until now we have also considered other examples, such as $A\mapsto
\bS/A$ and $A\mapsto \Ho(\Top/A)$.  However, in order to have
interesting dualities, it is generally necessary for the fibers to be
``additive'' in some sense, as is the case in the examples above.

We now observe that there are two canonical ``embeddings'' of the fiber
categories $\sC^A$ into $\calBi CS$.  Namely, we can consider an object
$M\in\sC^A$ either as a 1-cell $\widehat{M}\maps A\hto \star$ or a
1-cell $\widecheck{M}\maps \star\hto A$ in $\calBi CS$, since the
isomorphisms $\star\times A \iso A\iso A\times \star$ induce
equivalences
\begin{align*}
  \sC^A \simeq \sC^{A\times \star} &= \calBi CS(A,\star) \qquad \text{and}\\
  \sC^A \simeq \sC^{\star\times A} &= \calBi CS(\star,A).
\end{align*}
However, unlike $\sC^A$, which is of course a symmetric monoidal
category, the bicategory structure of $\calBi CS$ does not endow $\calBi
CS(A,\star)$ or $\calBi CS(\star,A)$ with monoidal structures; rather,
it equips them with two composition functors
\begin{align*}
  \odot&\maps \calBi CS(A,\star)\times \calBi CS(\star,A)
  \too \calBi CS(A,A) \simeq \sC^{A\times A}\\
  \odot&\maps \calBi CS(\star,A)\times \calBi CS(A,\star)
  \too \calBi CS(\star,\star) \simeq \sC^\star.
\end{align*}
The relationship of these functors to the monoidal structure of
$\sC^A$ is easily obtained from the definition of $\odot$; we have
\begin{align*}
  \widehat{M} \odot \widecheck{N} &\iso M \boxtimes N \in
  \sC^{A\times A}\\
  \widecheck{M} \odot \widehat{N} &\iso (\pi_A)_!(M\ten_A N) \in \sC^\star
\end{align*}
From this we can deduce fundamental relationships between duality and
trace in the fiber categories of \sC\ versus the bicategory $\calBi CS$.
All proofs in this section involve fairly lengthy string diagram
calculations, so we defer them to \S\ref{sec:fiberwise-proofs}.

The following theorem was proven, for parametrized spectra,
in~\cite{maysig:pht}.

\begin{thm}\label{thm:fiberwise-duality}
  An object $M\in\sC^A$ is dualizable in the symmetric monoidal category
  $\sC^A$ if and only if $\widehat{M}\maps A\hto \star$ is right
  dualizable in the bicategory $\calBi CS$.  Moreover, we have
  $\rdual{(\smash{\widehat{M}})} \iso \widecheck{\rdual{M}}$.
\end{thm}

Therefore, duality in $\calBi CS$ includes, as a special case, duality
in the fiber categories $\sC^A$.  The following proposition helps to
clarify the nature of this duality.

\begin{prop}\label{thm:fibdual-fibers}\
  \begin{enumerate}
  \item If $M\in\sC^A$ is dualizable in $\sC^A$, then for any morphism
    $a\colon \star\to A$, we have that $a^*(M)\in\sC^\star$ is
    dualizable in $\sC^\star$.
  \item Suppose that
    \begin{enumerate}
    \item the fiber categories of $\sC$ are all \emph{closed} monoidal,
      as are the reindexing functors $f^*$,
      and\label{item:fiberwise-condition1}
    \item the collection of all functors $a^*\colon \sC^A \to \sC^\star$
      is jointly conservative
      (isomorphism-reflecting).\label{item:fiberwise-condition2}
    \end{enumerate}
    Then if $a^*(M)\in\sC^\star$ is dualizable for every $a\colon
    \star\to A$, it follows that $M$ is dualizable in $\sC^A$.
  \end{enumerate}
\end{prop}
\begin{proof}
  The first statement is easy, since strong monoidal functors (like
  $a^*$) always preserve dualizability.  The second is proven
  in~\cite[15.1.1]{maysig:pht} for parametrized spectra, but the proof
  requires only the two properties mentioned above.
\end{proof}

In our examples, a morphism $a\colon \star\to A$ is the same as a
``point'' of $A$, and the functor $a^*$ computes the ``fiber'' of $M$
over that point.  Thus, since conditions~\ref{item:fiberwise-condition1}
and~\ref{item:fiberwise-condition2} above hold quite frequently (in
particular, they hold in our three primary examples), it is usually the
case that $M$ is dualizable in $\sC^A$ if and only if each of its fibers
is dualizable in $\sC^\star$.  Thus (following~\cite{maysig:pht}), in
the situation of \autoref{thm:fiberwise-duality} we say that $M$ is
\textbf{fiberwise dualizable}.

We can now state an equivalent form of \autoref{thm:intro-fiberwise}.

\begin{thm}\label{thm:fibtrace4}
  Let $M\in \sC^A$ be fiberwise dualizable and let $f\maps M\to M$ be an
  endomorphism in $\sC^A$.  Then the diagram
  \begin{equation}
    \xymatrix{
      (\pi_A)_!I_A\ar[d]_{(\pi_A)_!\tr(f)}\ar[r]&\sh{A}\ar[d]^{\tr(\widehat{f})}\\
      (\pi_A)_!I_A\ar[r]&I_\star
    }\label{eq:fibtr4}
  \end{equation}
  commutes.
\end{thm}

Note that the lower-left composite in~\eqref{eq:fibtr4} is simply the
adjunct of
\[\tr(f)\colon I_A \to I_A \cong \pi_A^* I_\star
\]
under the adjunction $(\pi_A)_!\dashv \pi_A^*$, so the conclusion of the
theorem is equivalent to the assertion that the symmetric monoidal trace
of $f\maps M\to M$ is equal to the composite
\[ I_A
\to (\pi_A)^*(\pi_A)_! (\Delta_A)^* (\Delta_A)_! I_A
= (\pi_A)^* \sh{A}
\xto{(\pi_A)^* \tr (\widehat{f})} (\pi_A)^* I_\star
= I_A .
\]
(This is the form in which we stated \autoref{thm:intro-fiberwise}.)
In particular, $\tr(f)$ can be recovered from $\tr(\widehat{f})$.
Thus, the bicategorical trace $\tr(\widehat{f})$ carries at least as
much information as the fiberwise trace $\tr(f)$.

Moreover, the top morphism in~\eqref{eq:fibtr4} is an instance of the
``monic diagonals'' Beck-Chevalley morphism.  (In fact, it is exactly
the comparison morphism $\Sigma(A) \to \sh{A}$ that we saw in
\S\ref{sec:shadows}.)  Thus, if diagonals are monic, then the two traces
carry exactly the same information, while otherwise the bicategorical
trace can be strictly more informative.

We now consider fiberwise traces in our three primary examples.

\begin{eg}\label{eg:fib-duality-mat}
  By \autoref{thm:fibdual-fibers}, a \Set-indexed family of abelian
  groups $(M_a)_{a\in A}\in\Ab^A$ is dualizable just when each $M_a$ is
  a dualizable abelian group, which is to say it is finitely generated
  and projective.  If $f = (f_a)\maps M\to M$ is an endomorphism, then
  its trace in $\Ab^A$ is the family of traces of the endomorphisms $f_a
  \colon M_a \to M_a$:
  \[\tr(f) = (\tr f_a)_{a\in A} \colon (\bbZ)_{a\in A} \too (\bbZ)_{a\in A}.
  \]
  Since diagonals are monic in this case, the trace of $\widehat{f}$ is
  just the adjunct of this under the adjunction $(\pi_A)_! \dashv
  (\pi_A)^*$, which turns out to be the induced map
  \[\bigoplus_a \bbZ \xto{[\tr(f_a)]} \bbZ.
  \]
  Of course, by the universal property of a coproduct, knowing this
  morphism is the same as knowing the individual morphisms $\tr(f_a)$.
\end{eg}

\begin{eg}\label{eg:fib-duality-gpd}
  Again, by \autoref{thm:fibdual-fibers}, a $\mathbf{Gpd}$-indexed
  diagram $M$ of chain complexes is dualizable just when each chain
  complex $M(a)$ is dualizable.  In the underived context, this means it
  is finitely generated and projective, while in the derived case it
  just means it is quasi-isomorphic to such a complex.  The symmetric
  monoidal unit $I_A\in \bCh\bbZ^A$ is just the constant functor at the
  unit object $\bbZ\in \bCh{\bbZ}$, and for an endomorphism $f\colon
  M\to M$, the symmetric monoidal trace is the natural transformation
  $I_A\to I_A$ consisting of the individual traces of the morphisms $f_a
  \colon M(a) \to M(a)$.

  As for the bicategorical trace, recall that the shadow of a groupoid
  $A$ is the copower of the unit $I = \bbZ$ by the set of conjugacy
  classes of automorphisms in $A$.  A computation shows that
  $\tr(\widehat{f})\colon \sh{A} \to \bbZ$ sends each automorphism
  $\gamma \in \hom_A(a,a)$ to the ordinary symmetric monoidal trace of
  the composite
  \[M(a) \xto{M(\gamma)} M(a) \xto{f_a} M(a).
  \]
  (Cyclicity of the trace implies that this is invariant under
  conjugacy.)  Since the unit $I_A \to (\pi_A)^* \sh{A}$ picks out the
  identities, it is clear that the composite of these two will find
  exactly the traces of the $f_a$'s.  As we saw above, this is the
  symmetric monoidal traces in $\bCh\bbZ^A$, as asserted by
  \autoref{thm:fiberwise-duality}.  Note that in this case, the
  bicategorical trace does carry strictly more information.
\end{eg}

\begin{eg}\label{eq:fib-duality-sp}
  Once again, by \autoref{thm:fibdual-fibers}, a parameterized spectrum
  $E$ over $B$ is fiberwise dualizable if and only if each of its fibers
  is dualizable in the usual stable homotopy category $\Ho(\Sp)$.

  In particular, if $p\colon E\rightarrow B$ is a fibration with fiber
  $F$ such that $\Sigma^\infty(F_+)$ is dualizable, then the ``fiberwise
  suspension spectrum'' $\Sigma_B^\infty(E\amalg B)$ is fiberwise
  dualizable over $B$.  In this case, for a fiberwise map $f\colon
  E\rightarrow E$, the symmetric monoidal trace of
  $\Sigma^\infty_B(f\amalg \id_B)$ in $\Ho (\Sp_B)$ is the
  \emph{fiberwise fixed point index} of $f$, \cite{d:index}.  This is a
  fiberwise endomorphism of the ``parametrized sphere spectrum'' $S_B$
  over $B$ (this is the unit object $I_A$).  Since $S_B \simeq (\pi_B)^*
  S$, where $S$ is the ordinary sphere spectrum, by adjunction this map
  is equivalent to a map
  \[(\pi_B)_! S_B \cong \Sigma^\infty(B_+)\rightarrow S.
  \]
  The homotopy classes of maps $\Sigma^\infty(B_+)\rightarrow S$ make up
  the $0^{th}$ \emph{stable cohomotopy} of $B_+$.  If $b\colon
  \star\rightarrow B$ is the inclusion of a point $b$ in $B$, the
  composite
  \[\xymatrix@C=3pc{
    S = \Sigma^\infty(\star_+) \ar[r]^-{\Sigma^\infty(b_+)} &
    \Sigma^\infty(B_+)\ar[rr]^-{\tr(\Sigma^\infty_B(f\amalg \id_B))} &&
    S}
  \]
  is the symmetric monoidal trace of the induced map
  $\Sigma^\infty((f|_{p^{-1}(b)})_+)$ on the fiber over $b$.  Thus, just
  as in the previous examples, the symmetric monoidal fiberwise trace
  consists of all the traces on all the fibers, ``put together'' in a
  suitable way (here, in a way ``continuously parametrized'' by $B$).

  As for the bicategorical trace, recall from
  \autoref{eg:free-loop-space} that the shadow of $U_B$ is the
  suspension spectrum of the free loop space of $B$.  Thus,
  \autoref{thm:fibtrace4} gives a factorization of the fiberwise trace
  as
  \[\Sigma^\infty(B_+)\rightarrow \Sigma^\infty(LB_+)\rightarrow S.
  \] 
  The first map is induced by the inclusion of $B$ into $LB$ as constant
  loops.  As for the second, a loop $\gamma$ in $B$ based at $b$ induces
  an endomorphism $E_\gamma$ of $p^{-1}(b)$, and the trace of
  ($\Sigma^\infty_+$ applied to) the composite
  \[\xymatrix@C=3pc{
    p^{-1}(b) \ar[r]^{E_\gamma} & p^{-1}(b) \ar[r]^{f|_{p^{-1}(b)}} & p^{-1}(b)
  }\]
  gives an endomorphism $S\rightarrow S$ of the sphere spectrum.
  Together, these maps comprise the bicategorical trace.
\end{eg}

\begin{rmk}\label{eg:shaug}
  As remarked previously, the unit object $I_A$ of $\sC^A$ is always
  fiberwise dualizable.  The symmetric monoidal trace of $\id_{I_A}$ is
  just itself, but its bicategorical trace is a nontrivial morphism
  $\sh{A} \to I_\star$, which we can view as an ``augmentation'' of
  $\sh{A}$.  In \S\ref{sec:total} we will see that this augmentation
  plays an important role in comparing the two types of traces
  appropriate for total duality.
\end{rmk}

We would also like a version of \autoref{thm:fibtrace4} for twisted
traces.  There are various forms of this, depending on where we choose
the twisting objects $Q$ and $P$ to live.

\begin{thm}\label{thm:fibtrace}
  Let $M\in \sC^A$ be fiberwise dualizable.
  \begin{enumerate}
  \item For any $Q\in \sC^{A\times A}$, $P\in \sC^{\star}$, and $g\colon
    Q\odot \widehat{M} \to \widehat{M}\odot P$, there is a corresponding
    morphism $\overline{g}\colon (\Delta_A)^* Q \otimes M \to M \otimes
    \pi_A^* P$ in $\sC^A$, such that the bicategorical
    trace\label{item:fibtrace1}
    \[ \tr(g) \colon (\pi_A)_! (\Delta_A)^* Q = \sh{Q}\too \sh{P} = P \]
    and the symmetric monoidal trace
    \[ \tr(\overline{g}) \colon (\Delta_A)^* Q \too (\pi_A)^* P \]
    are adjuncts under the adjunction $(\pi_A)_! \dashv (\pi_A)^*$.
    In other words, the following triangles commute.
    \begin{equation}\label{eq:fibtr1}
      \xymatrix{ (\pi_A)_! (\Delta_A)^* Q
        \ar[d]_{(\pi_A)_!\tr(\overline{g})} \ar[dr]^{\tr(g)}\\
        (\pi_A)_!(\pi_A)^* P \ar[r] & P}
      \quad
      \xymatrix{ (\Delta_A)^* Q \ar[r] \ar[dr]_{\tr(\overline{g})} &
        (\pi_A)^*(\pi_A)_! (\Delta_A)^* Q \ar[d]^{(\pi_A)^*\tr(g)} \\
        & (\pi_A)^* P}
    \end{equation}
  \item For any $Q,P\in\sC^A$, and $f\colon Q\otimes M \to M\otimes P$,
    there is a corresponding morphism $\widetilde{f}\colon (\Delta_A)_!
    Q \odot \widehat{M} \to \widehat{M} \odot (\pi_A)_! P$ such that the
    following triangle commutes.\label{item:fibtrace2}
    \begin{equation}
      \xymatrix{ (\pi_A)_! Q \ar[r] \ar[dr]_{(\pi_A)_! \tr(f)} &
        (\pi_A)_! (\Delta_A)^* (\Delta_A)_! Q \ar[d]^{\tr(\widetilde{f})}\\ 
        & (\pi_A)_! P 
      }\label{eq:fibtr2}
    \end{equation}
  \item For any $Q\in\sC^A$ and $P\in \sC^{\star}$, there is a bijection
    between morphisms $f\colon Q\otimes M \to M\otimes (\pi_A)^*P$ and
    morphisms $\widehat{f}\colon (\Delta_A)_! Q \odot \widehat{M} \to
    \widehat{M} \odot P$, and for a corresponding pair of such
    morphisms, the following square commutes.\label{item:fibtrace3}
    \begin{equation}
      \xymatrix{ (\pi_A)_! Q \ar[r] \ar[d]_{(\pi_A)_! \tr(f)} &
        (\pi_A)_! (\Delta_A)^* (\Delta_A)_! Q \ar[d]^{\tr(\widehat{f})}\\
        (\pi_A)_! (\pi_A)^* P \ar[r] & P }\label{eq:fibtr3}
    \end{equation}
  \end{enumerate}
\end{thm}

Note that $g$ in~\ref{item:fibtrace1} is of the maximally general form
for a morphism of which we could take the bicategorical trace with
respect to $\widehat{M}$.  Thus, the conclusion of~\ref{item:fibtrace1}
says that for any bicategorical trace with respect to $\widehat{M}$,
there is a symmetric monoidal trace with respect to $M$ which carries
exactly the same information.

By contrast, in~\ref{item:fibtrace2}, $f$ is of the maximally general
form for a morphism of which we could take the symmetric monoidal trace
with respect to $M$.  Thus, the conclusion of~\ref{item:fibtrace2} says
that for any symmetric monoidal trace with respect to $M$, there is a
bicategorical trace with respect to $\widehat{M}$ which is related to it
in a suitable way.

Unlike  the situation of~\ref{item:fibtrace1}, neither of these
traces can be recovered from the other in general.  However, if
diagonals are monic, then the top map in~\eqref{eq:fibtr2} is an
isomorphism, and thus the bicategorical trace $\tr(\widetilde{f})$ can
be recovered from the symmetric monoidal trace $\tr(f)$; hence the
latter carries at least as much information as the former.

In the case of~\ref{item:fibtrace3}, the lower-left composite
\[(\pi_A)_!Q \xto{(\pi_A)_! \tr(f)} (\pi_A)_! (\pi_A)^* P \too P
\]
is the adjunct of $\tr (f)$ under the adjunction $(\pi_A)_! \dashv
(\pi_A)^*$, and thus carries exactly the same information.  Therefore,
in the situation of~\ref{item:fibtrace3}, the bicategorical trace
carries at least as much information than the symmetric monoidal one.
If diagonals are monic, then the top morphism in~\eqref{eq:fibtr3} is an
isomorphism, and so the two traces carry exactly the same information.

Note that \autoref{thm:fibtrace4} is a special case of
\autoref{thm:fibtrace}\ref{item:fibtrace3}, taking $Q=I_A$ and
$P=I_\star$; thus we only need to prove
Theorems~\ref{thm:fiberwise-duality} and~\ref{thm:fibtrace}.  In
\S\ref{sec:fiberwise-proofs}, we will do this using string diagram
calculations.

\begin{eg}\label{eg:fam-fibtransfer}
  Let $A$ be a set and $F\in \mathbf{FinSet}^A$ be an $A$-indexed family of
  finite sets, and define $M(a) = \bbZ[F(a)]$.  Then $M\in\Ab^A$ and is
  fiberwise dualizable, since each $M(a)$ is finitely generated and
  free.  Moreover, we have a diagonal morphism
  \[ \Delta\colon M \to M\otimes M \]
  induced by the diagonal $F \to F\times F$.

  Now if $f\colon M \to M$ is any endomorphism, then the composite
  $\Delta \circ f\colon M \to M\otimes M$ is of the form assumed in
  \autoref{thm:fibtrace}\ref{item:fibtrace2}, with $Q=I_A$ and $P=M$.
  Its symmetric monoidal trace, of course, is the transfer of $f$ in the
  symmetric monoidal category $\Ab^A$.  This is a morphism $I_A \to M$,
  which simply picks out at each $a\in A$ the transfer of $f_a$, as in
  \autoref{eg:fam-smtransfer}.  The image of this under $(\pi_A)_!$ is a
  morphism $\bbZ[A] \to \bigoplus_{a\in A} M(a)$, which of course sends
  each generator $a\in A$ to the transfer of $f_a$.  Since diagonals are
  monic in this case, the top morphism in~\eqref{eq:fibtr2} is an
  isomorphism, and so the bicategorical trace of $\widetilde{f}$ is
  exactly the same.
\end{eg}

\begin{eg}\label{eg:gpd-fibtransfer}
  We can perform the same construction as in the previous example, but
  now starting with a \emph{groupoid} $A$ and a finite-set-valued
  functor $F\in \mathbf{FinSet}^A$.  Again we obtain a fiberwise
  dualizable object $M$ equipped with a diagonal, so we can apply
  \autoref{thm:fibtrace}\ref{item:fibtrace2} to $\Delta \circ f$ for any
  $f\colon M\to M$.  The symmetric monoidal transfer will, again,
  consist exactly of all the transfers of the fiber maps $f_a\colon M(a)
  \to M(a)$, and its image under $(\pi_A)_!$ will again be the sum of
  all of these.

  Since $Q=I_A$, the domain $(\pi_A)_! (\Delta_A)^* (\Delta_A)_! Q$ of
  the bicategorical trace is isomorphic to $\sh{A}$.  Unsurprisingly,
  the bicategorical trace $\tr(\widetilde{f})\colon \sh{A} \to (\pi_A)_!
  M$ sends each conjugacy class $[\gamma]$ to the symmetric monoidal
  transfer of the composite
  \[M(a) \xto{M(\gamma)} M(a) \xto{f_a} M(a)
  \]
  thereby combining Examples~\ref{eg:fib-duality-gpd} and
  \ref{eg:fam-fibtransfer}.
\end{eg}

\section{Base change objects}
\label{sec:base-change}

In the previous section we stated our first comparison theorem for
traces, with the proof deferred to \S\ref{sec:fiberwise-proofs}.  In
\S\ref{sec:total} we will state and prove our second such theorem, but
first we need to introduce some more structure implied by an indexed 
symmetric monoidal category.

Recall that we have equivalences $\sC^A\simeq \calBi CS(A,\star) \simeq
\calBi CS(\star,A)$.  In fact, the action of the reindexing and
indexed-coproduct functors is also visible in $\calBi CS$, as follows.
For any morphism $f\colon A\to B$ in \bS, we define its \textbf{base
  change objects} to be
\begin{align*}
  B_f &= (\id_B\times f)^* U_B \qquad\text{and}\\
  _fB &= (f\times \id_B)^* U_B,
\end{align*}
regarded as 1-cells $B\hto A$ and $A\hto B$ in $\calBi CS$,
respectively.  Note that by the Beck-Chevalley condition from
Figure~\ref{fig:slidbc} and its dual (and using the definition of
$U_B$), we have isomorphisms
\begin{align*}
  B_f &=(\id\times f)^*U_B\\
  &= (\id\times f)^*(\Delta_B)_!\pi_B^*(U)\\
  &\cong (f\times \id)_!(\Delta_A)_!f^*\pi_B^*(U)\\
  &\cong (f\times \id)_!(\Delta_A)_!\pi_A^*(U)\\
  &\cong (f\times\id)_! U_A
\end{align*}
and $_fB \cong (\id\times f)_! U_A$.  This equivalence for $_fB$ is
shown using string diagrams in Figure~\ref{fig:bco} on
page~\pageref{fig:bco}.

\begin{eg}
  In the bicategory of spans in \bS, the base change objects of a map
  $f\colon A \to B$ in \bS\ are the spans $A \xleftarrow{\id} A \xto{f}
  B$ and $B \xleftarrow{f} A \xto{\id} A$.
\end{eg}

\begin{eg}
  In the bicategory of \bC-valued matrices, the base change objects of a
  set-function $f\colon A \to B$ are the matrix
  \[({}_fB)_{a,b} =
  \begin{cases}
    I &\quad \text{if }b=f(a)\\
    \emptyset &\quad\text{otherwise}
  \end{cases}
  \]
  and its transpose.
\end{eg}

\begin{eg}
  In the bicategory of \bC-valued profunctors, the base change objects
  of a functor $f\colon A\to B$ are defined by
  \[ ({}_fB)(a,b) = \hom_B(f(a),b) \cdot I
  \qquad \text{and}\qquad
  (B_f)(b,a) = \hom_B(b,f(a)) \cdot I
  \]
\end{eg}

The purpose of introducing the base change objects is to prove the
following lemma.

\begin{lem}\label{thm:bco-action}
  For any $f\colon A\to B$, $M\in\sC^B$, and $N\in\sC^A$, we have
  isomorphisms
  \begin{alignat*}{2}
    \widehat{f^* M} &\cong {}_fB \odot \widehat{M} &\qquad
    \widehat{f_! N} &\cong B_f \odot \widehat{N}\\
    \widecheck{f^* M} &\cong \widecheck{M}\odot B_f &\qquad
    \widecheck{f_! N} &\cong \widecheck{N}\odot {}_fB
  \end{alignat*}
\end{lem}
\begin{proof}
  The first isomorphism is the composite
  \begin{align*}
    \widehat{f^* M} 
    &\cong f^*(\id_B\times \pi_B)_!(\Delta_B)_!(U\boxtimes M)\\
    &\cong (\id_A\times \pi_B)_!(f\times \id_B)^*(\Delta_B)_!(U\boxtimes M)\\
    &\cong (\id_A\times \pi_B)_!(f\times \id_B)^*(\Delta_B)_!
    (\Delta_B)^*(\pi_B\times \id_B)^*(U\boxtimes M)\\
    &\cong (\id_A\times \pi_B)_!(f\times \id_B)^*(\id_B\times \Delta_B)^*
    (\Delta_B\times \id_B)_!(\pi_B\times \id_B)^*(U\boxtimes M)\\
    &\cong (\id_A\times \pi_B)_!(\id_A\times \Delta_B)^*
    (f\times \id_B\times \id_B)^*(\Delta_B\times \id_B)_!
    (\pi_B\times \id_B)^*(U\boxtimes M)\\
    &\cong (\id_A\times \pi_B)_!(\id_A\times \Delta_B)^*
    (((f\times \id_B)^*(\Delta_B)_!\pi_B^*(U))\boxtimes M)\\
    &\cong {}_fB \odot \widehat{M}.
  \end{align*}
  These isomorphisms come from the Frobenius axiom and the equality
  $(\pi\times\id)\Delta=\id$.  The other three are analogous.
\end{proof}

The first isomorphism in \autoref{thm:bco-action} is shown using string
diagrams in Figure~\ref{fig:bco-action} on page~\pageref{fig:bco-action}
(and the others are analogous).  As usual, the string diagram version is
considerably simpler.

\begin{rmk}\label{rmk:bicat-ismc-surface}
  In particular, this means that the indexed category
  $\sC\colon\bS\op\to \Cat$ is recoverable from the bicategory $\calBi
  CS$ together with the base change objects.  If we recall from
  \autoref{rmk:ismc-bicat-surface} that $\calBi CS$ is actually a
  symmetric monoidal bicategory, then we can recover the monoidal
  structure of $\sC$ from it as well.  The external product
  $\boxtimes\colon \sC^A\times\sC^B \to \sC^{A\times B}$ can be
  identified with the functor
  \[\boxtimes\colon \calBi CS(A,\star) \times \calBi CS(B,\star) \too
  \calBi CS(A\times B,\star)\]
  arising from the monoidal structure of $\calBi CS$.

  In particular, since $M\otimes_A N \cong (\Delta_A)^*(M\boxtimes N)$,
  we also have
  \[ \widehat{M\otimes_A N} \cong {}_{\Delta}(A\times A) \odot
  (\widehat{M} \boxtimes \widehat{N}).
  \]
  In the language of~\cite{ds:monbi-hopfagbd,franco:autpsmon}, this
  means that $\otimes_A$ is the ``convolution monoidal structure'' on
  $\calBi CS(A,\star)$ induced by the ``autonomous map pseudo-comonoid
  structure'' of $A$ arising from $\Delta$.  Under this identification,
  \autoref{thm:fiberwise-duality} becomes a special case
  of~\cite[Prop.~4.6]{franco:autpsmon}.  (We are indebted to Daniel
  Sch\"appi for pointing out this connection.)  It seems likely that
  versions of Theorems~\ref{thm:fibtrace4} and~\ref{thm:fibtrace} are
  also true in that generality.
\end{rmk}

Essentially the same proof as of \autoref{thm:bco-action} shows that
more generally, for any $f\colon A\to B$ and $M\in \sC^{B\times C}
\simeq \calBi CS(B,C)$, we have ${}_fB \odot M\cong (f\times \id)^* M$.
In particular, if $M = {}_g C$ for some $g\colon B\to C$, we have
\[{}_fB \odot {}_g C \cong (f\times \id)^*({}_gC) \cong
(f\times \id)^*(g\times \id)^* U_C \cong (gf \times\id)^* U_C \cong {}_{gf}C.
\]
It is easy to verify coherence of these isomorphisms, so that the
assignment $f\mapsto {}_f B$ defines a pseudofunctor $\bS \to \calBi CS$
which is the identity on objects.

We can also conclude:

\begin{lem}\label{thm:bco-dual-indexed}
  For any $f\colon A \to B$ in \bS, we have a dual pair $({}_fB,B_f)$.
\end{lem}
\begin{proof}
  The remarks above show that the functors ${}_fB \odot -$ and $B_f
  \odot -$ are naturally adjoint.  The result then follows from the
  bicategorical Yoneda lemma.
\end{proof}

Alternative proofs of Lemmas~\ref{thm:bco-action}
and~\ref{thm:bco-dual-indexed} can be found in~\cite{shulman:frbi},
using the language of double categories.  Together, they imply that
$\calBi CS$ is a \emph{proarrow equipment} in the sense
of~\cite{wood:proarrows-i} (called a \emph{framed bicategory}
in~\cite{shulman:frbi}).

\section{Total duality}
\label{sec:total}

In \S\ref{sec:fiberwise}, we saw that duality in the symmetric monoidal
category $\sC^A$ can be identified with a special case of duality in the
bicategory $\calBi CS$.  However, even for objects of $\sC^A$, the
bicategory $\calBi CS$ introduces an additional, new type of duality: we
can ask whether $\widecheck{M}$, rather than $\widehat{M}$, is right
dualizable.  In this case we say that $M$ is \textbf{totally
  dualizable}.

In contrast to fiberwise duality, total duality does not take place
entirely within $\sC^A$, and thus can incorporate information about the
object $A$ as well.  This is made precise by the following theorems,
whose proofs are surprisingly easier than the corresponding ones in
\S\ref{sec:fiberwise}.  (However, we will still need to postpone some
computations to be done with string diagrams in \S\ref{sec:tdproofs}.)
The first theorem was proven by~\cite{maysig:pht} in the case of
parametrized spectra, 
and we give the same proof here.

\begin{thm}\label{thm:total-duality}
  If $M\in \sC^A$ is totally dualizable, then $(\pi_A)_! M$ is
  dualizable in $\sC^\star$.
\end{thm}
\begin{proof}
  By construction, the base change objects for $\pi_A\colon A\to \star$
  can be identified with $\widehat{I_A}$ and $\widecheck{I_A}$, and in
  particular we have a dual par $(\widehat{I_A},\widecheck{I_A})$.
  Therefore, by \autoref{thm:bco-action} we have $(\pi_A)_! M \cong
  \widecheck{M} \odot \widehat{I_A}$.  But since the composite of dual
  pairs is again a dual pair, if $\widecheck{M}$ is right dualizable
  then so must $(\pi_A)_! M$ be, as a 1-cell $\star\hto\star$.  However,
  it is easy to verify that $\calBi CS(\star,\star)\eqv \sC^\star$ as
  monoidal categories; thus $(\pi_A)_!M$ is dualizable in $\sC^\star$ as
  desired.
\end{proof}

In particular, for $M=I_A$ we have:

\begin{cor}
  If $I_A\in\sC^A$ is totally dualizable, then $\Sigma (A)$ is
  dualizable in $\sC^\star$.
\end{cor}

This is in sharp contrast to fiberwise duality, since $I_A$ (being the
unit object of a symmetric monoidal category) is always fiberwise
dualizable.  However, we will see later that objects $M\in \sC^A$ other
than $I_A$ can be totally dualizable without implying any finiteness
conditions on $A$.

The relationship between bicategorical and symmetric monoidal traces for
totally dualizable objects is quite simple.

\begin{thm}\label{thm:totaltr1}
  If $M\in \sC^A$ is totally dualizable and $f\colon M\to M$ is an
  endomorphism in $\sC^A$, then the following triangle commutes.
  \[\xymatrix{ I_\star
    \ar[r]^-{\tr(\widecheck{f})}\ar[dr]_-{\tr((\pi_A)_! f)}
    & \sh{A} \ar[d]\\
    & I_\star
  }\]
  Here the right-hand vertical map is the augmentation of the shadow
  from \autoref{eg:shaug}.
\end{thm}

We will prove a more general version of this theorem momentarily, but
first we consider some examples.

\begin{eg}\label{eg:total-duality-mat}
  We have seen in \autoref{eg:fib-duality-mat} that an $A$-indexed
  family of abelian groups $M=(M_a)_{a\in A}\in\Ab^A$ is fiberwise
  dualizable just when each $M_a$ is dualizable, which is to say,
  finitely generated and free.  On the other hand, if $M$ is totally
  dualizable, then \autoref{thm:total-duality} implies that $\bigoplus_a
  M_a$ must be dualizable, which implies that all but finitely many of
  the $M_a$ must be zero and the rest must be dualizable.

  In fact, this latter condition is also sufficient.  For total
  dualizability of $M$ means that there is an $A$-indexed family
  $(N_a)_{a\in A}$ and morphisms
  \begin{align*}
    \eta&\maps U_\star \too \widecheck{M}\odot \widehat{N}\\
    \ep &\maps \widehat{N} \odot \widecheck{M} \too U_A.
  \end{align*}
  satisfying appropriate triangle identities.  Under the assumption that
  $\bigoplus_a M_a$ is dualizable, we may take $N_a = \rdual{M_a}$ and
  define
  \[\eta\colon \mathbb{Z} \too \bigoplus_{a\in A} M_a\ten N_a,\]
  to be the sum of all the coevaluations of the $M_a$, and
  \[\ep\colon N_{a_1}\ten M_{a_2} \too
  \begin{cases}
    \mathbb{Z} & \text{if }a_1=a_2\\
    0 & \text{if }a_1\neq a_2.
  \end{cases}
  \]
  to be the evaluation of $M_a$ if $a_1=a_2$, and zero otherwise.  The
  same approach works for a finite family of dualizable objects in any
  \bC, as long as \bC\ is additive so that we can add up the
  coevaluations.

  Note that in particular, this means that the unit object $I_A$ is
  totally dualizable if and only if $A$ is finite.  On the other hand,
  there can be totally dualizable objects in $\bC^A$ for arbitrarily
  large $A$, as long as their ``support'' is a finite subset of $A$.

  Finally, if $M$ is totally dualizable as above, and $f\colon M\to M$
  is an endomorphism, then the trace of $\widecheck{f}$ is the induced
  map
  \[\bbZ \xto{\sum_a\tr(f_a)} \bigoplus_a \mathbb{Z}.
  \]
  which is the direct sum of all the traces of the components $f_a$.
  The trace of $(\pi_A)_! f$ is the numerical sum of all these traces
  (as an endomorphism of $\bbZ$), and the augmentation $\sh{A} =
  \bigoplus_a \bbZ \to \bbZ$ is ``addition''.  Thus, we can see that the
  bicategorical trace clearly carries more information than the
  symmetric monoidal one.
\end{eg}

\begin{eg}\label{eg:total-duality-sup}
  The finiteness restriction in the previous example arises because in
  an abelian group (or in a hom-set of any additive category) we can
  only add up finitely many elements.  An example of a situation in
  which we can ``add up'' infinitely many elements is the monoidal
  category $\bC=\mathbf{Sup}$ of \emph{suplattices}, i.e.\ partially
  ordered sets with arbitrary suprema, and suprema-preserving maps.  We
  think of the supremum of a subset of a suplattice as an infinitary
  version of the ``sum'' of that subset.

  $\mathbf{Sup}$ is a monoidal category, with a tensor product whose
  universal property says that it represents ``bilinear maps'', i.e.\
  functions that preserve suprema in each variable separately.  The unit
  object is the two-element lattice.  Many more suplattices than abelian
  groups are dualizable, precisely because we can ``add up'' infinitely
  many elements: for instance, any power set $\mathcal{P}(A)$ is
  dualizable as a suplattice.  However, traces in $\mathbf{Sup}$ carry
  correspondingly less information: there are only two maps $I\to I$, so
  a trace can only record the presence or absence of a ``fixed point''
  rather than any numerical information.

  The same sort of argument as above now shows that $M\in
  \mathbf{Sup}^A$ is totally dualizable just when each $M_a$ is a
  dualizable suplattice, irrespective of the cardinality of $A$.  (For
  the coevaluation, we take a pointwise supremum.)  In particular, $I_A$
  is totally dualizable for \emph{any} set $A$.

  If $f\colon M\to M$ is an endomorphism of such an $M$, then the trace
  of $\widecheck{f}$ is the map
  \[ I \too \coprod_a I \cong \mathcal{P}(A) \] which picks out the
  subset of those $a\in A$ for which the trace of $f_a$ is nonzero.  Its
  composite with the augmentation $\mathcal{P}(A) \to I$ remembers only
  whether the trace of \emph{any} $f_a$ was nonzero.
\end{eg}

\begin{eg}\label{eg:total-duality-prof}
  Now we consider the underived example of groupoids and \Ab-valued
  profunctors.  For simplicity, we also consider first the case of
  groups (one-object groupoids), in which case (as we remarked in
  \autoref{eg:monfib-prof}) profunctors can be identified with bimodules
  between group rings.  Thus, an object $M\in \Ab^G$ is totally
  dualizable just when it is right dualizable as a
  $\bbZ$-$\bbZ[G]$-bimodule, which is equivalent to its being finitely
  generated and projective as a $\bbZ[G]$-module.

  In particular, $\bbZ[G]$ itself is always totally dualizable,
  regardless of how large $G$ might be.  On the other hand, the trivial
  module $\bbZ$, which is the unit object of $\Ab^G$, is quite rarely
  totally dualizable, since it is quite rarely projective.  (But if we
  took $\bC=\mathbf{Vect}_{\mathbb{Q}}$ to be the category of rational
  vector spaces instead of \Ab, then by Maschke's theorem, the unit
  object $\mathbb{Q}$ over any finite group $G$ would be totally
  dualizable~\cite[4.2.1, 4.2.2]{weibel:homological}.)

  Just as shadows of profunctors in this case can be identified with
  shadows of bimodules, the resulting traces for totally dualizable
  $G$-modules can be identified with the Hattori-Stallings trace (the
  trace in the bicategory of bimodules; see \autoref{eg:bimodules}
  and~\cite{PS2}).  In this case, \autoref{thm:totaltr1} says that if
  $f\colon M\to M$ is an endomorphism with Hattori-Stallings trace
  $\tr(\widecheck{f}) \in \sh{\bbZ[G]}$, and $\epsilon\colon \bbZ[G] \to
  \bbZ$ is the canonical augmentation of the group ring, then
  \[ \epsilon(\tr(\widecheck{f})) = \tr(\epsilon_!(f)) \] (where
  $\epsilon_!$ denotes the ``extension of scalars'' functor along
  $\epsilon$).

  The case of arbitrary groupoids can be reduced to a combination of the
  case of groups, as just described, and the case of sets, as in
  \autoref{eg:total-duality-mat}.  Namely, for a groupoid $A$, an object
  $M\in \Ab^A$ consists essentially of one module over each $\bbZ[G]$,
  as $G$ runs through the isotropy groups of the connected components of
  $A$.  Such an $M$ is totally dualizable just when it is zero at all
  but finitely many connected components and totally dualizable at the
  others.  As we have seen, the shadow of a groupoid is the direct sum
  (coproduct) of the shadows of the group rings of each of its isotropy
  groups:
  \[ \sh{A} = \bigoplus_{x\in \pi_0(A)} \bbZ[\mathrm{Conj}(\hom_A(x,x))]
  \]
  Finally, the trace of an endomorphism of a totally dualizable $M$ is
  the direct sum of its traces at each connected component:
  \[ \tr(f) = \sum_{x\in\pi_0(A)} \Big(\tr(f_x)
  \in\bbZ[\mathrm{Conj}(\hom_A(x,x))]\Big)  \quad\in\sh{A}
  \]
  and \autoref{thm:totaltr1} says that the sum in \bbZ\ of the
  augmentations of these traces is equal to the trace of the sum of the
  augmentations of $f$ at each connected component:
  \[ \sum_{x\in\pi_0(A)} \epsilon\big(\tr(f_x)\big) =
  \tr\left(\sum_{x\in\pi_0(A)} \epsilon_!(f)\right) \quad\in\bbZ.\]
\end{eg}

\begin{eg}\label{eg:total-duality-hoprof}
  The $\mathbf{Gpd}$-indexed category of chain complexes up to homotopy
  is similar to that of $\mathbf{Gpd}$-indexed abelian groups, but with
  some important differences.  The principal one is that now a chain
  complex is dualizable if it is \emph{quasi-isomorphic} to a finitely
  generated complex of projective modules.  This makes for many more
  dualizable objects.

  For instance, the unit object over a group $G$ is totally dualizable
  just when $\bbZ$ admits a finite projective resolution as a
  $\bbZ[G]$-module.  In this case the group $G$ is said to be of
  \emph{type FP}; see e.g.~\cite[VIII.6]{brown:cog}.  Groups of type FP
  include finitely generated free groups, surface groups, and finitely
  generated free abelian groups.  All groups of type FP are torsion-free.

  We can analyze the case of a general groupoid as we did in the
  previous example, by decomposing a profunctor into a direct sum of
  modules over group rings.  However, for the purposes of computing an
  explicit example, it is instructive to proceed differently.
  Given any groupoid $A$ and an object $x\in A$, we write $\bbZ[A_x]$
  and $\bbZ[{}_xA]$ for the profunctors defined by
  \[ \bbZ[A_x](a) = \bbZ[\hom_A(x,a)] \qquad\text{and}\qquad
  \bbZ[{}_xA](a) = \bbZ[\hom_A(a,x)].
  \]
  These are the base change objects, as in \S\ref{sec:base-change},
  associated to the functor $x\colon \star \to A$ which picks out the
  object $x$, relative to the $\mathbf{Gpd}$-indexed symmetric monoidal
  category of families of abelian groups.  (The corresponding base
  change objects for chain complexes are the same profunctors, regarded
  as concentrated in degree 0.)

  In particular, therefore, $\bbZ[{}_xA]$ is right dualizable, with
  right dual $\bbZ[{A}_x]$.  We regard $\bbZ[{}_xA]$ as the counterpart
  for groupoids of a rank-one free module.  The main difference from the
  case of rings is that now, there may be more than one isomorphism
  class of such modules, if $A$ has more than one isomorphism class of
  objects.  Just as a map $R\to M$ from a rank-one free $R$-module to
  any $R$-module $M$ is determined uniquely by an element of $M$ (the
  image of $1\in R$), a map $\bbZ[{}_xA] \to M$ of $A$-modules is
  determined uniquely by an element of $M(x)$ (the image of $\id_x \in
  \hom_A(x,x) \subseteq \bbZ[{}_x A](x)$).  (This is essentially just
  the Yoneda Lemma.)

  Now suppose that $A$ is a \emph{finitely generated free groupoid}, as
  in \autoref{eg:gpd-smtrace}, with object set $A_0$ and generating set
  of morphisms $A_1$, with source- and target-assigning maps $s,t\colon
  A_1 \rightrightarrows A_0$.  Let $M\in \Ab^A$ be an $A$-indexed
  diagram of abelian groups such that each abelian group $M(x)$ is
  finitely generated and free.  Write $M_x$ for a basis of $M(x)$; then
  the action of any generator $\gamma\in A_1$ can be described by a
  ``matrix'' $\gamma_{M,p,q}$, with coefficients defined by
  \[ \gamma(p) = \sum_{q\in M_{t(\gamma)}} \gamma_{M,p,q} \cdot q
  \]
  for $p\in M_{s(\gamma)}$.  Then $M$ is quasi-isomorphic to the
  following chain complex of $A$-modules concentrated in degrees 1 and
  0:
  \begin{equation}\label{eq:module-2termres}
    \xymatrix{
      \bigoplus_{\gamma \in A_1} \bigoplus_{p\in M_{s(\gamma)}}
      \bbZ[{}_{t(\gamma)} A] \ar[r]^-{d} &
      \bigoplus_{x\in A_0} \bigoplus_{p\in M_x} \bbZ[{}_x A].}
  \end{equation}
  Here the differential sends the generator $(\gamma,p,\id_{t(\gamma)})$
  of the $(\gamma,p)$ summand to the difference
  \[ (s(\gamma), p, \gamma) -
  \sum_{q\in M_{t(\gamma)}} \gamma_{M,p,q} \cdot (t(\gamma), q, \id_{t(\gamma)}).
  \]
  In the case when $A_0$ has one object, so that $A$ is a finitely
  generated free group, and $M$ is the trivial module $\bbZ$, then the
  degree-0 part of this complex is just the group ring $\bbZ[G]$, and
  the degree-1 part is the \emph{augmentation ideal}.  In that case, at
  least, the fact that the augmentation ideal is also a finitely
  generated free module is well-known~\cite[6.1.5]{weibel:homological}.

  Since this complex consists of finite sums of shifts of the dualizable
  objects $\bbZ[{}_x A]$, it is dualizable.  Therefore, $M$ is totally
  dualizable in the derived context of chain complexes, when regarded as
  a chain complex concentrated in degree 0.  Its dual is the complex
  concentrated in degrees 0 and $-1$:
  \[\xymatrix{
  \bigoplus_{x\in A_0} \bigoplus_{p\in M_x} \bbZ[A_x] \ar[r]^-{\partial} &
  \bigoplus_{\gamma \in A_1} \bigoplus_{p\in M_{s(\gamma)}} \bbZ[A_{t(\gamma)}]}
  \]
  (we leave the computation of the differential to the reader).

  Now, if $f\colon M\to M$ is an endomorphism, with matrix elements
  $f_{p,q}$ defined by $f(p) = \sum_q f_{p,q} \cdot q$, then a
  corresponding endomorphism of~\eqref{eq:module-2termres} can be
  defined by
  \begin{align*}
    f(x,p,\id_x) &= \sum_q f_{p,q} \cdot (x,q,\id_x)\\
    f(\gamma,p,\id_{t(\gamma)}) &= \sum_q f_{p,q} \cdot (\gamma,q, \id_{t(\gamma)}).
  \end{align*}
  Therefore, noting that $\sum_{p\in M_x} f_{p,p}$ is the trace
  $\tr(f_x)$ of $f$ restricted to $M(x)$, we can calculate
  \[\tr(f) = \sum_{x\in A_0} \tr(f_x) \cdot [\id_x] -
  \sum_{\gamma\in A_1} \tr(f_{t(\gamma)}) \cdot [\id_{t(\gamma)}]
  \]
  where each $[\id_x]$ denotes the image of $\id_x$ in $\sh{A}$.  (As in
  \autoref{eg:gpd-smtrace}, the minus sign arises from the symmetry
  isomorphism for the tensor product of chain complexes.)
  
  Note that this result is invariant under equivalence of groupoids $A$,
  as it must be.  In fact, it can be rephrased as a sum over the
  connected components of $A$:
  \[ \tr(f) = \sum_{x\in \pi_0(A)} (1 - D_{A,x}) \tr(f_x)\cdot [\id_x]
  \]
  where $D_{A,x}$ denotes the number of generators of the isotropy group
  of $A$ at $x$ (which is finitely generated and free).

  Finally, since the augmentation $\sh{A} \to \bbZ$ sends each $[\id_x]$
  to $1$, \autoref{thm:totaltr1} says that the trace of $(\pi_A)_!(f)$
  must be
  \[ \sum_{x\in A_0} \tr(f_x) - \sum_{\gamma\in A_1} \tr(f_{t(\gamma)})
  = \sum_{x\in \pi_0(A)} (1 - D_{A,x}) \tr(f_x).
  \]
  This is easy to verify directly, using the fact that since the
  homotopy colimit $(\pi_A)_!(M)$ can be calculated using the projective
  resolution~\eqref{eq:module-2termres} to be
  \[\xymatrix{\bigoplus_{\gamma \in A_1} \bigoplus_{p\in M_{s(\gamma)}} \bbZ \ar[r]^-{d} &
  \bigoplus_{x\in A_0} \bigoplus_{p\in M_x} \bbZ.}
  \]
  Note that if $M$ is the constant functor at \bbZ, then this is exactly
  the chain complex representing $\Sigma(A)$ which we used in Examples
  \ref{eg:gpd-smtrace} and \ref{eg:gpd-smtransfer}.
\end{eg}

\begin{eg}\label{thm:total-cw}
  Total duality was first studied in~\cite{maysig:pht} for the derived
  bicategory of parametrized spectra, where it was called
  \emph{Costenoble-Waner duality}; see also \cite{PS2}.  In this
  bicategory, \autoref{thm:total-duality} implies that if the
  parametrized sphere spectrum $S_A$ over $A$ is Costenoble-Waner
  dualizable, then $\Sigma^\infty(A_+)$ is dualizable.  This is the
  original observation of~\cite{maysig:pht} that
  \autoref{thm:total-duality} generalizes.

  Concrete examples of Costenoble-Waner duality can also be found
  in~\cite{maysig:pht}.  For instance, the parametrized sphere spectrum
  $S_A$ over any closed smooth manifold $A$ is Costenoble-Waner
  dualizable~\cite[18.6.1]{maysig:pht}.

  For simplicity, we consider only the trace of its identity map.  The
  trace of the identity of $\Sigma^\infty(A_+)$ can be identified with
  the \emph{Euler characteristic} of $A$ (as an endomorphism of the
  sphere spectrum $S$).  Thus, the theorem says that this Euler
  characteristic factors as
  \[S \rightarrow \Sigma^\infty (L A_+)\rightarrow S.\]
  The first map turns out to be the composite
  \[ S \to \Sigma^\infty(A_+) \to \Sigma^\infty (LA_+)
  \]
  of the transfer of the identity map with the inclusion $A\rightarrow
  LA$ as constant paths.
\end{eg}

While \autoref{thm:totaltr1} is a direct analogue of
\autoref{thm:fibtrace4}, often we are interested in a different
situation.  As mentioned in the introduction, usually we are given a
morphism in the cartesian monoidal base category \bS, and we want to
compute a trace which gives us information about its fixed points.

Our final theorem, which solves this problem, will be a version of
\autoref{thm:totaltr1} which compares twisted traces.  Combined with the
base change objects from \S\ref{sec:base-change}, this will enable us to
compare symmetric monoidal and total duality traces involving an
endomorphism of the base object.

\begin{thm}\label{thm:totaltr2}
  Suppose $M\in\sC^A$ is totally dualizable, $Q\in \sC^\star$, $P\in
  \sC^{A\times A}$, and we have a morphism $f\colon Q\odot
  \widecheck{M}\rightarrow \widecheck{M}\odot P$.  Suppose furthermore
  that we are given $R\in \sC^\star$ and a morphism
  \[\xi\colon P \odot \widehat{I_A} =
  (\id\times\pi_A)_! P \too (\pi_A)^* R = \widehat{I_A} \odot R.\]
  Then the following triangle commutes:
  \begin{equation}
    \xymatrix{ Q\ar[r]^-{\tr(\widecheck{f})}
      \ar[dr]_-{\tr(\xi\circ f)}&\sh{P}\ar[d]^{\tr(\xi)}\\
      & R
    }\label{eq:totaltr2}
  \end{equation}
  Here $\xi\circ f$ denotes the following composite in $\sC^\star$:
  \[
  Q \otimes (\pi_A)_! M \cong
  Q\odot \widecheck{M} \odot \widehat{I_A} \xto{f\odot \id}
  \widecheck{M} \odot P \odot \widehat{I_A} \xto{\id\odot \xi}
  \widecheck{M} \odot \widehat{I_A} \odot R \cong
  (\pi_A)_! M \otimes R
  \]
\end{thm}
\begin{proof}
  Continuing from the proof of \autoref{thm:total-duality}, we observe
  that the equivalence $\calBi CS(\star,\star)\eqv \sC^\star$ identifies
  the shadow of $\calBi CS$ (restricted to $\calBi CS(\star,\star)$)
  with the identity functor of $\sC^\star$.  Thus, this equivalence
  respects traces as well, and so the symmetric monoidal trace of
  $\xi\circ f$ can be identified with its bicategorical trace
  (considered as a morphism in $\calBi CS(\star,\star)$).  But with this
  identification,~\eqref{eq:totaltr2} simply becomes an instance of the
  functoriality of bicategorical trace~\cite[Prop.~7.5]{PS2}.
\end{proof}

As in the case of untwisted traces, we see that for total duality,
bicategorical traces carry more information than the induced symmetric
monoidal ones.

Probably the most important case of \autoref{thm:totaltr2} is the
following.  Let $\phi\colon A\to A$ be an endomorphism of $A\in \bS$,
and let $f$ be a morphism $M \to \phi^* M$.  We think of such an $f$ as
a ``$\phi$-equivariant endomorphism of $M$''.  Moreover, by
\autoref{thm:bco-action} we can equivalently regard $f$ as a morphism
$\widecheck{M} \to \widecheck{M}\odot A_\phi$; thus its bicategorical
trace will be a map $I_\star \to \sh{A_\phi}$ in $\sC^\star$.

We now want to apply \autoref{thm:totaltr2} with $Q = R = I_\star$ and
$P=A_\phi$.  We choose $\xi$ to be the morphism
\[ A_\phi \odot \widehat{I_A} \xto{\cong}
\widehat{\phi_! I_A} \xto{\cong}
\widehat{\phi_! \phi^* I_A} \too \widehat{I_A}.
\]
The trace of this $\xi$ is an augmentation $\sh{A_\phi} \to I_\star$,
which reduces to the augmentation of $\sh{A}$ if $\phi$ is the identity.
\autoref{thm:totaltr2} then implies that $\tr(\widecheck{f})$ factors
$\tr(\xi\circ f)$ via this augmentation.

Specializing even further, suppose that $M = I_A$ is totally dualizable;
in this case we may say that $A$ itself is ``totally dualizable''.  Then
for any $\phi$, we can choose $f$ to be the isomorphism
\begin{equation}\label{eq:f-over-f}
  I_A \cong \phi^* I_A.
\end{equation}
The bicategorical trace of this $f$ is a morphism $I_\star \to
\sh{A_\phi}$, which we call the \textbf{total-duality trace of $\phi$}.
In \autoref{prop:xi_circ_f}, we will prove that for this $f$ and the
above $\xi$, we have $\xi \circ f = \Sigma (\phi)$.  Thus, we can deduce
\autoref{thm:intro-total}.

\begin{cor}\label{cor:totaltr3}
  For any totally dualizable $A\in\bS$ and any $\phi\colon A \to A$, the
  total-duality trace of $\phi$ is a morphism $I_\star \to \sh{A_\phi}$
  which factors the symmetric monoidal trace $\tr (\Sigma(\phi))\colon
  I_\star \to I_\star$ via an augmentation $\sh{A_\phi} \to I_\star$.
\end{cor}

In order to see what this means in examples, we need a way to compute
the object $\sh{A_\phi}$.  We can do this with a generalization of the
technique used for $\sh{A}$ in \S\ref{sec:shadows} involving free loop
spaces.  Namely, suppose we have a homotopy pullback
\[ \xymatrix{ L_\phi A \ar[r]^-p \ar[d]_q &
  A \ar[d]^{(\id,\phi)} \\
  A \ar[r]_-{\Delta} & A\times A.}
\]
Then just as before, we have
\[\sh{A_\phi} = (\pi_A)_! (\Delta_A)^* (\id_A \times \phi)^* (\Delta_A)_! I_A \cong
(\pi_A)_! p_! q^* I_A \cong
(\pi_{L_\phi A})_! I_{L_\phi A} =
\Sigma(L_\phi A).
\]
The actual pullback of $\Delta$ and $(\id,\phi)$ is the equalizer of
$\phi$ and the identity.  Thus, in non-derived situations, $\sh{A_\phi}$
is $\Sigma$ applied to this equalizer.

\begin{eg}\label{thm:families-totaltr}
  Let $A$ be a finite set, so that the unit object $I_A$ in $\Ab^A$ is
  totally dualizable, and let $\phi\colon A\to A$ be an endomorphism.
  Then the equalizer of $\phi$ and $\id$ is the set $\mathrm{Fix}(\phi)$
  of fixed points of $\phi$, so that $\sh{A_\phi} =
  \bigoplus_{\phi(a)=a} \bbZ = \bbZ[\mathrm{Fix}(\phi)]$ is the free
  abelian group on this set.  The augmentation $\sh{A_\phi}\to\bbZ$
  takes each basis element to 1.  The total-duality trace of $\phi$ can
  then be computed as the homomorphism $\mathbb{Z} \to \sh{A_\phi}$
  which sends $1$ to $\sum_{\phi(a)=a} a$.  Thus, its image under the
  augmentation is the number of fixed points of $\phi$, which is exactly
  the trace of $\xi\circ f = \Sigma(\phi)$.
\end{eg}

In derived situations, however, $L_\phi A$ will generally be a ``twisted
free loop space'' whose points are ``paths'' from a point $x$ to
$\phi(x)$.

\begin{eg}\label{eg:gpd-totaltrace}
  For a groupoid $A$ with an endofunctor $\phi$, the pseudo-pullback of
  $(\id,\phi)$ and $\Delta$ is the groupoid $L_\phi A$ whose points are
  pairs $(x,\gamma)$ where $x\in A$ is an object and $\gamma\in
  \hom_A(x,\phi(x))$ is a morphism.  A morphism in $L_\phi A$ from
  $(x,\gamma)$ to $(y,\delta)$ is a morphism $\alpha\colon x\to y$ in
  $A$ such that an evident square commutes.  Then $\sh{A_\phi}$ is the
  colimit or homotopy quotient (depending on whether we are in the
  underived or derived setting) of the constant diagram on $L_\phi A$.

  In the underived case of \Ab-valued profunctors, this implies that
  $\sh{A_\phi}$ is the free abelian group on the set of connected
  components of $L_\phi A$.  Using terminology from fixed-point theory,
  we may call the connected components of $L_\phi A$ the
  \emph{fixed-point classes} of $\phi$.

  Similarly, in the derived case of $\bCh{\bbZ}$-valued profunctors,
  $\sh{A_\phi}$ is the complex of chains on the nerve of $L_\phi A$.  By
  a straightforward generalization of \autoref{eg:hochschild}, we can
  identify the nerve of $L_\phi A$ with the ``$\phi$-twisted cyclic
  nerve'' of $A$.  When $A$ is a group $G$, the chains on its twisted
  cyclic nerve form the Hochschild complex of $\bbZ[G]_\phi$ as a
  bimodule over $\bbZ[G]$.  As before, in the general case we obtain a
  direct sum over isotropy groups.  Also as before, the
  derived-underived difference is not very important here, since the 0th
  homology of the derived version gives us back the underived version.

  However, passing to the derived version \emph{is} important in order
  to have nontrivial totally dualizable objects.  As we saw in
  Examples~\ref{eg:total-duality-prof}
  and~\ref{eg:total-duality-hoprof}, the unit object is rarely totally
  dualizable as an \Ab-valued profunctor, but as a $\bCh{\bbZ}$-valued
  profunctor in the derived case, it can sometimes be---for instance,
  when our groupoid is a group of type FP.

  For a concrete example, let $A$ be a finitely generated free groupoid,
  and let $\phi\colon A\to A$ be an endofunctor.  As we saw in
  \autoref{eg:total-duality-hoprof}, the unit object $I_A$ (i.e.\ the
  constant functor at \bbZ) is equivalent to the 2-term chain complex
  \[\xymatrix{
  \bigoplus_{\gamma \in A_1} \bbZ[{}_{t(\gamma)} A] \ar[r]^d &
  \bigoplus_{x\in A_0} \bbZ[{}_x A].}
  \]
  Given $\phi\colon A\to A$, the twisted target $\phi^* I_A$
  of~\eqref{eq:f-over-f} is represented by the same complex, but with
  the action of $A$ twisted by $\phi$.  We can write this as
  \[\xymatrix{
  \bigoplus_{\gamma \in A_1} \bbZ[{}_{t(\gamma)} A_\phi] \ar[r]^d &
  \bigoplus_{x\in A_0} \bbZ[{}_x A_\phi].}
  \]
  where for any object $x$, the profunctor $\bbZ[{}_x A_\phi]$ is defined by
  \[ \bbZ[{}_x A_\phi] (y) = \bbZ[\hom_A(\phi(y),x)].\]
  Now the morphism~\eqref{eq:f-over-f} can be represented by the chain map
  \[\xymatrix{
  \bigoplus_{\gamma \in A_1} \bbZ[{}_{t(\gamma)} A]
  \ar[r]^d \ar[d] &
  \bigoplus_{x\in A_0} \bbZ[{}_x A] \ar[d]\\
  \bigoplus_{\gamma \in A_1} \bbZ[{}_{t(\gamma)} A_\phi]
  \ar[r]_d &
  \bigoplus_{x\in A_0} \bbZ[{}_x A_\phi]}
  \]
  defined as follows:
  \begin{itemize}
  \item Each degree-0 generator $(x,\id_x)$ goes to $(\phi(x),
    \id_{\phi(x)})$.
  \item Suppose that $\gamma$ is a generating morphism such that
    $\phi(\gamma)$ is written in terms of generators as
    \[ \phi(\gamma) = \alpha_n^{\epsilon_n} \cdots \alpha_1^{\epsilon_1},\]
    with each $\epsilon_i \in \{+1, -1\}$.
    Then the degree-1 generator $(\gamma,\id_{t(\gamma)})$ goes to the sum
    \[ \sum_{i=1}^n \epsilon_i \cdot \big(\alpha_i,\, \beta_i \big)
    \]
    where $\beta_i\in \hom_A(t(\gamma), t(\alpha_i))$ is defined by
    \[ \beta_i =
    \begin{cases}
      \alpha_{i+1}^{-\epsilon_{i+1}} \cdots \alpha_{n}^{-\epsilon_{n}}
      & \quad \epsilon_i = 1\\
      \alpha_i \alpha_{i+1}^{-\epsilon_{i+1}} \cdots \alpha_{n}^{-\epsilon_{n}}
      & \quad \epsilon_i = -1
    \end{cases}
    \]
  \end{itemize}
  By our calculation of $\sh{A_\phi}$, the total-duality trace of $\phi$
  will be a sum of conjugacy classes of morphisms $[\delta]$ such that
  $\delta\in \hom_A(\phi(x),x)$ for some $x$ (note $\delta$ is not
  necessarily a generator).  This can now be computed as a sum of the
  following terms:
  \begin{enumerate}
  \item For every object $x$, if $\phi(x) = x$, then we have a
    contribution of $[\id_x]$.
  \item For every generating morphism $\gamma \colon x\to y$, we have a
    contribution of
    \[
    - \sum_{\alpha_i = \gamma} \epsilon_i \cdot [\beta_i]
    \]
    with $\beta_i$ defined as above.
  \end{enumerate}
  The augmentation $\sh{A_\phi}\to \bbZ$ takes each generator $[\delta]$
  to 1, so applying it to this sum, we recover the trace as computed in
  \autoref{eg:gpd-smtrace}.
\end{eg}

\begin{eg}\label{eg:rtrace}
  Finally, this factorization is familiar in fixed point theory (the
  original context that we are working to generalize).  Let $B$ be a
  closed smooth manifold and $\phi\colon B\rightarrow B$ be an
  endomorphism.  Then working in the derived indexed symmetric monoidal
  category of parametrized spectra, we have computed that $\sh{B_\phi}$
  is the suspension spectrum of the twisted loop space $\Lambda^\phi
  B=\{\gamma\in B^I|\phi(\gamma(0))=\gamma(1)\}$.  The augmentation is
  $\Sigma^\infty_+$ applied to the map
  \[\Lambda^\phi B\rightarrow \ast.
  \] 
  As remarked in \autoref{thm:total-cw}, when $B$ is a closed smooth
  manifold, its parametrized sphere spectrum $S_B$ is totally
  dualizable.  As usual, $\phi$ then induces a fiberwise morphism
  \[f\colon S_B\rightarrow S_B\odot B_\phi,
  \]
  whose bicategorical trace is a map $S\rightarrow
  \Sigma^\infty((\Lambda^\phi B)_+)$, hence an element of the 0th stable
  homotopy group of $(\Lambda^\phi B)_+$.  However, this group is
  canonically isomorphic to $\mathbb{Z}[\pi_0(\Lambda^\phi B)]$, and
  $\pi_0(\Lambda^\phi B)$ is the set of \emph{fixed-point classes} of
  the map $\phi$.  Thus, the trace of $\phi$ is equivalently a formal
  integral combination of fixed-point classes.  Under this
  interpretation, it can be identified with the original definition of
  the \emph{Reidemeister trace} from \cite{b:lefschetz, j:nielsen},
  which assigns to each fixed-point class the sum of the indices of the
  fixed points in that class.

  Acting on 0th stable homotopy, the augmentation takes each fixed-point
  class in $\pi_0(\Lambda^\phi B)$ to $1$.  Therefore,
  \autoref{cor:totaltr3} reduces to the obvious fact that if we add up
  the Reidemeister coefficients over all fixed-point classes, we obtain
  the sum of the indices of \emph{all} fixed points, i.e.\ the trace of
  $\Sigma^\infty(\phi_+)$.
\end{eg}

Returning to the general situation, it is also natural to ask how the
total-duality trace of $\phi$ is related to the \emph{transfer} of
$\Sigma(\phi)$, as defined in \S\ref{sec:smc-traces}.  We can compare
these by applying \autoref{thm:totaltr2} with a different choice of $R$
and $\xi$.  We take $R=\Sigma(A)$, of course, and we take $\xi$ to be
the following composite:
\[ A_\phi \odot \widehat{I_A}
\xto{\cong} \phi_! I_A
\xto{\cong} \phi_! \phi^* I_A
\too I_A
\too (\pi_A)^* (\pi_A)_! I_A
\xto{\cong} \widehat{I_A} \odot \Sigma(A)
\]
which we denote by $\zeta$.  (Note that the previous $\xi$ is a factor
of this $\zeta$.)

The trace of $\zeta$ is a morphism $\tr(\zeta)\colon \sh{A_\phi} \to
\Sigma(A)$, which factors the augmentation $\sh{A_\phi}\to I_\star$
through the augmentation $\Sigma(A)\to I_\star$ (this follows from a
triangle identity for the adjunction $(\pi_A)_! \dashv (\pi_A)^*$).  If
$\phi$ is the identity, so that $\sh{A_\phi}=\sh{A}$, then $\tr(\zeta)$
is the retraction of the comparison map $\Sigma(A) \to \sh{A}$.  In
\autoref{thm:zeta-circ-f} we will prove that $\zeta\circ f =
\Sigma(\Delta \circ \phi)$.  Thus, we obtain \autoref{thm:intro-total2}.

\begin{cor}\label{thm:totaltr-transfer}
  For any totally dualizable $A\in\bS$ and any $\phi\colon A \to A$, the
  total-duality trace of $\phi$ factors the symmetric monoidal transfer
  $\tr (\Sigma(\Delta\circ \phi))\colon I_\star \to \Sigma(A)$ via the
  map $\tr(\zeta)\colon \sh{A_\phi} \to \Sigma(A)$.
\end{cor}

Thus, not only does the total-duality trace carry more information than
the symmetric monoidal trace, it also carries more information than the
symmetric monoidal transfer.

\begin{eg}
  As in \autoref{thm:families-totaltr}, let $\phi\colon A\to A$ be an
  endomorphism of a finite set.  Then $\Sigma(A) \cong \sh{A}$ is the
  free abelian group with basis $A$, and the morphism $\sh{A_\phi} \to
  \Sigma(A)$ maps the basis elements of $\sh{A_\phi}$ (the fixed points
  of $\phi$) to themselves.  Thus, the total-duality trace
  $\sum_{\phi(a)=a} a \in \sh{A_\phi}$ maps to the same sum
  $\sum_{\phi(a)=a} a \in \Sigma(A)$, which is the transfer $\tr
  (\Delta\circ\phi)$.
\end{eg}

\begin{eg}
  As in \autoref{eg:gpd-totaltrace}, let $\phi\colon A\to A$ be an
  endomorphism of a finitely generated free groupoid, and recall from
  \autoref{eg:gpd-smtransfer} that $\Sigma(A)$ is the free abelian group
  on the connected components of $A$.  The morphism $\sh{A_\phi} \to
  \Sigma(A)$ takes each conjugacy class $[\delta]$ to its connected
  component, and we see that the total-duality trace from
  \autoref{eg:gpd-totaltrace} maps to the transfer from
  \autoref{eg:gpd-smtransfer}.
\end{eg}

\begin{eg}
  As noted before, the transfer of an endomorphism $\phi\colon
  M\rightarrow M$ of a closed smooth manifold is an element of
  $\bbZ[\pi_0(M)]$.  The induced morphism $\pi_0(\Lambda^\phi
  M)\rightarrow \pi_0(M)$ associates a homotopy class of paths to the
  component containing it.  Thus, the image in $\bbZ[\pi_0(M)]$ of the
  Reidemeister trace has the coefficient of a component in $\pi_0(M)$
  being the sum of the fixed point indices of the fixed point classes
  lying in that component.
\end{eg}

\begin{rmk}
  The idea in the proof of \autoref{thm:totaltr2} can be applied to more
  general traces in the bicategory $\calBi CS$.  Namely, suppose $M\in
  \sC^{A\times B}$ is right dualizable, when regarded as a 1-cell $A\hto
  B$ in $\calBi CS$.  Let $Q\in \sC^{A\times A}$ and $P\in\sC^{B\times
    B}$, and let $f\colon Q\odot M \to M\odot P$ be a morphism, which
  therefore has a trace $\tr(f)\colon \sh{Q}\to\sh{P}$.  Suppose
  furthermore that as in \autoref{thm:totaltr2}, we have an
  $R\in\sC^\star$ and a morphism $\xi\colon P \odot \widehat{I_A} \to
  \widehat{I_A} \odot R$.  Then by composition of dual pairs, $M\odot
  \widehat{I_A} \cong \widehat{(\pi_B)_! M}$ is right dualizable, and by
  the functoriality of bicategorical trace, the following triangle
  commutes:
  \[ \xymatrix{ \sh{Q} \ar[r]^{\tr(f)} \ar[dr]_{\tr(\xi\circ f)} &
    \sh{P} \ar[d]^{\tr(\xi)} \\ & R. } \] However, by
  \autoref{thm:fiberwise-duality}, it follows that $(\pi_B)_! M$ is
  dualizable in the symmetric monoidal category $\sC^A$, and $\xi\circ
  f$ is a morphism
  \[Q\odot \widehat{(\pi_B)_! M} \to \widehat{(\pi_B)_! M}\odot R
  \]
  of the sort to which we can apply
  \autoref{thm:fibtrace}\ref{item:fibtrace1}.  Therefore, there is a
  morphism
  \[\overline{\xi\circ f}\colon (\Delta_A)^*Q \otimes_A (\pi_B)_! M
  \to (\pi_B)_! M \otimes_A (\pi_A)^* R
  \]
  whose symmetric monoidal trace in $\sC^A$ carries the same information
  as $\tr(\xi\circ f)$.

  Thus, although it may seem that we have only considered two very
  special cases of duality and trace in $\calBi CS$, the general case
  involves no new ideas, being essentially just a combination of these
  two.
\end{rmk}

This concludes our results on refinements of the symmetric monoidal
trace, and also the main part of the paper.  In the subsequent sections,
we introduce the string diagram calculus for indexed monoidal categories
and apply it to complete the postponed proofs from \S\ref{sec:fiberwise}
and \S\ref{sec:total}.

\section{String diagrams for objects}
\label{sec:string-diagrams}

Our string diagram calculus for indexed monoidal categories is inherited
from a similar calculus used by C. S. Peirce in his ``System Beta.''
This was given a categorical interpretation in terms of hyperdoctrines
(indexed monoidal posets such as Examples~\ref{eg:factsys}
and~\ref{eg:hyperdoctrine}) by~\cite{bt:peirce}.

We begin with the usual string diagram calculus for morphisms in the
base category \bS, drawn proceeding down the page, with the morphisms of
\bS\ contained in inverted triangles, as shown in
Figure~\ref{fig:cartmon}.  Since the diagonal and projection maps
$\Delta_A\colon A\to A\times A$ and $\pi_A\colon A \to \star$ play such
an essential role, to reduce clutter we represent them by empty
triangles, as in Figures~\ref{fig:diag} and~\ref{fig:proj}.

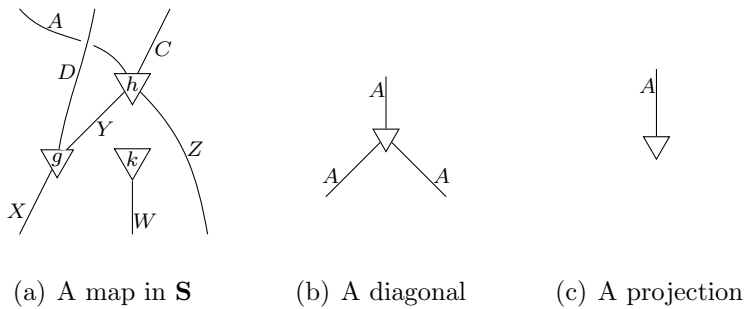
\begin{figure}
  \centering
  \subfigure[A map in \bS]{\label{fig:cartmon}
  \begin{tikzpicture}
    \node[pb] (g) at (1,1) {$g$};
    \node[pb] (f) at (2,2) {$h$};
    \draw (0.5,0) -- node[ed,near start] {$X$} (g)
    -- node[ed,swap] {$Y$} (f)
    to[out=120,in=-50] node[ed,swap,near end] {$A$} (.5,3);
    \draw (f) to[out=-45,in=100] node[ed] {$Z$} (3,0);
    \draw (f) -- node[ed,swap] {$C$} (2.5,3);
    \draw[white,line width=5pt] (g) to[out=80,in=-100] (1.5,3);
    \draw (g) to[out=80,in=-100] node[ed] {$D$} (1.5,3);
    \node[pb] (k) at (2,1) {$k$};
    \draw (k) -- node[ed,near end] {$W$} (2,0);
  \end{tikzpicture}}
\qquad
  \subfigure[A diagonal]{\label{fig:diag}
    \quad\raisebox{.5cm}{\begin{tikzpicture}[scale=.8]
    \node[pbe] (delta) at (1,1) {};
    \draw (delta) -- node[ed,swap,near end] {$A$} (0,0);
    \draw (delta) -- node[ed,near end] {$A$} (2,0);
    \draw (delta) -- node[ed,near end] {$A$} (1,2);
  \end{tikzpicture}}\quad}
\qquad
  \subfigure[A projection]{\label{fig:proj}
    \hspace{1cm}\raisebox{1cm}{\begin{tikzpicture}
      \node[pbe] (pi) at (1,0) {};
      \draw (pi) -- node [ed,near end] {$A$} ++(0,1);
    \end{tikzpicture}}\hspace{1cm}}
  \caption{String diagrams in a cartesian monoidal base category}
\end{figure}

To this string diagram calculus we now add a new type of vertex, which
we draw as a square box.  Such a vertex can only have strings coming in
the top, never out the bottom, and if the strings entering its top are
labeled by objects $A$, $B$, $C$ of \bS, then the box vertex must be
labeled by an object of the fiber category $\sC^{A\times B\times C}$.
Finally, we require that our diagrams have no strings coming out the
bottom either; all strings must end at a vertex of one type or the
other.  Thus we arrive at string diagrams such as in
Figure~\ref{fig:indmon-dgm}.

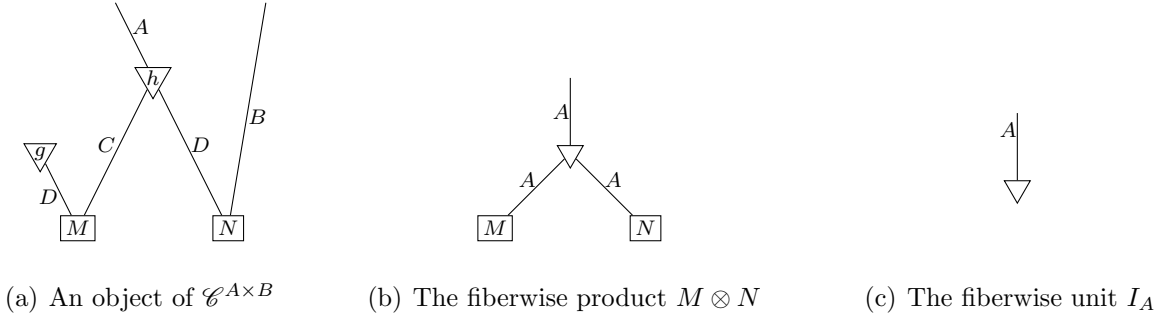
\begin{figure}
  \centering
  \subfigure[\mbox{An object of $\sC^{A\times B}$}]{\label{fig:indmon-dgm}
    \hspace{.3cm}\begin{tikzpicture}
      \node[fiber] (M) at (0,0) {$M$};
      \node[fiber] (N) at (2,0) {$N$};
      \node[pb] (g) at (-.5,1) {$g$};
      \node[pb] (h) at (1,2) {$h$};
      \draw (M) -- node[ed] {$D$} (g);
      \draw (M) -- node[ed] {$C$} (h) -- node[ed,swap] {$A$} ++(-.5,1);
      \draw (N) -- node[ed,swap] {$D$} (h);
      \draw (N) -- node[ed,swap] {$B$} ++(.5,3);
    \end{tikzpicture}\hspace{.3cm}}
  \qquad
  \subfigure[\mbox{The fiberwise product} \mbox{$M\otimes N$}]{\label{fig:intprod}
    \hspace{1.4cm}\begin{tikzpicture}
      \node[fiber] (M) at (0,0) {$M$};
      \node[fiber] (N) at (2,0) {$N$};
      \node[pbe] (delta) at (1,1) {};
      \draw (M) -- node[ed] {$A$} (delta) -- node[ed] {$A$} ++(0,1);
      \draw (N) -- node[ed,swap] {$A$} (delta);
    \end{tikzpicture}\hspace{1.4cm}}
  \qquad
  \subfigure[\mbox{The fiberwise unit $I_A$}]{\label{fig:intunit}
    \hspace{2cm}\raisebox{.5cm}{\begin{tikzpicture}
      \node[pbe] (pi) at (1,0) {};
      \draw (pi) -- node [ed,near end] {$A$} ++(0,1);
    \end{tikzpicture}}\hspace{2cm}}
  \caption{String diagrams in an indexed monoidal category}
  \label{fig:indmon}
\end{figure}

One way to define the \emph{value} of such a diagram is as follows:
first we take the external product of all the fiber objects appearing
(in box nodes), then we apply the reindexing functor $f^*$, where $f$ is
the composite of the \bS-portion of the diagram.  Thus, according to
this scheme, the diagram in Figure~\ref{fig:indmon-dgm} would have the
value
\[(g\times h\times \id_B)^*(M\boxtimes N).
\]
However, there are other natural ways to ``compose up'' the same
diagram, which for Figure~\ref{fig:indmon-dgm} give results such as
\begin{gather*}
  (h\times \id)^* (g\times \id)^*(M\boxtimes N) \mathrlap{\qquad\text{and}}\\
  (h\times \id)^*\Big(\big((g\times\id)^*M\big) \boxtimes N\Big).
\end{gather*}
The point of the string diagram notation is that all of these are
canonically isomorphic (using the coherence isomorphisms for the
monoidal structures and reindexing functors).  A proper proof of
validity for these string diagrams would make this precise, but we do
not have space to give such a proof here.  Thus, properly speaking our
string diagrams are only an informal guide to the necessary
calculations.

As a useful example, Figures~\ref{fig:intprod} and~\ref{fig:intunit}
show string diagrams for the expressions~\eqref{eq:intprod}
and~\eqref{eq:intunit}, giving the fiberwise monoidal product and unit
in terms of the external ones.

We now need a way to notate the adjoint functors $f_!$.  We do this by
introducing a third type of node, drawn with an upward-pointing
triangle, which is also labeled by a morphism of \bS\ but with the
codomain on top and the domain on the bottom.  For instance, the diagram
in Figure~\ref{fig:coprodeg} represents the object $(f_! \Delta^*
(M\boxtimes g_! N)) \boxtimes h_! P$.  The diagrams in
Figures~\ref{fig:bicatops} and \ref{fig:sigma} are examples we have used
earlier: Figure~\ref{fig:bicatops} is the bicategorical product, unit
and shadow from \autoref{thm:mf-bi}, while Figure~\ref{fig:sigma} is the
functor $\Sigma$ from \S\ref{sec:smc-traces}.
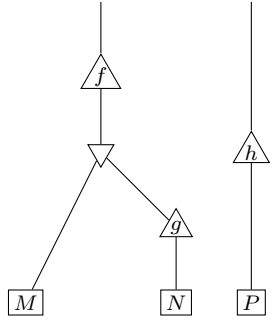
\begin{figure}
  \centering
  \begin{tikzpicture}
    \node[pfsm] (f) at (1,3) {$f$};
    \node[pbe] (d) at (1,2) {};
    \node[pf] (g) at (2,1) {$g$};
    \node[fiber] (m) at (0,0) {$M$};
    \node[fiber] (n) at (2,0) {$N$};
    \node[fiber] (p) at (3,0) {$P$};
    \node[pf] (h) at (3,2) {$h$};
    \draw (m) -- (d) -- (f) -- +(0,1);
    \draw (n) -- (g) -- (d);
    \draw (p) -- (h) -- (3,4);
  \end{tikzpicture}
  \caption{A string diagram involving indexed coproducts}
  \label{fig:coprodeg}
\end{figure}

\begin{figure}
  \centering
  \begin{tabular}{c@{\hspace{1cm}}c@{\hspace{1cm}}c}
    \subfigure[$M\odot N$]{\label{fig:bicatcomp}\begin{tikzpicture}
      \node[fiber] (M) at (1,0) {$M$};
      \node[fiber] (N) at (3,0) {$N$};
      \node[pbe] (del) at (2,1) {};
      \node[pfe] (pi) at (2,2) {};
      \draw (0,2) -- (M) -- (del) edge (pi) -- (N) -- ++(1,2);
    \end{tikzpicture}}
  &
  \subfigure[$U_A$]{\begin{tikzpicture}
    \node[pbe] (pi) at (0,0) {};
    \node[pfe] (delta) at (0,1) {};
    \draw (pi) -- (delta) edge (1,2) -- ++(-1,1);
  \end{tikzpicture}}
  &
  \subfigure[$\protect\sh{M}$]{\label{fig:bicatshadow}\quad\begin{tikzpicture}
      \node[fiber] (M) at (0,0) {$M$};
      \node[pbe] (delta) at (0,1) {};
      \node[pfe] (pi) at (0,2) {};
      \draw (M) to[out=45,in=-45] (delta);
      \draw (M) to[out=135,in=-135] (delta);
      \draw (delta) -- (pi);
  \end{tikzpicture}\quad}
  \end{tabular}
  \caption{The operations of the bicategory $\calBi CS$}
  \label{fig:bicatops}
\end{figure}
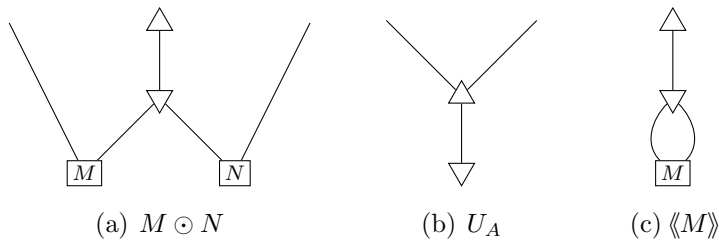

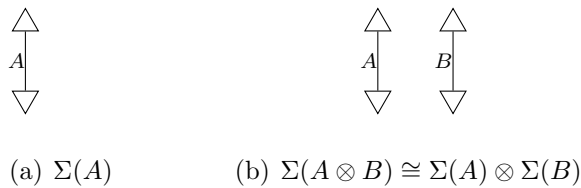
\begin{figure}
  \centering
  \subfigure[$\Sigma(A)$]{\hspace{1cm}\begin{tikzpicture}
    \node[pfe] (d1) at (0,1) {};
    \node[pbe] (d0) at (0,0) {};
    \draw (d0) -- node[ed,auto] {$A$} (d1);
  \end{tikzpicture}\hspace{2cm}}
  \subfigure[$\Sigma(A\otimes B) \cong \Sigma(A) \otimes \Sigma(B)$]{\hspace{2cm}
    \begin{tikzpicture}
    \node[pfe] (d1a) at (0,1) {};
    \node[pbe] (d0a) at (0,0) {};
    \draw (d0a) -- node[ed,auto] {$A$} (d1a);
    \node[pfe] (d1b) at (1,1) {};
    \node[pbe] (d0b) at (1,0) {};
    \draw (d0b) -- node[ed,auto] {$B$} (d1b);
  \end{tikzpicture}\hspace{2cm}}
  \caption{The functor $\Sigma$}
  \label{fig:sigma}
\end{figure}

Using this notation, the Beck-Chevalley conditions corresponding to the
four pullback diagrams in Figure~\ref{fig:pullbacks} on
page~\pageref{fig:pullbacks} are shown in Figure~\ref{fig:beckchev} as
isomorphisms between two (fragments of) string diagrams.  In all cases,
the natural morphism goes in the direction shown, and the content of the
Beck-Chevalley condition is that this map is an isomorphism.  The
assumption that squares satisfying the Beck-Chevalley condition are
closed under taking cartesian products with another fixed object implies
that these morphisms are still invertible when they occur as fragments
of larger diagrams.  The transposes of Figures~\ref{fig:frobpb}
and~\ref{fig:slidpb} are represented by the top-to-bottom reflections of
Figures~\ref{fig:frobbc} and~\ref{fig:slidbc}.
\begin{figure}
  \centering
  \begin{tabular}{c@{\hspace{1cm}}c}
    \subfigure[Commutativity with reindexing]{\label{fig:reindbc}
      \hspace{1cm}\begin{tikzpicture}[scale=.8]
      \node[pfsm] (f) at (0,2) {$f$};
      \node[pb] (g) at (1,1) {$g$};
      \draw (0,0) -- (f) -- (0,3);
      \draw (1,0) -- (g) -- (1,3);
      \node at (2,1.5) {$\xto{\cong}$};
      \begin{scope}[xshift=3cm]
        \node[pfsm] (f) at (0,1) {$f$};
        \node[pb] (g) at (1,2) {$g$};
        \draw (0,0) -- (f) -- (0,3);
        \draw (1,0) -- (g) -- (1,3);
      \end{scope}
    \end{tikzpicture}\hspace{1cm}}
  &
  \subfigure[The Frobenius axiom]{\label{fig:frobbc}
    \hspace{.5cm}\begin{tikzpicture}[scale=.8]
      \node[pfe] (del2) at (0,2) {};
      \node[pbe] (del1) at (0,1) {};
      \draw (-.5,0) -- (del1) -- (del2) -- (-.5,3);
      \draw (.5,0) -- (del1);
      \draw (del2) -- (.5,3);
      \node at (1.5,1.5) {$\xto{\cong}$};
      \node[pfe] (del3) at (3,1) {};
      \node[pbe] (del4) at (3.5,2) {};
      \draw (3,0) -- (del3) -- (del4) -- (3.5,3);
      \draw (4,0) -- (4,1) -- (del4);
      \draw (2.5,3) -- (2.5,2) -- (del3);
    \end{tikzpicture}\hspace{.5cm}}
  \\
  \subfigure[Sliding and splitting]{\label{fig:slidbc}
    \begin{tikzpicture}[scale=.6]
      \node[pbsm] (f1) at (0,1) {$f$};
      \node[pfe] (del1) at (0,2) {};
      \node[pfsm] (f2) at (1,3) {$f$};
      \draw (0,0) -- (f1) -- (del1) -- (f2) -- ++(0,1);
      \draw (del1) -- ++(-1,2);
      \node at (2,2) {$\xto{\cong}$};
      \node[pfe] (del2) at (4,1.5) {};
      \node[pbsm] (f3) at (3,3) {$f$};
      \draw (4,0) -- (del2) -- (f3) -- ++(0,1);
      \draw (del2) -- ++(1,2.5);
    \end{tikzpicture}}
  &
  \subfigure[Monic diagonals]{\label{fig:mdbc}\qquad
    \begin{tikzpicture}[scale=.8]
      \draw (0,0) -- (0,3);
      \node at (1,1.5) {$\xto{\cong}$};
      \node[pfe] (del1) at (2,1) {};
      \node[pbe] (del2) at (2,2) {};
      \draw (2,0) -- (del1) (del2) -- ++(0,1);
      \draw (del1) to[bend right] (del2);
      \draw (del1) to[bend left] (del2);
    \end{tikzpicture}\qquad}
  \end{tabular}
  \caption{Beck-Chevalley conditions}
  \label{fig:beckchev}
\end{figure}
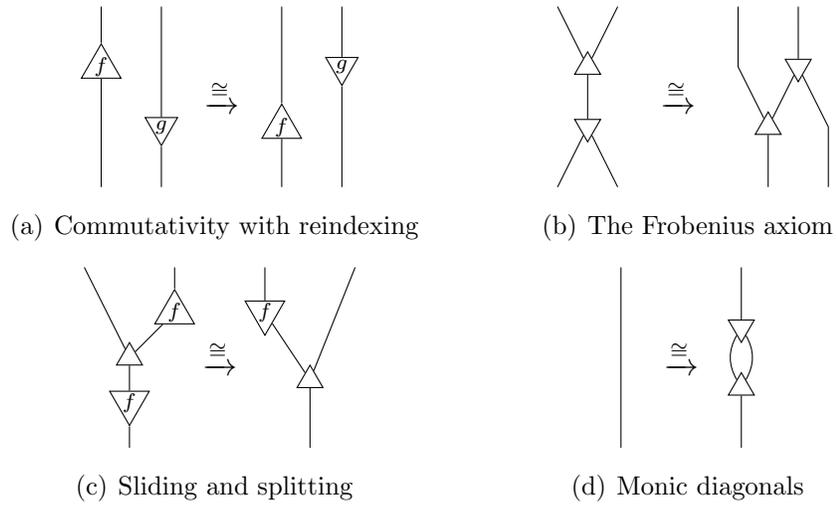

The isomorphism in Figure~\ref{fig:reindbc} is essential to
our ability to make deformation-invariant sense of string diagrams
involving $f_!$ nodes.  For this we also require that $\ten$ preserves
indexed coproducts, in order that the diagram in
Figure~\ref{fig:tenprescoprod} have an unambiguous meaning.

\begin{figure}
  \centering
  \begin{tikzpicture}
    \node[pfsm] (f) at (0,1) {$f$};
    \node[pf] (g) at (1,1) {$g$};
    \node[fiber] (M) at (0,0) {$M$};
    \node[fiber] (N) at (1,0) {$N$};
    \draw (M) -- (f) -- +(0,1);
    \draw (N) -- (g) -- +(0,1);
  \end{tikzpicture}
  \caption{$(f\times g)_!(M\boxtimes N) \cong f_! M \boxtimes g_! N$}
  \label{fig:tenprescoprod}
\end{figure}
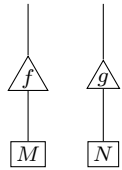

As examples of reasoning using these diagrams, the associativity and
unit isomorphisms of $\calBi CS$ are displayed graphically in
Figure~\ref{fig:bicat-constr}, while the shadow isomorphism $\sh{M\odot
  N}\cong \sh{N\odot M}$ is shown in Figure~\ref{fig:shadow}.  In
Figures~\ref{fig:bicatassoc} and~\ref{fig:shadow}, there are simple
deformations connecting the two sides (using the symmetry of \bS\ in the
case of the shadow), which involve implicit applications of the
``commutativity with reindexing'' Beck-Chevalley condition.  In
Figure~\ref{fig:bicatunit} we also need to use the ``Frobenius''
Beck-Chevalley condition.  The coherence of these isomorphisms would
follow from the ``general validity'' theorem for string diagrams which
we have omitted.

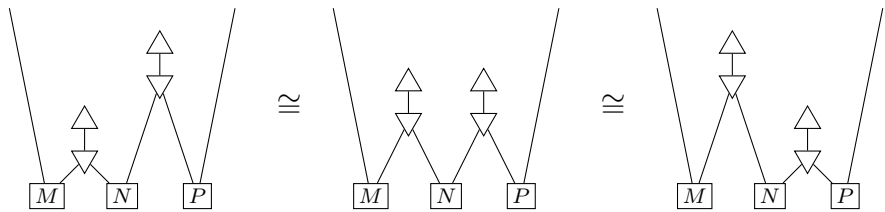
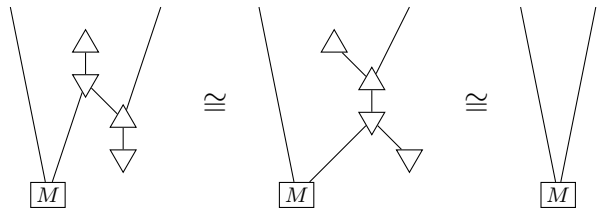
\begin{figure}
  \centering
  \subfigure[Associativity]{\label{fig:bicatassoc}
    \begin{tikzpicture}[scale=.5]
      \node[fiber] (m) at (1,0) {$M$};
      \node[fiber] (n) at (3,0) {$N$};
      \node[fiber] (p) at (5,0) {$P$};
      \node[pbe] (del1) at (2,1) {};
      \node[pbe] (del2) at (4,3) {};
      \node[pfe] (pi1) at (2,2) {};
      \node[pfe] (pi2) at (4,4) {};
      \draw (m) -- ++(-1,5);
      \draw (m) -- (del1) -- (n) -- (del2) -- (p) -- ++(1,5);
      \draw (del1) -- (pi1);
      \draw (del2) -- (pi2);
    \end{tikzpicture}
    \quad\raisebox{1.25cm}{$\cong$}\quad
    \begin{tikzpicture}[scale=.5]
      \node[fiber] (m) at (1,0) {$M$};
      \node[fiber] (n) at (3,0) {$N$};
      \node[fiber] (p) at (5,0) {$P$};
      \node[pbe] (del1) at (2,2) {};
      \node[pbe] (del2) at (4,2) {};
      \node[pfe] (pi1) at (2,3) {};
      \node[pfe] (pi2) at (4,3) {};
      \draw (m) -- ++(-1,5);
      \draw (m) -- (del1) -- (n) -- (del2) -- (p) -- ++(1,5);
      \draw (del1) -- (pi1);
      \draw (del2) -- (pi2);
    \end{tikzpicture}
    \quad\raisebox{1.25cm}{$\cong$}\quad
    \begin{tikzpicture}[scale=.5]
      \node[fiber] (m) at (1,0) {$M$};
      \node[fiber] (n) at (3,0) {$N$};
      \node[fiber] (p) at (5,0) {$P$};
      \node[pbe] (del1) at (2,3) {};
      \node[pbe] (del2) at (4,1) {};
      \node[pfe] (pi1) at (2,4) {};
      \node[pfe] (pi2) at (4,2) {};
      \draw (m) -- ++(-1,5);
      \draw (m) -- (del1) -- (n) -- (del2) -- (p) -- ++(1,5);
      \draw (del1) -- (pi1);
      \draw (del2) -- (pi2);
    \end{tikzpicture}}
  \\
  \subfigure[Unitality]{\label{fig:bicatunit}
    \begin{tikzpicture}[scale=.5]
      \node[fiber] (m) at (1,0) {$M$};
      \node[pbe] (del1) at (2,3) {};
      \node[pfe] (del2) at (3,2) {};
      \node[pfe] (pi1) at (2,4) {};
      \node[pbe] (pi2) at (3,1) {};
      \draw (m) -- ++(-1,5);
      \draw (m) -- (del1) -- (del2) -- (pi2);
      \draw (del2) -- ++(1,3);
      \draw (del1) -- (pi1);
    \end{tikzpicture}
    \quad\raisebox{1.25cm}{$\cong$}\quad
    \begin{tikzpicture}[scale=.5]
      \node[fiber] (m) at (1,0) {$M$};
      \node[pbe] (del1) at (3,2) {};
      \node[pfe] (del2) at (3,3) {};
      \node[pfe] (pi1) at (2,4) {};
      \node[pbe] (pi2) at (4,1) {};
      \draw (m) -- ++(-1,5);
      \draw (m) -- (del1) -- (del2) -- (pi1);
      \draw (del2) -- ++(1,2);
      \draw (del1) -- (pi2);
    \end{tikzpicture}
    \quad\raisebox{1.25cm}{$\cong$}\quad
    \begin{tikzpicture}[scale=.5]
      \node[fiber] (m) at (1,0) {$M$};
      \draw (m) -- ++(-1,5);
      \draw (m) -- ++(1,5);
    \end{tikzpicture}}
  \caption{The constraints of the bicategory $\calBi CS$}
  \label{fig:bicat-constr}
\end{figure}

\begin{figure}
  \centering
  \begin{tikzpicture}[scale=.7]
    \node[fiber] (m) at (1,0) {$M$};
    \node[fiber] (n) at (3,0) {$N$};
    \node[pbe] (del1) at (2,1) {};
    \node[pfe] (pi1) at (2,2) {};
    \node[pbe] (del2) at (2,3) {};
    \node[pfe] (pi2) at (2,4) {};
    \draw (m) -- node[ed] {$A$} (del1) -- node[ed] {$A$} (n);
    \draw (del1) -- (pi1);
    \draw (m) to[out=120,in=-150] node[ed] {$B$} (del2);
    \draw (n) to[out=60,in=-30] node[ed,swap] {$B$} (del2);
    \draw (del2) -- (pi2);
  \end{tikzpicture}
  \quad\raisebox{1.25cm}{$\cong$}\quad
  \begin{tikzpicture}[scale=.7]
    \node[fiber] (m) at (1,0) {$N$};
    \node[fiber] (n) at (3,0) {$M$};
    \node[pbe] (del1) at (2,1) {};
    \node[pfe] (pi1) at (2,2) {};
    \node[pbe] (del2) at (2,3) {};
    \node[pfe] (pi2) at (2,4) {};
    \draw (m) -- node[ed] {$B$} (del1) -- node[ed] {$B$} (n);
    \draw (del1) -- (pi1);
    \draw (m) to[out=120,in=-150] node[ed] {$A$} (del2);
    \draw (n) to[out=60,in=-30] node[ed,swap] {$A$} (del2);
    \draw (del2) -- (pi2);
  \end{tikzpicture}
  \caption{The shadow constraint of $\calBi CS$}
  \label{fig:shadow}
\end{figure}

Another important example is provided by base change objects.
Figure~\ref{fig:bco} shows the string diagram definition of a base
change object, while Figure~\ref{fig:bco-action} gives a graphical proof
of \autoref{thm:bco-action}, the interaction between base change object
and the bicategory composition.
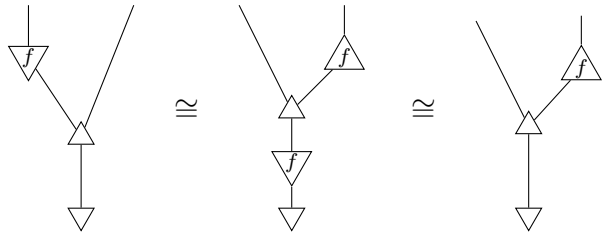
\begin{figure}
  \centering
  \begin{tikzpicture}[scale=.7]
    \node[pbe] (pi2) at (1,0) {};
    \node[pfe] (del2) at (1,1.5) {};
    \node[pbsm] (f3) at (0,3) {$f$};
    \draw (pi2) -- (del2) -- (f3) -- ++(0,1);
    \draw (del2) -- ++(1,2.5);
    \node at (3,2) {$\cong$};
    \begin{scope}[xshift=5cm]
    \node[pbe] (pi1) at (0,0) {};
    \node[pbsm] (f1) at (0,1) {$f$};
    \node[pfe] (del1) at (0,2) {};
    \node[pfsm] (f2) at (1,3) {$f$};
    \draw (pi1) -- (f1) -- (del1) -- (f2) -- ++(0,1);
    \draw (del1) -- ++(-1,2);
    \node at (2.5,2) {$\cong$};
  \end{scope}
  \begin{scope}[xshift=9.5cm]
    \node[pbe] (pi1) at (0,0) {};
    \node[pfe] (del1) at (0,1.7) {};
    \node[pfsm] (f2) at (1,2.8) {$f$};
    \draw (pi1) -- (del1) -- (f2) -- ++(0,1);
    \draw (del1) -- ++(-1,2);
  \end{scope}
  \end{tikzpicture}
  \caption{The base change object $_fB$}
  \label{fig:bco}
\end{figure}

\begin{figure}
  \centering
  \begin{tikzpicture}[scale=.8]
    \node[pbe] (pi2) at (1,.5) {};
    \node[pfe] (del2) at (1,1.5) {};
    \node[pbsm] (f3) at (0,2.5) {$f$};
    \draw (pi2) -- (del2) -- (f3) -- ++(0,2);
    \node[pbe] (del1) at (2,2.5) {};
    \draw (del2) -- (del1);
    \node[pfe] (pi1) at (2,3.5) {};
    \draw (del1) -- (pi1);
    \node[fiber] (m) at (3,.5) {$M$};
    \draw (del1) -- (m);
  \end{tikzpicture}
  \qquad\raisebox{1.5cm}{$\cong$}\qquad
  \begin{tikzpicture}[scale=.8]
    \node[pbe] (pi2) at (0,.5) {};
    \node[pbe] (del2) at (1,1.5) {};
    \node[pbsm] (f3) at (0,3.5) {$f$};
    \draw (pi2) -- (del2);
    \node[pfe] (del1) at (1,2.5) {};
    \draw (del2) -- (del1);
    \draw (del1) -- (f3) -- ++(0,1);
    \node[pfe] (pi1) at (2,3.5) {};
    \draw (del1) -- (pi1);
    \node[fiber] (m) at (2,.5) {$M$};
    \draw (del2) -- (m);
  \end{tikzpicture}
  \qquad\raisebox{1.5cm}{$\cong$}\qquad
  \begin{tikzpicture}[scale=.8]
    \node[fiber] (m) at (0,.5) {$M$};
    \node[pbsm] (f) at (0,2.5) {$f$};
    \draw (m) -- (f) -- (0,4.5);
  \end{tikzpicture}
  \caption{Proof of \autoref{thm:bco-action}}
  \label{fig:bco-action}
\end{figure}

\section{String diagrams for morphisms}
\label{sec:colored-strings}

There is still something important missing from our string diagrams:
morphisms in the fiber categories $\sC^A$.  (Here we go
beyond~\cite{bt:peirce}, since in their posetal setting there were no
morphisms to keep track of, only inequalities between objects.)  Since
we are representing the objects of $\sC^A$ by two-dimensional diagrams,
we need three dimensions for morphisms between them.

We begin by rotating our string diagrams from the previous section to
become horizontal slices, so that we can connect them with vertically
drawn strings.  For example, the string diagrams in
Figures~\ref{fig:string_diagram_switch1} and
\ref{fig:string_diagram_switch2} represent the same object.
Figure~\ref{fig:string_diagram_switch1} gives a representation as in the
previous section, while Figure~\ref{fig:string_diagram_switch2} is the
version we will use from now on.
\begin{figure}\label{fig:string_diagram_switch}
  \centering
	\subfigure[]{
  \begin{tikzpicture}[scale=.8]\label{fig:string_diagram_switch1}
    \node[pbe] (pi2) at (1,.5) {};
    \node[pfe] (del2) at (1,1.5) {};
    \node[pbsm] (f3) at (0,2.5) {$f$};
    \draw (pi2) -- (del2) -- (f3) -- ++(0,2);
    \node[pbe] (del1) at (2,2.5) {};
    \draw (del2) -- (del1);
    \node[pfe] (pi1) at (2,3.5) {};
    \draw (del1) -- (pi1);
    \node[fiber] (m) at (3,.5) {$M$};
    \draw (del1) -- (m);
  \end{tikzpicture}}
\hspace{1cm}
\subfigure[]{\label{fig:string_diagram_switch2}
  \begin{tikzpicture}[scale=.8]
   \node[pbeflat] (pi2) at (-.5,1.5) {};
    \node[pfeflat] (del2) at (-1.5,1.5) {};
    \node[pbflat] (f3) at (-3,1) {$f$};
    \node (end) at (-4.5,1){};
    \draw (pi2) -- (del2) -- (f3) -- (end);
    \node[pbeflat] (del1) at (-2.5,2) {};
    \draw (del2) -- (del1);
    \node[pfeflat] (pi1) at (-3.5,2) {};
    \draw (del1) -- (pi1);
    \node[fiber] (m) at (-0,2.5) {$M$};
    \draw (del1) -- (m);
  \end{tikzpicture}}
\caption{Transition to ``slice diagrams"}
\end{figure}
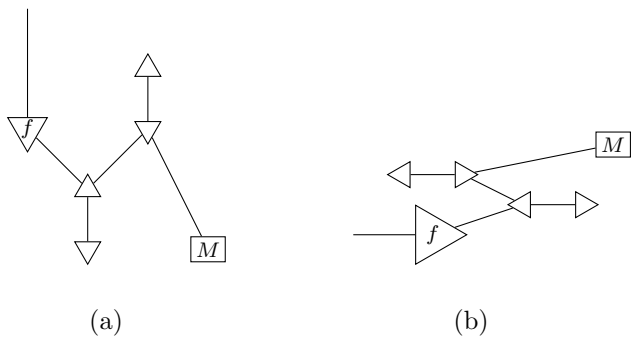
In these ``slice'' diagrams, the triangles for $f_!$ functors point to
the left, in contrast with those for $f^*$ functors which point to the
right.

The morphisms will be represented by composites of ``basic'' morphisms,
each of which is drawn as a node lying in between the corresponding
slices, connected by strings to the nodes above and below which are its
``direct'' input and output.  For other nodes which ``do not
participate'' in the morphism, and thus appear identically above and
below, we connect their incarnations in the upper and lower slices by a
direct string.  Just as with ordinary string diagrams, this allows us to
see visually when two basic morphisms ``do not interact'' at all, and
thus can be ``slid past each other'' by using a naturality property.

As an example, suppose we have morphisms $A\xto{f} B \xto{g} C$ in \bS,
along with a morphism $\phi\colon M\to N$ in $\sC^D$.  Then the
following square commutes, by naturality of the pseudofunctor
isomorphism $f^* g^* \cong (gf)^*$:
\begin{equation}\label{eq:egslice}
  \xymatrix{ f^* g^* M \ar[r]^{\cong} \ar[d]_{f^*g^* \phi} & (gf)^* M \ar[d]^{(gf)^*\phi} \\
    f^* g^* N \ar[r]_{\cong} & (gf)^* N. }
\end{equation}

\begin{figure}
  \centering
  \begin{tikzpicture}[strings]
    \node[fiber] (m0) at (0,0,0) {$M$};
    \node[pbflat] (g10) at (1,0,0) {$g$};
    \node[pbsmflat] (f20) at (2,0,0) {$f$};
    \draw (m0) -- (g10) -- (f20) -- (3,0,0);

    \node[fiber] (n1) at (0,1,0) {$N$};
    \node[pbflat] (g11) at (1,1,0) {$g$};
    \node[pbsmflat] (f21) at (2,1,0) {$f$};
    \draw (n1) -- (g11) -- (f21) -- (3,1,0);

    \draw[inner] (g10) -- (g11);
    \draw[inner] (f20) -- (f21);
    \draw[outer] (m0) -- node[outervert] {$\phi$} (n1);

    \node[fiber] (n2) at (0,2,0) {$N$};
    \node[pbsmflat] (gf12) at (1.5,2,0) {$gf$};
    \draw (n2) -- (gf12) -- (3,2,0);

    \draw[outer] (n1) -- (n2);
    \node[inneriso] (iso) at (1.5,1.5,0) {};
    \draw[inner] (g11) -- (iso);
    \draw[inner] (f21) -- (iso);
    \draw[inner] (iso) -- (gf12);
  \end{tikzpicture}
  \quad\raisebox{2cm}{$=$}\quad
  \begin{tikzpicture}[strings]
    \node[fiber] (m0) at (0,0,0) {$M$};
    \node[pbflat] (g10) at (1,0,0) {$g$};
    \node[pbsmflat] (f20) at (2,0,0) {$f$};
    \draw (m0) -- (g10) -- (f20) -- (3,0,0);

    \node[fiber] (m1) at (0,1,0) {$M$};
    \node[pbsmflat] (gf11) at (1.5,1,0) {$gf$};
    \draw (n1) -- (gf11) -- (3,1,0);

    \draw[outer] (m0) -- (m1);
    \node[inneriso] (iso) at (1.5,.5,0) {};
    \draw[inner] (g10) -- (iso);
    \draw[inner] (f20) -- (iso);
    \draw[inner] (iso) -- (gf11);

    \node[fiber] (n2) at (0,2,0) {$N$};
    \node[pbsmflat] (gf12) at (1.5,2,0) {$gf$};
    \draw (n2) -- (gf12) -- (3,2,0);

    \draw[inner] (gf11) -- (gf12);
    \draw[outer] (m1) -- node[outervert] {$\phi$} (n2);
  \end{tikzpicture}
  \caption{A slice diagram}
  \label{fig:egslice}
\end{figure}
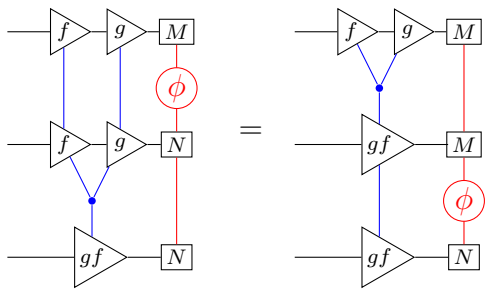

The equality of diagrams representing this commutative square is shown
in Figure~\ref{fig:egslice}.  Hopefully the reader can see the advantage
of this over~\eqref{eq:egslice}.  To distinguish the strings \emph{in}
the slices from the strings that represent slice transitions, we draw
the former with
\begin{bwcompromise}solid black\end{bwcompromise}
\begin{bwonly}solid\end{bwonly}
\begin{coloronly}black\end{coloronly}
lines and the latter with
\begin{bwcompromise}dashed or dotted colored\end{bwcompromise}
\begin{bwonly}dashed or dotted\end{bwonly}
\begin{coloronly}colored\end{coloronly}
lines.  For additional clarity, we further distinguish two different
types of transition strings: those which connect to triangle vertices
(which represent morphisms of \bS) and those which connect to box
vertices (which represent objects in fiber categories).  We draw the
former with
\begin{bwcompromise}dotted blue\end{bwcompromise}
\begin{bwonly}dotted\end{bwonly}
\begin{coloronly}blue\end{coloronly}
lines and the latter with
\begin{bwcompromise}dashed red\end{bwcompromise}
\begin{bwonly}dashed\end{bwonly}
\begin{coloronly}red\end{coloronly}
lines.

The units and counits of the adjunctions $f_!\dashv f^*$ will frequently
occur as basic morphisms in slice diagrams; we notate these with small
\innertext triangles.  As usual in string diagram notations, the
triangle identities for these adjunctions look like simple topological
deformations; some examples are shown in Figure~\ref{fig:surface-adj}.

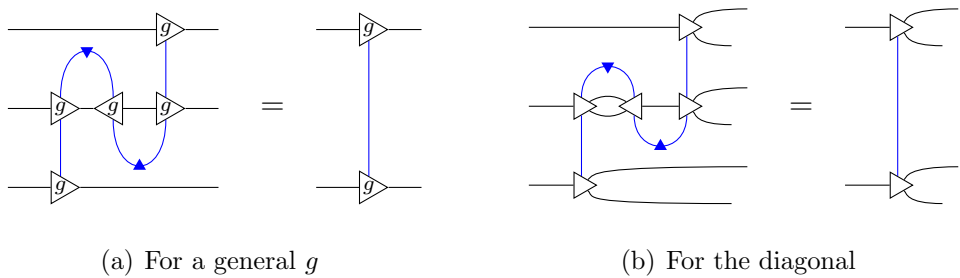
\begin{figure}
  \centering
  \subfigure[For a general $g$]{
    \begin{tikzpicture}[backwards, scale=.7]
      \node[pbsmflat] (0f3) at (1,0,0) {$g$};
      \draw (0,0,0) -- (0f3) -- (4,0,0);
      \node[pbsmflat] (1f1) at (1,1,0) {$g$};
      \node[pfsmflat] (1f2) at (2,1,0) {$g$};
      \node[pbsmflat] (1f3) at (3,1,0) {$g$};
      \draw (0,1,0) -- (1f1) -- (1f2) -- (1f3) -- (4,1,0);
      \node[pbsmflat] (2f1) at (3,2,0) {$g$};
      \draw (0,2,0) -- (2f1) -- (4,2,0);
      \draw[inner] (0f3) -- (1f1) to[out=-90,in=-90,looseness=3]
      node[innercounit] {} (1f2);
      \draw[inner] (1f2) to[out=90,in=90,looseness=3]
      node[innerunit] {} (1f3);
      \draw[inner] (1f3) -- (2f1);
    \end{tikzpicture}
    \quad\raisebox{1.2cm}{$=$}\quad
    \begin{tikzpicture}[backwards,scale=.7]
      \node[pbsmflat] (0f1) at (1,0,0) {$g$};
      \draw (0,0,0) -- (0f1) -- (2,0,0);
      \node[pbsmflat] (2f1) at (1,2,0) {$g$};
      \draw (0,2,0) -- (2f1) -- (2,2,0);
      \draw[inner] (0f1) -- (2f1);
    \end{tikzpicture}}
  \hspace{1cm}
  \subfigure[For the diagonal]{\label{fig:surface-adj-diag}
    \begin{tikzpicture}[backwards,scale=.7]
      \node[pbeflat] (0del1) at (1,0,0) {};
      \draw (0,0,.5) .. controls +(0:.7) and +(150:.2) .. (0del1);
      \draw (0,0,-.5) .. controls +(0:1) and +(210:.2) .. (0del1);
      \draw (0del1) -- (4,0,0);
      \node[pbeflat] (1del1) at (1,1,0) {};
      \draw (0,1,.5) .. controls +(0:.7) and +(150:.2) .. (1del1);
      \draw (0,1,-.5) .. controls +(0:1) and +(210:.2) .. (1del1);
      \node[pfeflat] (1del2) at (2,1,0) {};
      \node[pbeflat] (1del3) at (3,1,0) {};
      \draw (1del1) -- (1del2);
      \draw (1del2) to[bend right] (1del3);
      \draw (1del2) to[bend left] (1del3);
      \draw (1del3) -- (4,1,0);
      \node[pbeflat] (2del3) at (3,2,0) {};
      \draw (0,2,.5) .. controls +(0:3) and +(150:.3) .. (2del3);
      \draw (0,2,-.5) .. controls +(0:3) and +(210:.3) .. (2del3);
      \draw (2del3) -- (4,2,0);
        \draw[inner] (0del1) -- (1del1);
        \draw[inner]  (1del1) to[out=-90,in=-90,looseness=2]
        node[innercounit] {} (1del2);
        \draw[inner]  (1del2) to[out=90,in=90,looseness=2]
        node[innerunit] {} (1del3);
        \draw[inner]  (1del3) -- (2del3);
    \end{tikzpicture}
    \quad\raisebox{1.2cm}{$=$}\quad
    \begin{tikzpicture}[backwards,scale=.7]
      \node[pbeflat] (0del1) at (1,0,0) {};
      \draw (0,0,.5) .. controls +(0:.7) and +(150:.2) .. (0del1);
      \draw (0,0,-.5) .. controls +(0:1) and +(210:.2) .. (0del1);
      \draw (0del1) -- (2,0,0);
      \node[pbeflat] (2del1) at (1,2,0) {};
      \draw (0,2,.5) .. controls +(0:.7) and +(150:.2) .. (2del1);
      \draw (0,2,-.5) .. controls +(0:1) and +(210:.2) .. (2del1);
      \draw (2del1) -- (2,2,0);
      \draw[inner] (0del1) -- (2del1);
    \end{tikzpicture}}
  \caption{Triangle identities for the adjunctions $g_! \dashv g^*$}
  \label{fig:surface-adj}
\end{figure}

In Figure~\ref{fig:dual-fiber} we further illustrate this notation by
drawing the evaluation, coevaluation, and one triangle identity for a
dual pair in a fiber category.  (Recall that unlabeled triangles
represent diagonal and projection morphisms.)  We use small \innertext
dots to represent the pseudofunctoriality isomorphisms of the indexed
category; in Figure~\ref{fig:dual-fiber-triangle} the \innertext dots
represent the isomorphisms
\begin{align*}
  \Delta^*(\pi\times\id)^* &\cong \Id\\
  \Delta^*(\Delta\times \id)^* &\cong\Delta^*(\id\times \Delta)^*\\
  \Delta^*(\id\times\pi)^* &\cong \Id.
\end{align*}
Observe that the \outertext strings connecting the box nodes in
Figure~\ref{fig:dual-fiber-triangle} display the same shape as the usual
representation of a triangle identity in a monoidal category.
\begin{figure}
  \centering
  \subfigure[The evaluation]{\label{fig:dual-fiber-eval}
    \begin{tikzpicture}[x={(-1,0)},scale=.6]
      \node[pbeflat] (1pi) at (0,-1) {};
      \draw (1pi) -- ++(4,0);
      \node[fiber] (2m) at (.5,2) {$M$};
      \node[fiber] (2md) at (-1,3) {$\rdual{M}$};
      \node[pbeflat] (2del) at (2,2.5) {};
      \draw (2m) -- (2del);
      \draw (2md) -- (2del);
      \draw (2del) -- ++(2,0);
        \node[outervert] (ep) at (0,.3) {$\varepsilon$};
        \draw[inner] (1pi) -- (ep);
        \draw[outer] (ep) to[out=45,in=-90] (2md);
        \draw[outer] (ep) to[out=120,in=-90] (2m);
        \draw[inner] (ep) to[out=150,in=-90] (2del);
    \end{tikzpicture}}
  \hspace{2cm}
  \subfigure[The coevaluation]{\label{fig:dual-fiber-coeval}
    \begin{tikzpicture}[x={(-1,0)},y={(0,-1)},scale=.6]
      \node[pbeflat] (1pi) at (0,-1) {};
      \draw (1pi) -- ++(4,0);
      \node[fiber] (2m) at (.5,3) {$M$};
      \node[fiber] (2md) at (-1,2) {$\rdual{M}$};
      \node[pbeflat] (2del) at (2,2.5) {};
      \draw (2m) -- (2del);
      \draw (2md) -- (2del);
      \draw (2del) -- ++(2,0);
        \node[outervert] (eta) at (0,.3) {$\eta$};
        \draw[inner] (1pi) -- (eta);
        \draw[outer] (eta) to[out=-45,in=90] (2md);
        \draw[outer] (eta) to[out=-120,in=90] (2m);
        \draw[inner] (eta) to[out=-150,in=90] (2del);
    \end{tikzpicture}}
  \\\vspace{1cm}
  \subfigure[A triangle identity]{\label{fig:dual-fiber-triangle}
    \begin{tikzpicture}[strings,y={(0,1.3))}]
      \def\eqx{-7}
      \def\yy{0}
      \node[fiber] (m\yy) at (-2,\yy,-1) {$M$};
      \draw (m\yy) -- (3,\yy,-1);
      \node (eq\yy) at (\eqx,\yy,-1) {$M$};

      \def\yy{1}
      \node[fiber] (m\yy) at (-2,\yy,-1) {$M$};
      \node[pbeflat] (d1\yy) at (1,\yy,0) {};
      \node[pbeflat] (d2\yy) at (2,\yy,-1) {};
      \draw (m\yy) -- (d2\yy) -- (3,\yy,-1);
      \draw (d1\yy) -- (d2\yy);
      \node (eq\yy) at (\eqx,\yy,-1) {$\Delta^*(\pi^* U \boxtimes M)$};
      \draw[->] (eq0) -- node[ed,right=2pt] {$\cong$} (eq1);

      \def\yy{2}
      \node[fiber] (m\yy) at (-2,\yy,-1) {$M$};
      \node[fiber] (md\yy) at (-1,\yy,0) {$\rdual{M}$};
      \node[fiber] (m'\yy) at (0,\yy,1) {$M$};
      \node[pbeflat] (d1\yy) at (1,\yy,0) {};
      \node[pbeflat] (d2\yy) at (2,\yy,-1) {};
      \draw (m\yy) -- (d2\yy) -- (3,\yy,-1);
      \draw (d1\yy) -- (d2\yy);
      \draw (md\yy) -- (d1\yy);
      \draw (m'\yy) -- (d1\yy);
      \node (eq\yy) at (\eqx,\yy,0)
      {$\Delta^*(\Delta^*(M\boxtimes \rdual{M})\boxtimes M)$};
      \draw[->] (eq1) -- node[ed,right=2pt]
      {$\Delta^*(\eta\boxtimes\id)$} (eq2);

      \def\yy{3}
      \node[fiber] (m\yy) at (-2,\yy,-1) {$M$};
      \node[fiber] (md\yy) at (-1,\yy,0) {$\rdual{M}$};
      \node[fiber] (m'\yy) at (0,\yy,1) {$M$};
      \node[pbeflat] (d1\yy) at (1,\yy,-1) {};
      \node[pbeflat] (d2\yy) at (2,\yy,0) {};
      \draw (m\yy) -- (d1\yy) -- (d2\yy) -- (3,\yy,0);
      \draw (md\yy) -- (d1\yy);
      \draw (m'\yy) -- (d2\yy);
      \node (eq\yy) at (\eqx,\yy,0)
      {$\Delta^*(M\boxtimes \Delta^*(\rdual{M}\boxtimes M))$};
      \draw[->] (eq2) -- node[ed,right=2pt] {$\cong$} (eq3);

      \def\yy{4}
      \node[fiber] (m'\yy) at (0,\yy,1) {$M$};
      \node[pbeflat] (d1\yy) at (1,\yy,-1) {};
      \node[pbeflat] (d2\yy) at (2,\yy,0) {};
      \draw (d1\yy) -- (d2\yy) -- (3,\yy,0);
      \draw (m'\yy) -- (d2\yy);
      \node (eq\yy) at (\eqx,\yy,1) {$\Delta^*(M\boxtimes \pi^*U)$};
      \draw[->] (eq3) -- node[ed,right=2pt]
      {$\Delta^*(\id\boxtimes\,\varepsilon)$} (eq4);

      \def\yy{5}
      \node[fiber] (m'\yy) at (0,\yy,1) {$M$};
      \draw (m'\yy) -- (3,\yy,1);
      \node (eq\yy) at (\eqx,\yy,1) {$M$};
      \draw[->] (eq4) -- node[ed,right=2pt] {$\iso$} (eq5);

      \node[outervert] (eta) at (-.5,1.5,0) {$\eta$};
      \node[outervert] (ep) at (-1.5,3.5,0) {$\ep$};

      \begin{scope}[outer]
        \draw (m0) -- (m1) -- (m2) -- (m3);
        \draw (m'2) -- (m'3) -- (m'4) -- (m'5);
        \draw (md2) -- (md3);
        \draw (eta) to[out=-45,in=90] (md2);
        \draw (eta) to[out=-135,in=90] (m'2);
        \draw (ep) to[out=45,in=-90] (m3);
        \draw (ep) to[out=135,in=-90] (md3);
      \end{scope}
      \begin{scope}[inner]
        \draw (d11) to[out=90,in=90,looseness=3] node[inneriso] {} (d21);
        \draw (d21) -- (d22) -- (d13);
        \draw (d14) to[out=-90,in=-90,looseness=3] node[inneriso] {} (d24);
        \draw (d12) -- (d23) -- (d24);
        \node[inneriso] at (intersection of d22--d13 and d12--d23) {};
        \draw (d11) to[out=-90,in=165] (eta);
        \draw (eta) to[out=-165,in=90] (d12);
        \draw (d13) to[out=-90,in=165] (ep);
        \draw (ep) to[out=180,in=90] (d14);
      \end{scope}
      \node at (-3.5,2.5) {$=$};
      \node at (-11,2.5) {$=\; \id_M$};
    \end{tikzpicture}}
  \caption{A dual pair in a fiber category}
  \label{fig:dual-fiber}
\end{figure}

The Beck-Chevalley conditions corresponding to the four pullback
diagrams in Figure~\ref{fig:pullbacks}, and which were shown in
Figure~\ref{fig:beckchev} as isomorphisms between two (fragments of)
string diagrams, are displayed in Figure~\ref{fig:surface-bc} as
morphisms between slices.  The transposes of Figures~\ref{fig:frobpb}
and~\ref{fig:slidpb} are represented by the left-to-right reflections of
Figures~\ref{fig:frobbc-surf} and~\ref{fig:slidbc-surf}, respectively.
\begin{figure}
  \centering
  \begin{tabular}{c@{\hspace{1cm}}c}
  \subfigure[Commutativity with reindexing]{\label{fig:reindbc-surf}\qquad
    \begin{tikzpicture}[x={(-1,0)},y={(0,-1)}]
      \coordinate (1m) at (1,.5);
      \node[pfsmflat] (1f) at (3,.5) {$f$};
      \draw (1m) -- (1f) -- (4.5,.5);
      \coordinate (1n) at (0.5,0);
      \node[pbflat] (1g) at (2,0) {$g$};
      \draw (1n) -- (1g) -- (4,0);
      %
      \coordinate (2m) at (1,2.5);
      \node[pfsmflat] (2f) at (2,2.5) {$f$};
      \draw (2m) -- (2f) -- (4.5,2.5);
      \coordinate (2n) at (0.5,2);
      \node[pbflat] (2g) at (3,2) {$g$};
      \draw (2n) -- (2g) -- (4,2);
      %
        \draw[inner] (1g) -- (2g);
        \draw[line width=4pt,white] (1f) -- (2f);
        \draw[inner] (1f) -- (2f);
    \end{tikzpicture}\qquad}
  &
  \subfigure[The Frobenius axiom]{\label{fig:frobbc-surf}
    \hspace{.4cm}\begin{tikzpicture}[backwards]
      \node[pbeflat] (1del1) at (1,0,.5) {};
      \node[pfeflat] (1del2) at (2,0,.5) {};
      \draw (0,0,0) -- (1del1) -- (1del2) -- (3,0,0);
      \draw (0,0,1) -- (1del1) (1del2) -- (3,0,1);
      \node[pfeflat] (2del1) at (.8,1,1) {};
      \node[pbeflat] (2del2) at (2.2,1,.5) {};
      \draw (0,1,1) -- (2del1) -- (2del2) -- (3,1,.5);
      \draw (0,1,0) -- (1,1,0) -- (2del2);
      \draw (2del1) -- (2,1,1.5) -- (3,1,1.5);
      %
        \draw[inner] (1del1) -- (2del2);
        \draw[inner] (1del2) -- (2del1);
        \node[innerbc, inner] at (intersection of 1del1--2del2 and 1del2--2del1) {};
      \end{tikzpicture}\hspace{.4cm}}
    \\
    \subfigure[Sliding and splitting]{\label{fig:slidbc-surf}%
      \begin{tikzpicture}[z={(0,-.3)},x={(-1,0)},y={(0,-1.9)}]
        \node[pbsmflat] (f10) at (1,0,0) {$f$};
        \node[pfeflat] (del20) at (2,0,0) {};
        \node[pfsmflat] (f30) at (3,0,-1) {$f$};
        \draw (0,0,0) -- (f10) -- (del20) -- (f30) -- +(1,0,0);
        \draw (del20) -- ++(1,0,1) -- +(1,0,0);
        \node[pfeflat] (del21) at (2,1,0) {};
        \node[pfsmflat] (f31) at (3,1,1) {$f$};
        \draw (0,1,0) -- (del21) -- (f31) -- +(1,0,0);
        \draw (del21) -- ++(1,0,-1) -- +(1,0,0);
        \node[innerbc] (ss) at (2,.5,0) {};
        \draw[inner] (del20) -- (ss) -- (del21);
        \draw[inner] (f10) -- (ss);
        \draw[inner] (f30) -- (ss) -- (f31);
      \end{tikzpicture}}
    &
    \subfigure[Monic diagonals]{\label{fig:mdbc-surf}%
      \hspace{.4cm}\begin{tikzpicture}[z={(0,-.3)},x={(-1,0)},y={(0,-1.9)}]
        \draw (0,0,0) -- (3,0,0);
        \node[pfeflat] (del11) at (1,1,0) {};
        \node[pbeflat] (del21) at (2,1,0) {};
        \draw (0,1,0) -- (del11);
        \draw (del21) -- (3,1,0);
        \draw (del11) to[bend right] (del21);
        \draw (del11) to[bend left] (del21);
        \draw[inner] (del11) to[out=90,in=90,looseness=3]
        node [innerunit] {} (del21);
      \end{tikzpicture}\hspace{.4cm}}
    \end{tabular}
  \caption{Surfaces for Beck-Chevalley conditions}
  \label{fig:surface-bc}
\end{figure}
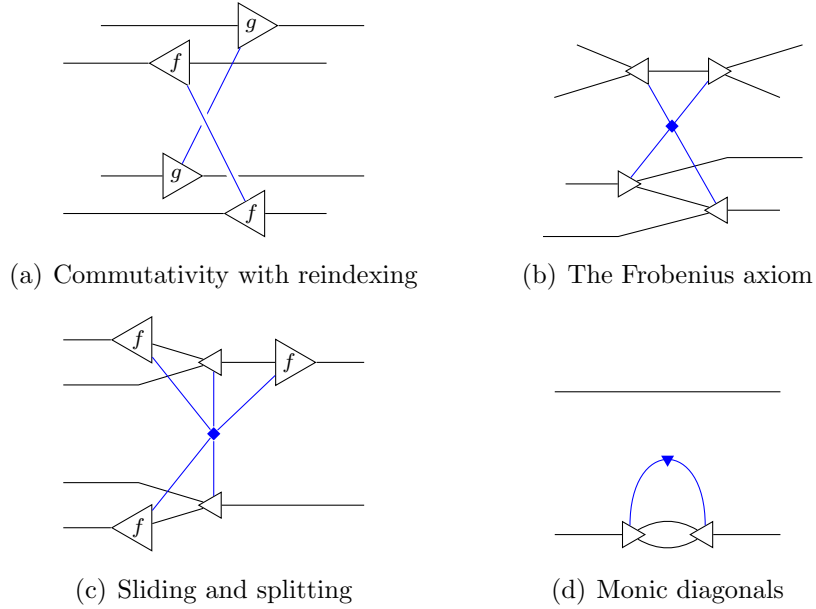

In Figure~\ref{fig:reindbc-surf} we draw the Beck-Chevalley morphism
simply as one \innertext string crossing in front of another, since it
represents merely a deformation of the slice diagrams. In
Figure~\ref{fig:mdbc-surf} we notate the Beck-Chevalley morphism with a
triangle, since it can be identified with the unit of the adjunction
$\Delta_! \dashv \Delta^*$.  Finally, for the Beck-Chevalley morphisms
in Figures~\ref{fig:frobbc-surf} and~\ref{fig:slidbc-surf}, and their
inverses when these occur, we use a small \innertext diamond.  The
isomorphisms in Figures~\ref{fig:frobbc-surf}, \ref{fig:slidbc-surf},
and~\ref{fig:mdbc-surf}, however, are less ``topological'' than that in
Figure~\ref{fig:reindbc-surf} (although Figure~\ref{fig:frobbc-surf}
becomes a topological deformation if we replace our strings by
\emph{tubes}, as is well-known in the context of Topological Quantum
Field Theory).

We end this section with an important computation that will be used many
times in the following sections.  First of all, since the composite $A
\xto{\Delta_A} A\times A \xto{\pi_A\times \id_A} A$ is the identity
$\id_A$, we have
\begin{equation}
  (\Delta_A)^* (\pi_A\times\id_A)^* \cong \Id
  \qquad\text{and}\qquad
  (\pi_A\times\id_A)_! (\Delta_A)_! \cong \Id.\label{eq:delpiid}
\end{equation}
These isomorphisms already appeared in
Figure~\ref{fig:dual-fiber-triangle}; in Figure~\ref{fig:delpiid} we
isolate them.

Secondly, we can say more: the adjunction $(\pi_A\times\id_A)_!
(\Delta_A)_! \dashv (\Delta_A)^* (\pi_A\times\id_A)^*$ is isomorphic to
the identity adjunction.  In particular, its unit
\[ \Id \longrightarrow (\Delta_A)^*(\Delta_A)_!
\longrightarrow (\Delta_A)^* (\pi_A\times\id_A)^* (\pi_A\times\id_A)_! (\Delta_A)_!
\]
is equal to the composite of the inverses of the
isomorphisms~\eqref{eq:delpiid}.  Graphically, this implies that the
``blob'' in Figure~\ref{fig:blob} is equal to the identity.  We obtain
Figure~\ref{fig:curlycue}, which we call a \emph{curlycue identity}, by
moving one of the isomorphisms~\eqref{eq:delpiid} to the other side.

Finally, in Figure~\ref{fig:curlycue-mate}, we go from the first diagram
to the second by adding a triangle identity for the adjunction $\Delta_!
\dashv \Delta^*$ at the top and sliding the counit down to the bottom,
then to the third diagram by applying the curlycue identity.  We will
refer to the equality of the first and third diagrams in
Figure~\ref{fig:curlycue-mate} as a \emph{broken zigzag identity}.
\begin{figure}
  \centering
  \subfigure[The isomorphisms~\eqref{eq:delpiid}]{\label{fig:delpiid}%
    \raisebox{1.5cm}{%
  \begin{tikzpicture}[scale=.75, z={(0,-.3)},x={(-1,0)},y={(0,1.9)}]
    \draw (0,0,0) -- (3,0,0);
    \node[pbeflat] (del11) at (2,1,0) {};
    \node[pbeflat] (pi11) at (1,1,1) {};
    \draw (0,1,-1) -- (del11) -- (3,1,0);
    \draw (pi11) -- (del11);
    \draw[inner] (del11) to[out=-90,in=-90,looseness=3] node[inneriso] {} (pi11);
  \end{tikzpicture}
  \qquad
  \begin{tikzpicture}[scale=.75, z={(0,-.3)},x={(1,0)},y={(0,1.9)}]
    \draw (0,0,0) -- (3,0,0);
    \node[pfeflat] (del11) at (2,1,0) {};
    \node[pfeflat] (pi11) at (1,1,1) {};
    \draw (0,1,-1) -- (del11) -- (3,1,0);
    \draw (pi11) -- (del11);
    \draw[inner] (del11) to[out=-90,in=-90,looseness=3] node[inneriso] {} (pi11);
  \end{tikzpicture}}}
  \hspace{2cm}
  \subfigure[The blob]{\label{fig:blob}%
    \begin{tikzpicture}[scale=.75, z={(0,-.3)},x={(-1,0)},y={(0,-1.9)}]
    \draw (0,0,0) -- (5,0,0);
    \node[pfeflat] (del11) at (1,1,0) {};
    \node[pbeflat] (del14) at (4,1,0) {};
    \draw (0,1,0) -- (del11);
    \draw (del14) -- (5,1,0);
    \draw (del11) to[bend right] (del14);
    \draw (del11) to[bend left] (del14);
    \draw[inner] (del11) to[out=90,in=90,looseness=1] node[innerunit] {} (del14);
    \node[pfeflat] (del21) at (1,2,0) {};
    \node[pbeflat] (del24) at (4,2,0) {};
    \draw (0,2,0) -- (del21);
    \draw (del24) -- (5,2,0);
    \node[pfeflat] (pi22) at (2,2,1) {};
    \node[pbeflat] (pi23) at (3,2,1) {};
    \draw (del21) -- (pi22);
    \draw (del24) -- (pi23);
    \draw (del21) to[bend right] (del24);
    \draw[inner] (pi22) to[out=90,in=90,looseness=3.5] node[innerunit] {} (pi23);
    \draw[inner] (del11) -- (del21);
    \draw[inner] (del14) -- (del24);
    \draw (0,3,0) -- (5,3,0);
    \draw[inner] (del21) to[out=-90,in=-90,looseness=3] node[inneriso] {} (pi22);
    \draw[inner] (del24) to[out=-90,in=-90,looseness=3] node[inneriso] {} (pi23);
  \end{tikzpicture}}\\
\subfigure[Straightening out a curlycue]{\label{fig:curlycue}%
  \begin{tikzpicture}[scale=.75, z={(0,-.3)},x={(-1,0)},y={(0,-1.9)}]
    \draw (0,0,0) -- (5,0,0);
    \node[pfeflat] (del11) at (1,1,0) {};
    \node[pbeflat] (del14) at (4,1,0) {};
    \draw (0,1,0) -- (del11);
    \draw (del14) -- (5,1,0);
    \draw (del11) to[bend right] (del14);
    \draw (del11) to[bend left] (del14);
    \draw[inner] (del11) to[out=90,in=90,looseness=1] node[innerunit] {} (del14);
    \node[pfeflat] (del21) at (1,2,0) {};
    \node[pbeflat] (del24) at (4,2,0) {};
    \draw (0,2,0) -- (del21);
    \draw (del24) -- (5,2,0);
    \node[pfeflat] (pi22) at (2,2,1) {};
    \node[pbeflat] (pi23) at (3,2,1) {};
    \draw (del21) -- (pi22);
    \draw (del24) -- (pi23);
    \draw (del21) to[bend right] (del24);
    \draw[inner] (pi22) to[out=90,in=90,looseness=3.5] node[innerunit] {} (pi23);
    \draw[inner] (del11) -- (del21);
    \draw[inner] (del14) -- (del24);
    \node[pfeflat] (del31) at (1,3,0) {};
    \node[pfeflat] (pi32) at (2,3,1) {};
    \draw (0,3,0) -- (del31);
    \draw (del31) -- (pi32);
    \draw (del31) -- (5,3,-1);
    \draw[inner] (del24) to[out=-90,in=-90,looseness=3] node[inneriso] {} (pi23);
    \draw[inner] (del21) -- (del31);
    \draw[inner] (pi22) -- (pi32);
  \end{tikzpicture}
  \qquad \raisebox{2.5cm}{=} \qquad
  \raisebox{1.5cm}{\begin{tikzpicture}[scale=.75, z={(0,-.3)},x={(1,0)},y={(0,-1.9)}]
    \draw (0,0,0) -- (3,0,0);
    \node[pfeflat] (del11) at (2,1,0) {};
    \node[pfeflat] (pi11) at (1,1,1) {};
    \draw (0,1,-1) -- (del11) -- (3,1,0);
    \draw (pi11) -- (del11);
    \draw[inner] (del11) to[out=90,in=90,looseness=3.5] node[inneriso] {} (pi11);
  \end{tikzpicture}}}\\
\subfigure[The broken zigzag]{\label{fig:curlycue-mate}
  \raisebox{.7cm}{\begin{tikzpicture}[strings]
    \def\yy{0}
    \node[pbeflat] (d3\yy) at (3,\yy,0) {};
    \draw (0,\yy,-1) -- (d3\yy);
    \draw (0,\yy,1) -- (d3\yy) -- (4,\yy,0);

    \def\yy{1}
    \node[pbeflat] (d3\yy) at (3,\yy,0) {};
    \node[pbeflat] (d2\yy) at (2,\yy,1) {};
    \node[pfeflat] (d1\yy) at (1,\yy,1) {};
    \draw (0,\yy,-1) -- (d3\yy);
    \draw (0,\yy,1) -- (d1\yy);
    \draw (d2\yy) -- (d3\yy) -- (4,\yy,0);

    \def\yy{2}
    \node[pfeflat] (d1\yy) at (1,\yy,1) {};
    \draw (0,\yy,1) -- (d1\yy);
    \draw (0,\yy,0) -- (4,\yy,0);
    \begin{scope}[inner]
      \draw (d30) -- (d31);
      \draw (d31) to[out=-90,in=-90,looseness=3] node[inneriso] {} (d21);
      \draw (d21) to[out=90,in=90,looseness=3] node[innerunit] {} (d11);
      \draw (d11) -- (d12);
    \end{scope}
  \end{tikzpicture}}
  \quad \raisebox{2.5cm}{=} \quad
  \begin{tikzpicture}[scale=.6, z={(0,-.3)},x={(-1,0)},y={(0,-1.9)}]
    \node[pbeflat] (del00) at (0,0,0) {};
    \draw (del00) -- (5,0,0);
    \draw (-1,0,-1) -- (del00);
    \draw (-1,0,1) -- (del00);
    \node[pbeflat] (del10) at (0,1,0) {};
    \draw (-1,1,-1) -- (del10);
    \draw (-1,1,1) -- (del10);
    \node[pfeflat] (del11) at (1,1,0) {};
    \node[pbeflat] (del14) at (4,1,0) {};
    \draw (del10) -- (del11);
    \draw (del14) -- (5,1,0);
    \draw (del11) to[bend right] (del14);
    \draw (del11) to[bend left] (del14);
    \draw[inner] (del11) to[out=90,in=90,looseness=1] node[innerunit] {} (del14);
    \node[pbeflat] (del20) at (0,2,0) {};
    \draw (-1,2,-1) -- (del20);
    \draw (-1,2,1) -- (del20);
    \node[pfeflat] (del21) at (1,2,0) {};
    \node[pbeflat] (del24) at (4,2,0) {};
    \draw (del20) -- (del21);
    \draw (del24) -- (5,2,0);
    \node[pfeflat] (pi22) at (2,2,1) {};
    \node[pbeflat] (pi23) at (3,2,1) {};
    \draw (del21) -- (pi22);
    \draw (del24) -- (pi23);
    \draw (del21) to[bend right] (del24);
    \draw[inner] (pi22) to[out=90,in=90,looseness=3.5] node[innerunit] {} (pi23);
    \draw[inner] (del11) -- (del21);
    \draw[inner] (del14) -- (del24);
    \node[pbeflat] (del30) at (0,3,0) {};
    \draw (-1,3,-1) -- (del30);
    \draw (-1,3,1) -- (del30);
    \node[pfeflat] (del31) at (1,3,0) {};
    \node[pfeflat] (pi32) at (2,3,1) {};
    \draw (del30) -- (del31);
    \draw (del31) -- (pi32);
    \draw (del31) -- (5,3,-1);
    \draw[inner] (del24) to[out=-90,in=-90,looseness=3] node[inneriso] {} (pi23);
    \draw[inner] (del21) -- (del31);
    \draw[inner] (pi22) -- (pi32);
    \draw (-1,4,-1) -- (5,4,-1);
    \node[pfeflat] (pi42) at (2,4,1) {};
    \draw (-1,4,1) -- (pi42);
    \draw[inner] (del00) -- (del10) -- (del20) -- (del30);
    \draw[inner] (del30) to[out=-90,in=-90,looseness=3] node[innercounit] {} (del31);
    \draw[inner] (pi32) -- (pi42);
  \end{tikzpicture}
  \quad \raisebox{2.5cm}{=} \quad
  \raisebox{.7cm}{\begin{tikzpicture}[strings]
    \def\yy{0}
    \node[pbeflat] (d1\yy) at (1,\yy,0) {};
    \draw (0,\yy,-1) -- (d1\yy);
    \draw (0,\yy,1) -- (d1\yy) -- (4,\yy,0);

    \def\yy{1}
    \node[pfeflat] (d3\yy) at (3,\yy,1) {};
    \node[pfeflat] (d2\yy) at (2,\yy,0) {};
    \node[pbeflat] (d1\yy) at (1,\yy,0) {};
    \draw (0,\yy,-1) -- (d1\yy);
    \draw (0,\yy,1) -- (d1\yy);
    \draw (d1\yy) -- (d2\yy) -- (d3\yy);
    \draw (d2\yy) -- (4,\yy,-1);

    \def\yy{2}
    \node[pfeflat] (d3\yy) at (3,\yy,1) {};
    \draw (0,\yy,0) -- (4,\yy,-0);
    \draw (0,\yy,1) -- (d3\yy);

    \begin{scope}[inner]
      \draw (d10) -- (d11);
      \draw (d11) to[out=-90,in=-90,looseness=3] node[innercounit] {} (d21);
      \draw (d21) to[out=90,in=90,looseness=3] node[inneriso] {} (d31);
      \draw (d31) -- (d32);
    \end{scope}
  \end{tikzpicture}}}
\caption{Blobs, curlycues, and broken zigzags}
\label{fig:curl}
\end{figure}
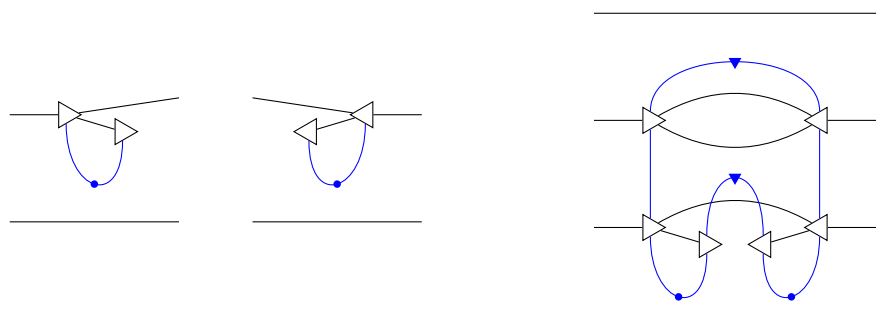
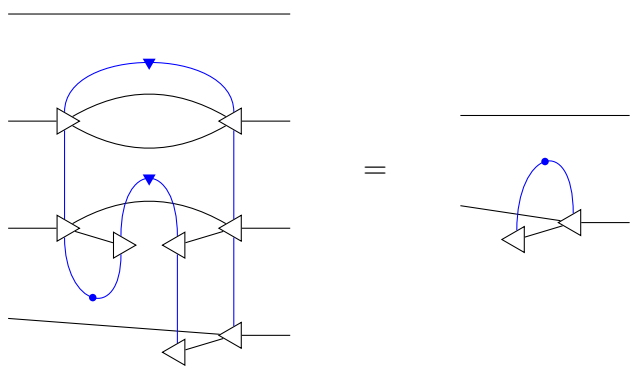
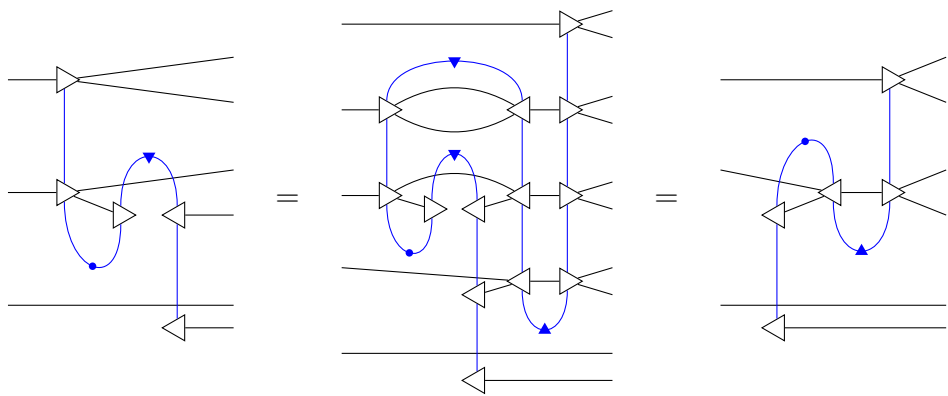

\begin{rmk}
  We have remarked that the Beck-Chevalley morphism for the
  non-homotopy-pullback square in Figure~\ref{fig:mdpb} is simply the
  unit $\Id \to (\Delta_A)^*(\Delta_A)_!$ of the adjunction
  $(\Delta_A)_! \dashv (\Delta_A)^*$.  Thus, the fact that
  Figure~\ref{fig:blob} is the identity means that regardless of whether
  the Beck-Chevalley condition holds, this morphism is always split
  monic; the lower 2/3 of Figure~\ref{fig:blob} supplies a retraction.
  (We would obtain a different retraction, however, by using
  $\id\times\pi$ instead of $\pi\times\id$.)
\end{rmk}

\begin{rmk}\label{rmk:surfaces}
  There are other ways to depict the adding of a dimension to string
  diagrams.  Rather than using a new type of string as we have done, a
  more usual approach would be to keep the codimension of all elements
  constant as we move up in dimension.  Thus, we would replace the
  strings and nodes occurring in the two-dimensional slice diagrams by
  surfaces and strings, respectively, and then use 0-dimensional nodes
  for the morphisms in the fiber categories.

  Arguably, using surfaces is the ``correct'' representation, and our
  ``slice'' diagrams can in fact be viewed as slices \emph{of} surfaces.
  The strings and nodes in the slices are slices of surfaces and
  strings, while the vertically drawn strings are actual strings in the
  surface diagram (singular junctions of surfaces) and the nodes along
  them are actual nodes.  The reader is welcome to interpret our
  diagrams in this way.  One advantage of this viewpoint is that then
  pictures such as those in Figure~\ref{fig:surface-adj-diag} are
  related by a simple topological deformation.  Moreover, if we view the
  an indexed monoidal category as sitting inside the bicategory
  constructed from it, as suggested in \autoref{rmk:bicat-ismc-surface},
  then such surface diagrams can be regarded as a fragment of the more
  traditional surface diagrams for monoidal bicategories.  (Our string
  diagrams are in fact an adaptation of a ``schematic'' or ``hybrid''
  sort of surface diagram for monoidal bicategories that was suggested
  to us by Daniel Sch\"appi.)

  However, the authors find it quite difficult to draw and visualize
  even moderately complicated surface diagrams---whereas we can easily
  manipulate our ``slice-transition'' diagrams schematically, without
  attempting to figure out what sort of ``surfaces'' they represent.
\end{rmk}

\section{Proofs for fiberwise duality and trace}
\label{sec:fiberwise-proofs}

In this section we will prove Theorems \ref{thm:fiberwise-duality} and
\ref{thm:fibtrace}, using string diagram calculations.

\begin{proof}[of \autoref{thm:fiberwise-duality}]
  For $\widehat{M}$ to be right dualizable, we require a 1-cell
  $\rdual{(\smash{\widehat{M}})}\maps 1\hto A$, which of course is of the
  form $\widecheck{N}$ for some $N\in\sC^A$, and morphisms
  \begin{align*}
    \overline{\eta}&\maps U_A \too \widehat{M} \odot \widecheck{N}\\
    \overline{\ep} &\maps \widecheck{N}\odot \widehat{M} \too U_\star
  \end{align*}
  satisfying the triangle identities.  Substituting in the definitions
  of the structure of $\calBi CS$, we can rewrite such maps as
  \begin{align*}
    \overline{\eta}&\maps (\Delta_A)_! I_A \too \pi_1^* M \ten \pi_2^*N\\
    \overline{\ep} &\maps (\pi_A)!(M\ten N) \too I_\star.
  \end{align*}
  These morphisms are depicted in Figures~\ref{fig:bicat-eval}
  and~\ref{fig:bicat-coeval}.

  However, giving such maps is equivalent to giving their adjuncts
  \begin{align*}
    \eta&\maps I_A \too \Delta_A^* (\pi_1^* M \ten \pi_2^*N) \too[\iso]
    \Delta_A^*\pi_1^* M \ten \Delta_A^*\pi_2^*N \too[\iso] M\ten N\\
    \ep &\maps (M\ten N) \too (\pi_A)^*I_\star \too[\iso] I_A
  \end{align*}
  and these are exactly the maps required to make $N$ into a dual of $M$
  in $\sC^A$, as shown in Figures~\ref{fig:dual-fiber-eval}
  and~\ref{fig:dual-fiber-coeval}.  This gives us a bijection between
  putative evaluation-coevaluation morphisms for dual pairs $(M,N)$ and
  $(\widehat{M},\widecheck{N})$.  In
  Figures~\ref{fig:bicat-to-smc-eval}, \ref{fig:bicat-to-smc-coeval},
  \ref{fig:smc-to-bicat-eval}, and~\ref{fig:smc-to-bicat-coeval}, we
  show explicitly how this bijection works.
  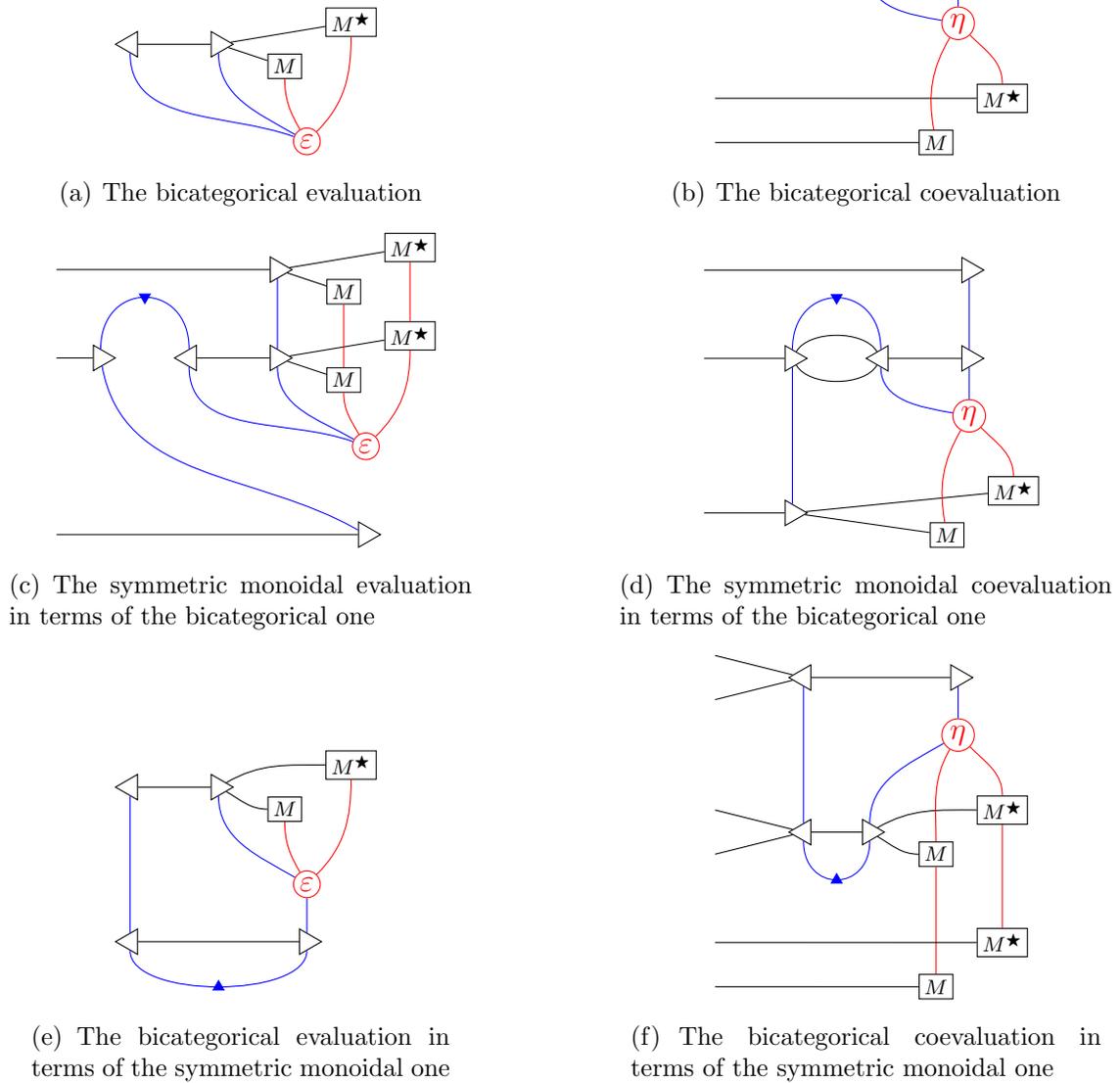
\begin{figure}
  \centering
  \begin{tabular}{c@{\hspace{2cm}}c}
  \subfigure[The bicategorical evaluation]{\label{fig:bicat-eval}
    \hspace{1cm}\begin{tikzpicture}[x={(-1,0)},scale=.6]
      \node[fiber] (2m) at (.5,2) {$M$};
      \node[fiber] (2md) at (-1,3) {$\rdual{M}$};
      \node[pbeflat] (2del) at (2,2.5) {};
      \node[pfeflat] (2del2) at (4,2.5) {};
      \draw (2m) -- (2del);
      \draw (2md) -- (2del);
      \draw (2del) -- (2del2);
      \node[outervert] (ep) at (0,.3) {$\varepsilon$};
      \draw[inner] (2del2) to [out=-90, in =160] (ep);
      \draw[outer] (ep) to[out=45,in=-90] (2md);
      \draw[outer] (ep) to[out=120,in=-90] (2m);
      \draw[inner] (ep) to[out=150,in=-90] (2del);
    \end{tikzpicture}\hspace{1cm}}
  &
  \subfigure[The bicategorical coevaluation]{\label{fig:bicat-coeval}
    \hspace{1cm}\begin{tikzpicture}[x={(-1,0)},y={(0,-1)},scale=.6]
      \node[pbeflat] (1pi) at (0,-1) {};
      \node[pfeflat](1del2) at (2, -1){};
      \draw (1pi) -- (1del2);
      \draw (1del2) -- (5.5, -.5);
      \draw (1del2) -- (5.5, -1.5);  
      \node[fiber] (2m) at (.5,3) {$M$};
      \node[fiber] (2md) at (-1,2) {$\rdual{M}$};
      \draw (2m) -- (5.5, 3);
      \draw (2md) -- (5.5, 2);  
      \node[outervert] (eta) at (0,.3) {$\eta$};
      \draw[inner] (1pi) -- (eta);
      \draw[outer] (eta) to[out=-45,in=90] (2md);
      \draw[outer] (eta) to[out=-120,in=100] (2m);
      \draw[inner] (eta) to[out=170,in=-90] (1del2);
    \end{tikzpicture}\hspace{1cm}}\\
  \subfigure[The symmetric monoidal evaluation in terms of the bicategorical one]{\label{fig:bicat-to-smc-eval}
    \hspace{.5cm}\begin{tikzpicture}[x={(-1,0)},y={(0,-1)},scale=.6]
      \node[fiber] (1m) at (.5,1) {$M$};
      \node[fiber] (1md) at (-1,0) {$\rdual{M}$};
      \node[pbeflat] (1del) at (2,.5) {};
      \draw (1m)-- (1del);
      \draw (1md)-- (1del);
      \draw (1del) -- (7, .5);
      \node[fiber] (2m) at (.5,3) {$M$};
      \node[fiber] (2md) at (-1,2) {$\rdual{M}$};
      \node[pbeflat] (2del) at (2,2.5) {};
      \node[pfeflat] (2del2) at (4,2.5) {};
      \node[pbeflat] (2del3) at (6,2.5) {};
      \draw (2m) -- (2del);
      \draw (2md) -- (2del);
      \draw (2del) -- (2del2);
      \draw (2del3) -- (7, 2.5);
      \draw [inner] (2del3) to [out =90, in =90, looseness=2] node[innerunit] {} (2del2);
      \draw [outer] (1m) --(2m);
      \draw [outer] (1md) -- (2md);
      \draw [inner] (1del) -- (2del);
      \node[outervert] (ep) at (0,4.5) {$\varepsilon$};
      \draw[inner] (2del2) to [out=-90, in =160] (ep);
      \draw[outer] (ep) to[out=45,in=-90] (2md);
      \draw[outer] (ep) to[out=120,in=-90] (2m);
      \draw[inner] (ep) to[out=150,in=-90] (2del);
      \node[pbeflat] (3del) at (0,6.5) {};
      \draw [inner] (2del3) to[out=-80, in= 150] (3del);
     \draw (3del) -- (7, 6.5);
    \end{tikzpicture}\hspace{.5cm}}
  &
  \subfigure[The symmetric monoidal coevaluation in terms of the bicategorical one]{\label{fig:bicat-to-smc-coeval}
    \hspace{1cm}\begin{tikzpicture}[x={(-1,0)},y={(0,-1)},scale=.6]
      \node[pbeflat] (0pi) at (0,-3) {};
      \draw (0pi) -- (6, -3);
      \node[pbeflat] (1pi) at (0,-1) {};
      \node[pfeflat](1del2) at (2, -1){};
      \node[pbeflat] (1del3) at (4,-1) {};
      \draw (1pi) -- (1del2);
      \draw (1del2) to[out = 120,  in = 60, looseness=.8] (1del3);
      \draw (1del2) to[out = -120,  in = -60,  looseness=.8] (1del3);
      \draw[inner] (1del2) to[out = 90,  in = 90, looseness =2] node[innerunit] {} (1del3);
      \draw (1del3) -- (6, -1);
      \node[fiber] (2m) at (.5,3) {$M$};
      \node[fiber] (2md) at (-1,2) {$\rdual{M}$};
      \node[pbeflat] (2del3) at (4,2.5) {};
      \draw (2m) -- (2del3);
      \draw (2md) -- (2del3); 
      \draw (2del3) -- (6, 2.5); 
      \draw[inner](2del3)--(1del3);
      \node[outervert] (eta) at (0,.3) {$\eta$};
      \draw[inner] (0pi) -- (1pi) -- (eta);
      \draw[outer] (eta) to[out=-45,in=90] (2md);
      \draw[outer] (eta) to[out=-120,in=100] (2m);
      \draw[inner] (eta) to[out=170,in=-90] (1del2);
    \end{tikzpicture}\hspace{1cm}}\\
  \subfigure[The bicategorical evaluation in terms of the symmetric monoidal one]{\label{fig:smc-to-bicat-eval}
    \hspace{1cm}\begin{tikzpicture}[x={(-1,0)},scale=.6]
      \node[pbeflat] (1pi) at (0,-1) {};
      \node[pfeflat] (1del2) at (4,-1) {};
      \draw (1pi) -- (1del2);
      \node[fiber] (2m) at (.5,2) {$M$};
      \node[fiber] (2md) at (-1,3) {$\rdual{M}$};
      \node[pbeflat] (2del) at (2,2.5) {};
      \node[pfeflat] (2del2) at (4,2.5) {};
      \draw (2m) to[out=180,in=-30] (2del);
      \draw (2md) to[out=180,in=30] (2del);
      \draw (2del) -- (2del2);
      \draw[inner] (1del2) -- (2del2);
      \node[outervert] (ep) at (0,.3) {$\varepsilon$};
      \draw[inner] (1pi) -- (ep);
      \draw[outer] (ep) to[out=45,in=-90] (2md);
      \draw[outer] (ep) to[out=120,in=-90] (2m);
      \draw[inner] (ep) to[out=150,in=-90] (2del);
      \draw [inner] (1del2) to [out=-90, in =-90, looseness =.7] node[innercounit] {} (1pi);
    \end{tikzpicture}\hspace{1cm}}
  &
  \subfigure[The bicategorical coevaluation in terms of the symmetric monoidal one]{\label{fig:smc-to-bicat-coeval}
    \hspace{1cm}\begin{tikzpicture}[x={(-1,0)},y={(0,-1)},scale=.6]
      \node[pbeflat] (1pi) at (0,-1) {};
      \node[pfeflat](1del2) at (3.5, -1){};
      \draw (1pi) -- (1del2);
      \draw (1del2) -- (5.5, -.5);
      \draw (1del2) -- (5.5, -1.5);  
      \node[fiber] (2m) at (.5,3) {$M$};
      \node[fiber] (2md) at (-1,2) {$\rdual{M}$};
      \node[pbeflat] (2del) at (2,2.5) {};
      \node[pfeflat](2del2) at (3.5, 2.5){};
      \draw (2m) to[out=180,in=-30] (2del);
      \draw (2md) to[out=180,in=30] (2del);
      \draw (2del) -- (2del2);
      \draw (2del2) -- (5.5, 3);
      \draw (2del2) -- (5.5, 2);  
      \draw[inner] (1del2) -- (2del2);
      \draw[inner] (2del) to [out=-90, in =-90, looseness =2] node[innercounit] {} (2del2);
      \node[outervert] (eta) at (0,.3) {$\eta$};
      \draw[inner] (1pi) -- (eta);
      \draw[outer] (eta) to[out=-45,in=90] (2md);
      \draw[outer] (eta) to[out=-120,in=90] (2m);
      \draw[inner] (eta) to[out=-150,in=90] (2del);
      \node[fiber] (3m) at (.5,6) {$M$};
      \node[fiber] (3md) at (-1,5) {$\rdual{M}$};
      \draw[outer] (2m)-- (3m);
      \draw[outer] (2md) -- (3md);
      \draw (3m) -- (5.5, 6);
      \draw (3md) -- (5.5, 5);
    \end{tikzpicture}\hspace{1cm}}
  \end{tabular}
  \caption{The correspondence between symmetric monoidal and
    bicategorical evaluation and coevaluation}
\label{fig:eta-ep-fib-to-smc}
\end{figure}

Thus, it remains only to verify that $\overline{\eta}$ and
$\overline{\ep}$ satisfy the appropriate triangle identities if and only
if $\eta$ and $\ep$ satisfy \emph{their} triangle identities.  We will
verify this for one triangle identity; the other is similar and left to
the reader.  Refer to Figures~\ref{fig:fib-duals}
and~\ref{fig:fib-duals-2}.

In Figure~\ref{fig:fib-duals-step1} we show the composite which the
bicategorical triangle identity asserts to be equal to an identity.  The
morphisms occuring at the top prior to the coevaluation
$\overline{\eta}$ comprise the isomorphism from
Figure~\ref{fig:bicatunit} (spelled out in more detail), using the
definition of the Beck-Chevalley morphism from Figure~\ref{fig:mdbc}.
In Figure~\ref{fig:fib-duals-step2} we have replaced $\overline{\eta}$
and $\overline{\ep}$ with their equivalents in terms of the symmetric
monoidal $\eta$ and $\ep$.  Then in Figure~\ref{fig:fib-duals-step3}, we
slide a number of morphisms past each other, in the way which string
diagram notation makes easy to see.  To get to
Figure~\ref{fig:fib-duals-step4}, we apply a triangle identity for the
adjunction $\Delta_! \dashv \Delta^*$, as in
Figure~\ref{fig:surface-adj-diag}.  In Figure~\ref{fig:fib-duals-step5}
we do some more sliding, and then finally in
Figure~\ref{fig:fib-duals-step6} we apply (the dual of) the curlycue
identity from Figure~\ref{fig:curlycue}.  But
Figure~\ref{fig:fib-duals-step6} is exactly the composite which the
symmetric monoidal triangle identity asserts equal to an identity.
Thus, one triangle identity holds if and only if the other does.
\end{proof}
\afterpage{\begin{landscape}
\begin{figure}
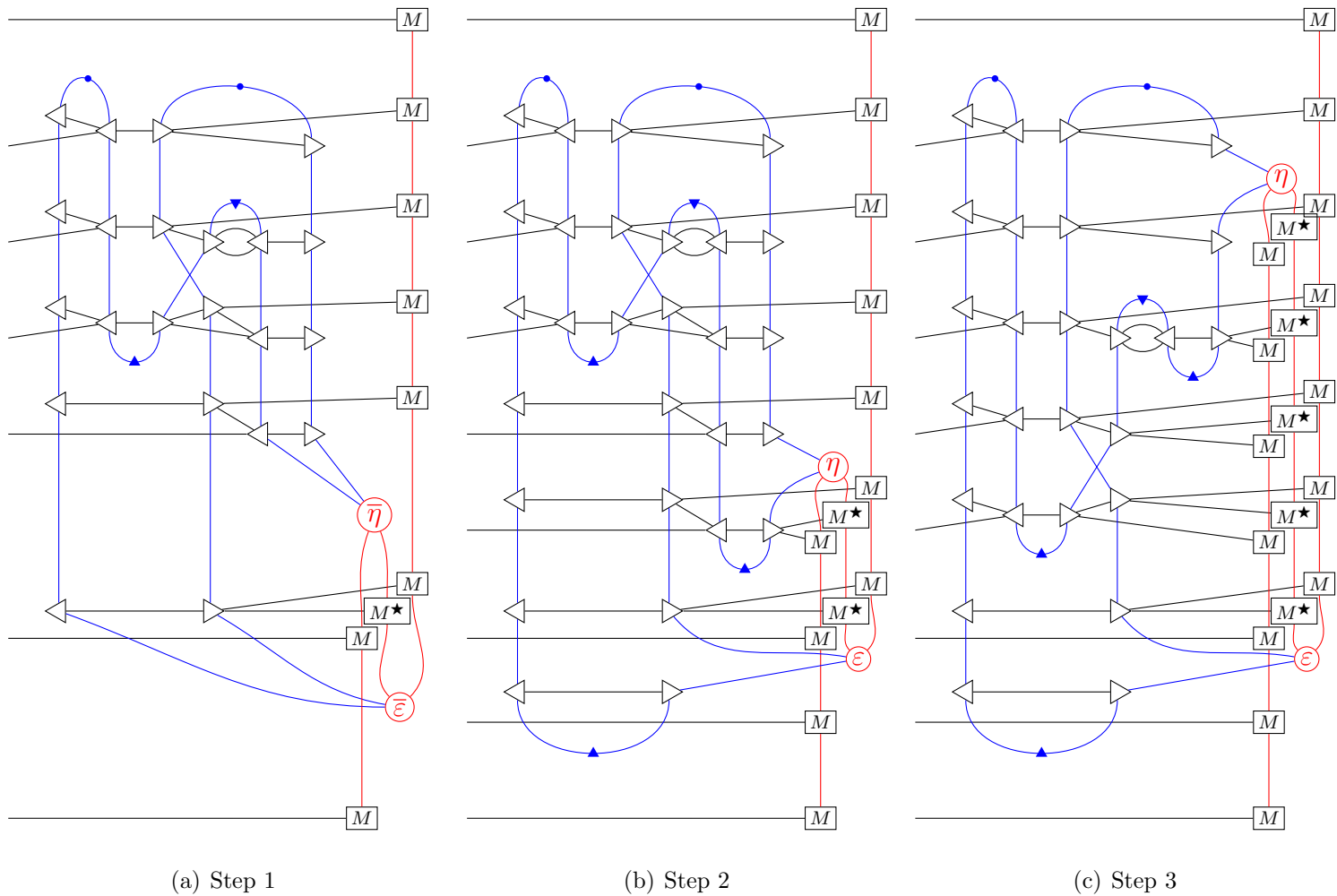

  \centering
  \subfigure[Step 1]{\label{fig:fib-duals-step1}

}
\caption{Comparison of fiberwise duals, steps 1--3}
\label{fig:fib-duals}
\end{figure}

\begin{figure}
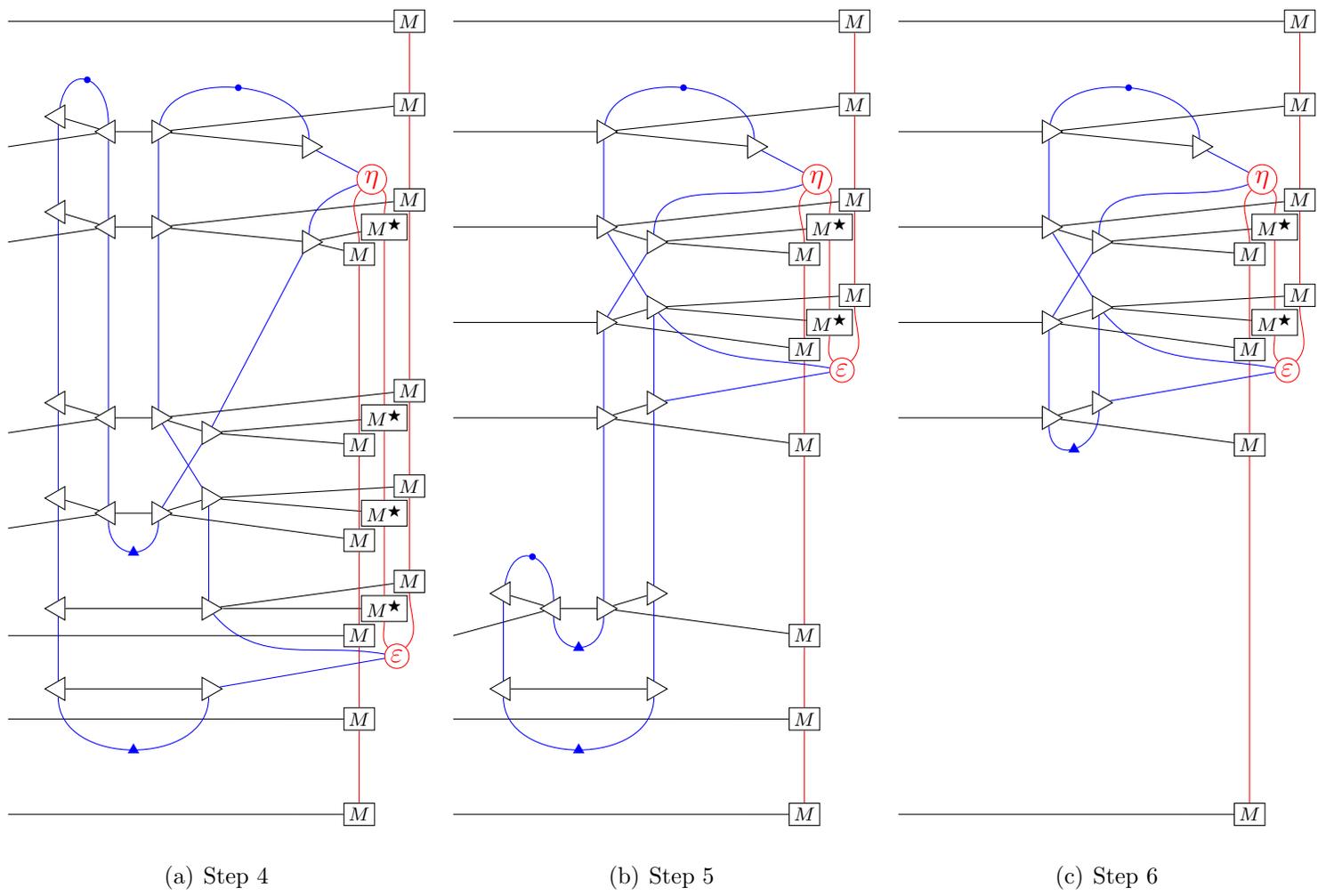

  \centering
  \subfigure[Step 4]{\label{fig:fib-duals-step4}
    %
}
\caption{Comparison of fiberwise duals, steps 4--6}
\label{fig:fib-duals-2}
\end{figure}
\end{landscape}
}

\begin{proof}[of {\normalfont\autoref{thm:fibtrace}\ref{item:fibtrace1}}]
  The definition of $\overline{g}$ is shown in
  Figure~\ref{fig:fibtrace1-gbar}.  We must show that the bicategorical
  trace
  \[ \tr(g) \colon (\pi_A)_! (\Delta_A)^* Q = \sh{Q}\too \sh{P} = P \]
  is equal to the composite
  \begin{equation}
    (\pi_A)_! (\Delta_A)^* Q \xto{(\pi_A)_!
      \tr(\overline{g})} (\pi_A)_!\pi_A^* P \too P.\label{eq:other-trg}
  \end{equation}
  In Figure~\ref{fig:fibtrace1-bicat} we show $\tr(g)$, while in
  Figure~\ref{fig:fibtrace1-symmon} we show~\eqref{eq:other-trg} with
  the definition of $\overline{g}$ substituted.  We can first of all
  consider the parts of these diagrams occurring above $g$ separately
  from those occurring below $g$; it will clearly suffice to show that
  each of these pairs are equal.
  \afterpage{\begin{landscape}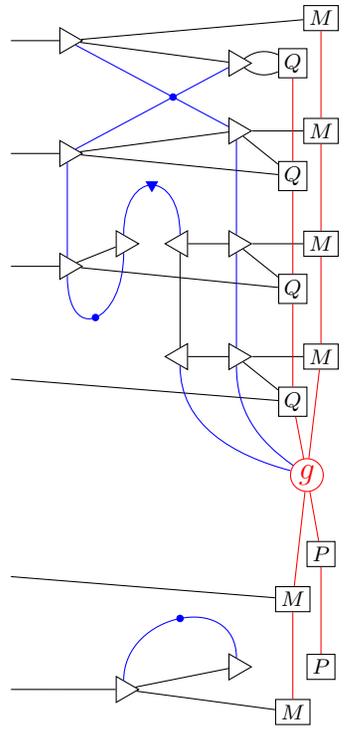
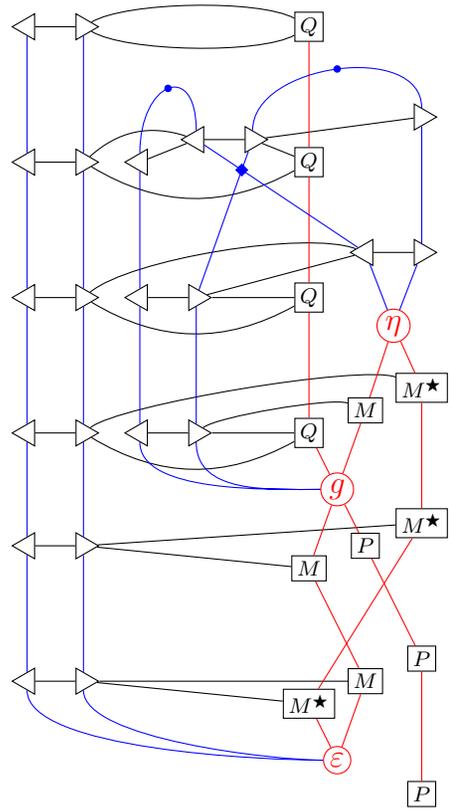
\begin{figure} \centering \subfigure[The
        map $\overline{g}$ in
        \autoref{thm:fibtrace}\ref{item:fibtrace1}]{\label{fig:fibtrace1-gbar}
    \hspace{1cm}\begin{tikzpicture}[strings]
    \def\yy{0}
    \node[fiber] (m\yy) at (-.5,\yy,0) {$M$};
    \node[fiber] (q\yy) at (0,\yy,2) {$Q$};
    \node[pbeflat] (d1\yy) at (1,\yy,2) {};
    \node[pbeflat] (d4\yy) at (4,\yy,1) {};
    \draw (m\yy) -- (d4\yy) -- (5,\yy,1);
    \draw (d1\yy) -- (d4\yy);
    \draw (q\yy) to[bend right] (d1\yy);
    \draw (q\yy) to[bend left] (d1\yy);

    \def\yy{1}
    \node[fiber] (m\yy) at (-.5,\yy,0) {$M$};
    \node[fiber] (q\yy) at (0,\yy,2) {$Q$};
    \node[pbeflat] (d1\yy) at (1,\yy,0) {};
    \node[pbeflat] (d4\yy) at (4,\yy,1) {};
    \draw (m\yy) -- (d1\yy) -- (d4\yy) -- (5,\yy,1);
    \draw (q\yy) -- (d1\yy);
    \draw (q\yy) -- (d4\yy);

    \def\yy{2}
    \node[fiber] (m\yy) at (-.5,\yy,0) {$M$};
    \node[fiber] (q\yy) at (0,\yy,2) {$Q$};
    \node[pbeflat] (d1\yy) at (1,\yy,0) {};
    \node[pfeflat] (d2\yy) at (2,\yy,0) {};
    \node[pbeflat] (d3\yy) at (3,\yy,0) {};
    \node[pbeflat] (d4\yy) at (4,\yy,1) {};
    \draw (m\yy) -- (d1\yy) -- (d2\yy);
    \draw (d3\yy) -- (d4\yy) -- (5,\yy,1);
    \draw (q\yy) -- (d1\yy);
    \draw (q\yy) -- (d4\yy);

    \def\yy{3}
    \node[fiber] (m\yy) at (-.5,\yy,0) {$M$};
    \node[fiber] (q\yy) at (0,\yy,2) {$Q$};
    \node[pbeflat] (d1\yy) at (1,\yy,0) {};
    \node[pfeflat] (d2\yy) at (2,\yy,0) {};
    \draw (m\yy) -- (d1\yy) -- (d2\yy);
    \draw (q\yy) -- (d1\yy);
    \draw (q\yy) -- (5,\yy,1);

    \draw[inner] (d10) -- (d41) -- (d42)
    to[out=-90,in=-90,looseness=3] node[inneriso] {} (d32);
    \draw[inner] (d32) to[out=90,in=90,looseness=3] node[innerunit] {} (d22);
    \draw (d22) -- (d23);
    \draw[inner] (d40) -- (d11) -- (d12) -- (d13);

    \draw[outer] (m0) -- (m1) -- (m2) -- (m3);
    \draw[outer] (q0) -- (q1) -- (q2) -- (q3);

    \begin{scope}[yshift=.5cm]
    \def\yy{5}
    \node[fiber] (p\yy) at (-.5,\yy,0) {$P$};
    \node[fiber] (m\yy) at (0,\yy,2) {$M$};
    \draw (m\yy) -- (5,\yy,1);

    \node[outervert] (g) at (-.25,4.1,1) {$g$};
    \draw[outer] (m3) -- (g);
    \draw[outer] (q3) -- (g);
    \draw[outer] (g) -- (p5);
    \draw[outer] (g) -- (m5);
    \draw[inner] (d13) to[out=-90,in=145] (g);
    \draw[inner] (d23) to[out=-90,in=165] (g);

    \def\yy{6}
    \node[fiber] (p\yy) at (-.5,\yy,0) {$P$};
    \node[fiber] (m\yy) at (0,\yy,2) {$M$};
    \node[pbeflat] (d1\yy) at (1,\yy,0) {};
    \node[pbeflat] (d3\yy) at (3,\yy,1) {};
    \draw (d1\yy) -- (d3\yy);
    \draw (m\yy) -- (d3\yy) -- (5,\yy,1);

    \draw[inner] (d16) to[out=90,in=90,looseness=1.5] node[inneriso] {} (d36);

    \draw[outer] (m5) -- (m6);
    \draw[outer] (p5) -- (p6);
    \node[inneriso] at (intersection of d10--d41 and d40--d11) {};
    \end{scope}
  \end{tikzpicture}\hspace{1cm}}
\hspace{1.5cm}
  \subfigure[The bicategorical $\tr(g)$]{\label{fig:fibtrace1-bicat}
    \begin{tikzpicture}[strings,y={(0,1.2)}]
    \def\yy{0}
    \node[fiber] (q\yy) at (0,\yy,2) {$Q$};
    \node[pbeflat] (d4\yy) at (4,\yy,2) {};
    \node[pfeflat] (d5\yy) at (5,\yy,2) {};
    \draw (q\yy) to[bend right,looseness=.5] (d4\yy);
    \draw (q\yy) to[bend left,looseness=.5] (d4\yy);
    \draw (d4\yy) -- (d5\yy);

    \def\yy{1}
    \node[fiber] (q\yy) at (0,\yy,2) {$Q$};
    \node[pbeflat] (d4\yy) at (4,\yy,2) {};
    \node[pfeflat] (d5\yy) at (5,\yy,2) {};
    \node[pbeflat] (d-2\yy) at (-2,\yy,0) {};
    \node[pbeflat] (d1\yy) at (1,\yy,1) {};
    \node[pfeflat] (d2\yy) at (2,\yy,1) {};
    \node[pfeflat] (d3\yy) at (3,\yy,2) {};
    \draw (q\yy) to[bend left,looseness=1] (d4\yy);
    \draw (d4\yy) -- (d5\yy);
    \draw (q\yy) -- (d1\yy) -- (d2\yy) to[bend right] (d4\yy);
    \draw (d2\yy) -- (d3\yy);
    \draw (d-2\yy) -- (d1\yy);

    \def\yy{2}
    \node[fiber] (q\yy) at (0,\yy,2) {$Q$};
    \node[pbeflat] (d4\yy) at (4,\yy,2) {};
    \node[pfeflat] (d5\yy) at (5,\yy,2) {};
    \node[pbeflat] (d-2\yy) at (-2,\yy,0) {};
    \node[pfeflat] (d-1\yy) at (-1,\yy,0) {};
    \node[pbeflat] (d2\yy) at (2,\yy,2) {};
    \node[pfeflat] (d3\yy) at (3,\yy,2) {};
    \draw (q\yy) to[bend left,looseness=1] (d4\yy);
    \draw (d4\yy) -- (d5\yy);
    \draw (q\yy) -- (d2\yy) -- (d3\yy);
    \draw (d-2\yy) -- (d-1\yy) -- (d2\yy);
    \draw (d-1\yy) to[bend right,looseness=.5] (d4\yy);

    \def\yy{3}
    \node[fiber] (q\yy) at (0,\yy,2) {$Q$};
    \node[pbeflat] (d4\yy) at (4,\yy,2) {};
    \node[pfeflat] (d5\yy) at (5,\yy,2) {};
    \node[fiber] (m\yy) at (-1,\yy,1) {$M$};
    \node[fiber] (md\yy) at (-2,\yy,0) {$\rdual{M}$};
    \node[pbeflat] (d2\yy) at (2,\yy,2) {};
    \node[pfeflat] (d3\yy) at (3,\yy,2) {};
    \draw (q\yy) to[bend left,looseness=1] (d4\yy);
    \draw (d4\yy) -- (d5\yy);
    \draw (q\yy) -- (d2\yy) -- (d3\yy);
    \draw (m\yy) to[bend right,looseness=.3] (d2\yy);
    \draw (md\yy) to[bend right,looseness=.3] (d4\yy);

    \def\yy{4}
    \node[fiber] (m\yy) at (0,\yy,2) {$M$};
    \node[fiber] (p\yy) at (-1,\yy,1) {$P$};
    \node[fiber] (md\yy) at (-2,\yy,0) {$\rdual{M}$};
    \node[pbeflat] (d4\yy) at (4,\yy,1) {};
    \node[pfeflat] (d5\yy) at (5,\yy,1) {};
    \draw (m\yy) -- (d4\yy) -- (d5\yy);
    \draw (md\yy) -- (d4\yy);

    \def\yy{5}
    \node[fiber] (m\yy) at (-1,\yy,1) {$M$};
    \node[fiber] (p\yy) at (-2,\yy,0) {$P$};
    \node[fiber] (md\yy) at (0,\yy,2) {$\rdual{M}$};
    \node[pbeflat] (d4\yy) at (4,\yy,1) {};
    \node[pfeflat] (d5\yy) at (5,\yy,1) {};
    \draw (m\yy) -- (d4\yy) -- (d5\yy);
    \draw (md\yy) -- (d4\yy);

    \def\yy{6}
    \node[fiber] (p\yy) at (-2,\yy,0) {$P$};

    \node[outervert] (eta) at (-1.5,2.5,.25) {$\eta$};
    \node[outervert] (g) at (-.5,3.5,1.5) {$g$};
    \node[outervert] (ep) at (-.5,5.5,1.5) {$\ep$};
    \draw[outer] (q0) -- (q1) -- (q2) -- (q3) --
      (g) -- (p4) -- (p5) -- (p6);
    \draw[outer] (eta) -- (m3) -- (g) -- (m4) --
      (m5) -- (ep) -- (md5) -- (md4) -- (md3) -- (eta);

    \draw[inner] (d50) -- (d51) -- (d52) -- (d53) --
      (d54) -- (d55) to[out=-90,in=180,looseness=.5] (ep);
    \draw[inner] (d40) -- (d41) -- (d42) -- (d43) --
      (d44) -- (d45) to[out=-90,in=180,looseness=.5] (ep);
    \draw[inner] (d31) -- (d32) -- (d33)
    to[out=-90,in=180,looseness=.7] (g)
    to[out=180,in=-90] (d23);
    \draw[inner] (d23) -- (d22) -- (d11)
    to[out=90,in=90,looseness=1] node[inneriso] {} (d-21);
    \draw[inner] (d-21) -- (d-22) -- (eta) -- (d-12) -- (d21)
    to[out=90,in=90,looseness=3] node[inneriso] {} (d31);
    \node[innerbc,inner] at (intersection of d22--d11 and d21--d-12) {};
  \end{tikzpicture}}
\caption{Two morphisms from the proof of \autoref{thm:fibtrace}\ref{item:fibtrace1}}
\end{figure}
\end{landscape}
\begin{figure}
  \centering
  \begin{tikzpicture}[strings]
    \def\yy{-3}
    \node[fiber] (q\yy) at (0,\yy,0) {$Q$};
    \node[pbeflat] (d1\yy) at (1,\yy,0) {};
    \node[pfeflat] (d6\yy) at (6,\yy,0) {};
    \draw (d1\yy) -- (d6\yy);
    \draw (q\yy) to[bend right] (d1\yy);
    \draw (q\yy) to[bend left] (d1\yy);

    \def\yy{-2}
    \node[fiber] (q\yy) at (0,\yy,1) {$Q$};
    \node[pbeflat] (d1\yy) at (1,\yy,1) {};
    \node[pbeflat] (d-1\yy) at (-1.5,\yy,0) {};
    \node[pbeflat] (d5\yy) at (5,\yy,0) {};
    \node[pfeflat] (d6\yy) at (6,\yy,0) {};
    \draw (d-1\yy) -- (d5\yy) -- (d6\yy);
    \draw (d1\yy) to[bend left,looseness=.1] (d5\yy);
    \draw (q\yy) to[bend right] (d1\yy);
    \draw (q\yy) to[bend left] (d1\yy);

    \def\yy{-1}
    \node[fiber] (md\yy) at (-2,\yy,0) {$\rdual{M}$};
    \node[fiber] (m\yy) at (-1,\yy,1) {$M$};
    \node[fiber] (q\yy) at (0,\yy,2) {$Q$};
    \node[pbeflat] (d1\yy) at (1,\yy,2) {};
    \node[pbeflat] (d4\yy) at (4,\yy,0) {};
    \node[pbeflat] (d5\yy) at (5,\yy,0) {};
    \node[pfeflat] (d6\yy) at (6,\yy,0) {};
    \draw (m\yy) -- (d4\yy) -- (d5\yy) -- (d6\yy);
    \draw (d1\yy) to[bend left,looseness=.1] (d5\yy);
    \draw (q\yy) to[bend right] (d1\yy);
    \draw (q\yy) to[bend left] (d1\yy);
    \draw (md\yy) -- (d4\yy);

    \def\yy{0}
    \node[fiber] (md\yy) at (-2,\yy,0) {$\rdual{M}$};
    \node[fiber] (m\yy) at (-1,\yy,1) {$M$};
    \node[fiber] (q\yy) at (0,\yy,2) {$Q$};
    \node[pbeflat] (d1\yy) at (1,\yy,2) {};
    \node[pbeflat] (d4\yy) at (4,\yy,1) {};
    \node[pbeflat] (d5\yy) at (5,\yy,0) {};
    \node[pfeflat] (d6\yy) at (6,\yy,0) {};
    \draw (m\yy) -- (d4\yy) -- (d5\yy);
    \draw (d1\yy) -- (d4\yy);
    \draw (q\yy) to[bend right] (d1\yy);
    \draw (q\yy) to[bend left] (d1\yy);
    \draw (md\yy) -- (d5\yy) -- (d6\yy);

    \def\yy{1}
    \node[fiber] (md\yy) at (-2,\yy,0) {$\rdual{M}$};
    \node[fiber] (m\yy) at (-1,\yy,1) {$M$};
    \node[fiber] (q\yy) at (0,\yy,2) {$Q$};
    \node[pbeflat] (d1\yy) at (1,\yy,1) {};
    \node[pbeflat] (d4\yy) at (4,\yy,1) {};
    \node[pbeflat] (d5\yy) at (5,\yy,0) {};
    \node[pfeflat] (d6\yy) at (6,\yy,0) {};
    \draw (m\yy) -- (d1\yy) -- (d4\yy) -- (d5\yy);
    \draw (q\yy) -- (d1\yy);
    \draw (q\yy) to[bend left,looseness=.1] (d4\yy);
    \draw (md\yy) -- (d5\yy) -- (d6\yy);

    \def\yy{2}
    \node[fiber] (md\yy) at (-2,\yy,0) {$\rdual{M}$};
    \node[fiber] (m\yy) at (-1,\yy,1) {$M$};
    \node[fiber] (q\yy) at (0,\yy,2) {$Q$};
    \node[pbeflat] (d1\yy) at (1,\yy,1) {};
    \node[pfeflat] (d2\yy) at (2,\yy,1) {};
    \node[pbeflat] (d3\yy) at (3,\yy,1) {};
    \node[pbeflat] (d4\yy) at (4,\yy,1) {};
    \node[pbeflat] (d5\yy) at (5,\yy,0) {};
    \node[pfeflat] (d6\yy) at (6,\yy,0) {};
    \draw (m\yy) -- (d1\yy) -- (d2\yy);
    \draw (d3\yy) -- (d4\yy) -- (d5\yy);
    \draw (q\yy) -- (d1\yy);
    \draw (q\yy) to[bend left, looseness=.2] (d4\yy);
    \draw (md\yy) -- (d5\yy) -- (d6\yy);

    \def\yy{3}
    \node[fiber] (md\yy) at (-2,\yy,0) {$\rdual{M}$};
    \node[fiber] (m\yy) at (-1,\yy,1) {$M$};
    \node[fiber] (q\yy) at (0,\yy,2) {$Q$};
    \node[pbeflat] (d1\yy) at (1,\yy,1) {};
    \node[pfeflat] (d2\yy) at (2,\yy,1) {};
    \node[pbeflat] (d5\yy) at (5,\yy,0) {};
    \node[pfeflat] (d6\yy) at (6,\yy,0) {};
    \draw (m\yy) -- (d1\yy) -- (d2\yy);
    \draw (q\yy) -- (d1\yy);
    \draw (q\yy) to[bend left,looseness=.2] (d5\yy);
    \draw (md\yy) -- (d5\yy) -- (d6\yy);

    \def\yy{4}
    \node[fiber] (md\yy) at (-2,\yy,0) {$\rdual{M}$};
    \node[fiber] (p\yy) at (-1,\yy,1) {$P$};
    \node[fiber] (m\yy) at (0,\yy,2) {$M$};
    \node[pbeflat] (d5\yy) at (5,\yy,0) {};
    \node[pfeflat] (d6\yy) at (6,\yy,0) {};
    \draw (m\yy) -- (d5\yy);
    \draw (md\yy) -- (d5\yy) -- (d6\yy);

    \def\yy{5}
    \node[fiber] (md\yy) at (-2,\yy,0) {$\rdual{M}$};
    \node[fiber] (p\yy) at (-1,\yy,1) {$P$};
    \node[fiber] (m\yy) at (0,\yy,2) {$M$};
    \node[pbeflat] (d1\yy) at (1,\yy,1) {};
    \node[pbeflat] (d3\yy) at (3,\yy,1) {};
    \node[pbeflat] (d5\yy) at (5,\yy,0) {};
    \node[pfeflat] (d6\yy) at (6,\yy,0) {};
    \draw (d1\yy) -- (d3\yy);
    \draw (m\yy) to[bend left,looseness=.1] (d3\yy);
    \draw (d3\yy) -- (d5\yy);
    \draw (md\yy) -- (d5\yy) -- (d6\yy);

    \def\yy{6}
    \node[fiber] (md\yy) at (0,\yy,2) {$\rdual{M}$};
    \node[fiber] (p\yy) at (-2,\yy,0) {$P$};
    \node[fiber] (m\yy) at (-1,\yy,1) {$M$};
    \node[pbeflat] (d1\yy) at (1,\yy,0) {};
    \node[pbeflat] (d3\yy) at (3,\yy,0) {};
    \node[pbeflat] (d5\yy) at (5,\yy,1) {};
    \node[pfeflat] (d6\yy) at (6,\yy,1) {};
    \draw (d1\yy) -- (d3\yy);
    \draw (m\yy) to[out=180,in=-30,looseness=.2] (d3\yy);
    \draw (d3\yy) -- (d5\yy);
    \draw (md\yy) -- (d5\yy) -- (d6\yy);

    \def\yy{7}
    \node[fiber] (md\yy) at (0,\yy,2) {$\rdual{M}$};
    \node[fiber] (p\yy) at (-2,\yy,0) {$P$};
    \node[fiber] (m\yy) at (-1,\yy,1) {$M$};
    \node[pbeflat] (d1\yy) at (1,\yy,0) {};
    \node[pbeflat] (d3\yy) at (3,\yy,1) {};
    \node[pbeflat] (d5\yy) at (5,\yy,1) {};
    \node[pfeflat] (d6\yy) at (6,\yy,1) {};
    \draw (d1\yy) to[bend right,looseness=.1] (d5\yy);
    \draw (m\yy) -- (d3\yy);
    \draw (md\yy) -- (d3\yy) -- (d5\yy) -- (d6\yy);

    \def\yy{8}
    \node[fiber] (p\yy) at (-2,\yy,0) {$P$};
    \node[pbeflat] (d0\yy) at (0,\yy,2) {};
    \node[pbeflat] (d1\yy) at (1,\yy,0) {};
    \node[pbeflat] (d5\yy) at (5,\yy,1) {};
    \node[pfeflat] (d6\yy) at (6,\yy,1) {};
    \draw (d1\yy) -- (d5\yy);
    \draw (d0\yy) -- (d5\yy) -- (d6\yy);

    \def\yy{9}
    \node[fiber] (p\yy) at (-2,\yy,0) {$P$};
    \node[pbeflat] (d1\yy) at (1,\yy,0) {};
    \node[pfeflat] (d6\yy) at (6,\yy,1) {};
    \draw (d1\yy) -- (d6\yy);

    \def\yy{10}
    \node[fiber] (p\yy) at (-2,\yy,0) {$P$};

    \node[outervert] (eta) at (-1.5,-1.5,.5) {$\eta$};
    \node[outervert] (g) at (-.5,3.5,1.3) {$g$};
    \node[outervert] (ep) at (-.5,7.5,1.5) {$\ep$};

    \begin{scope}[inner]
      \draw (d1-3) -- (d1-2) -- (d1-1) -- (d10) -- (d41) -- (d42)
      to[out=-90,in=-90,looseness=3] node[inneriso] {} (d32);
      \draw (d32) to[out=90,in=90,looseness=3] node[innerunit] {} (d22);
      \draw (d22) -- (d23);
      \draw (d5-2) -- (d5-1) -- (d40) -- (d11) -- (d12) -- (d13);
      \draw (d13) to[out=-90,in=145] (g);
      \draw (d23) to[out=-90,in=165] (g);
      \draw (d15) to[out=90,in=90,looseness=1.5] node[inneriso] {} (d35);
      \draw (d4-1) -- (d50) -- (d51) -- (d52) -- (d53) --
      (d54) -- (d55) -- (d56) -- (d37);
      \draw (d6-3) -- (d6-2) -- (d6-1) -- (d60) -- (d61) -- (d62) --
      (d63) -- (d64) -- (d65) -- (d66) -- (d67) -- (d68) -- (d69);
      \draw (d5-2) to[out=90,in=90,looseness=.5] node[inneriso] {} (d-1-2);
      \draw (d-1-2) -- (eta);
      \draw (eta) to[bend right,looseness=.5] (d4-1);
      \draw (d58) to[out=-90,in=-90,looseness=.4] node[inneriso] {} (d08);
      \draw (d69) to[out=-90,in=-90,looseness=.5] node[innercounit] {} (d19);
      \draw (d19) -- (d18) -- (d17) -- (d16) -- (d15);
      \draw (d35) -- (d36) -- (d57) -- (d58);
      \draw (d37) to[bend right,looseness=.5] (ep);
      \draw (ep) -- (d08);
    \end{scope}
    \begin{scope}[outer]
      \draw (m-1) -- (m0) -- (m1) -- (m2) -- (m3) -- (g) --
      (p4) -- (p5) -- (p6) -- (p7) -- (p8) -- (p9) -- (p10);
      \draw (md-1) -- (md0) -- (md1) -- (md2) -- (md3) --
      (md4) -- (md5) -- (md6) -- (md7);
      \draw (q-3) -- (q-2) -- (q-1) -- (q0) -- (q1) -- (q2) --
      (q3) -- (g) -- (m4) -- (m5) -- (m6) -- (m7);
      \draw (eta) to[bend right] (m-1);
      \draw (eta) to[bend left] (md-1);
      \draw (md7) to[bend right] (ep);
      \draw (m7) to[bend left] (ep);
    \end{scope}
  \end{tikzpicture}
  \caption{The symmetric monoidal \mbox{$\epsilon \circ (\pi_!)_!\tr(\overline{g})$}}
  \label{fig:fibtrace1-symmon}
\end{figure}}

In Figure~\ref{fig:fibtrace1-lower} we treat the lower parts.
Figure~\ref{fig:fibtrace1-lower-a} is the lower part of
Figure~\ref{fig:fibtrace1-symmon}.  To obtain
Figure~\ref{fig:fibtrace1-lower-b}, we change the isomorphism
\[ \Delta^* (\id\times\pi^*)\pi^* \cong \pi^* \] from one which
contracts $\Delta$ with the second $\pi$ to one which contracts it with
the first.  The two are equal by pseudofunctor coherence, since both are
induced by the equality $\pi (\id\times \pi)\Delta = \pi$.

Then in Figure~\ref{fig:fibtrace1-lower-c} we slide $\ep$ and the
symmetry down to the bottom.  This yields the lower part of
Figure~\ref{fig:fibtrace1-bicat} (recalling the relationship of the
symmetric monoidal and bicategorical evaluations) together with a
composite of pseudofunctoriality isomorphisms on top.  But by
pseudofunctor coherence, this composite is equal to the identity.
\afterpage{\begin{landscape}
\begin{figure}
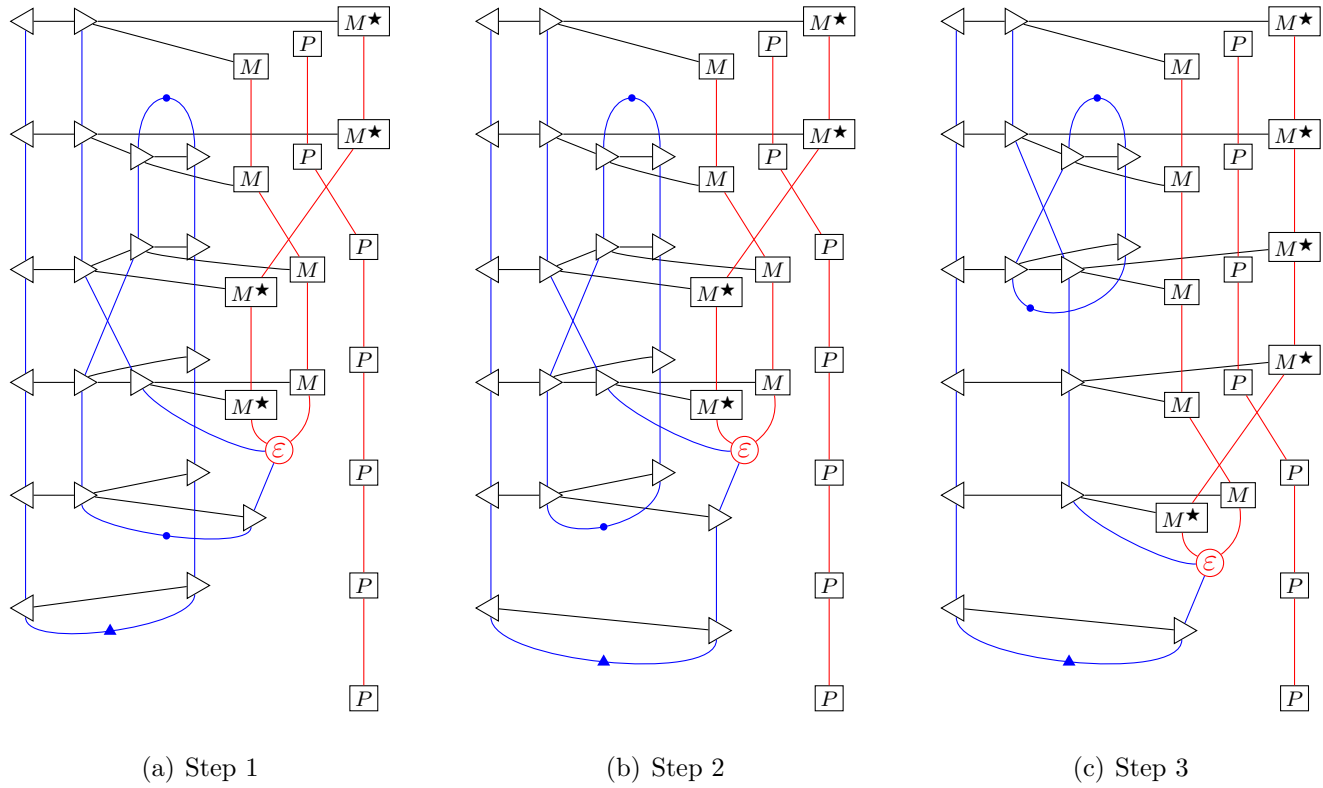

  \centering
  \subfigure[Step 1]{\label{fig:fibtrace1-lower-a}
}
  \caption{Lower part, from the proof of \autoref{thm:fibtrace}\ref{item:fibtrace1}}
  \label{fig:fibtrace1-lower}
\end{figure}

\begin{figure}
  \centering
  \subfigure[Step 1]{\label{fig:fibtrace1-upper-a}%
}
  \caption{Upper part I, from the proof of \autoref{thm:fibtrace}\ref{item:fibtrace1}}
  \label{fig:fibtrace1-upperI}
\end{figure}

\begin{figure}
  \centering
  \subfigure[Step 1]{\label{fig:fibtrace1-usub-1}%
}
  \caption{Upper part II, from the proof of \autoref{thm:fibtrace}\ref{item:fibtrace1}}
\end{figure}

\begin{figure}
  \centering
  \subfigure[Step 4]{\label{fig:fibtrace1-usub-4}%
}
  \caption{Upper part III, from the proof of \autoref{thm:fibtrace}\ref{item:fibtrace1}}
\end{figure}

\begin{figure}
  \centering
  \subfigure[Step 7]{\label{fig:fibtrace1-usub-7}%
}
  \caption{Upper part IV, from the proof of \autoref{thm:fibtrace}\ref{item:fibtrace1}}
\end{figure}
\end{landscape}
}

Next, in Figure~\ref{fig:fibtrace1-upperI} we begin considering the
upper parts.  In Figure~\ref{fig:fibtrace1-upper-a} we show the upper
part of Figure~\ref{fig:fibtrace1-symmon}, with the symmetric monoidal
coevaluation written in terms of the bicategorical one.  Then in
Figure~\ref{fig:fibtrace1-upper-b} we slide the bicategorical $\eta$ all
the way to the bottom (directly above $g$, which divides the upper from
the lower parts).
  
Since $\eta$ also occurs directly above $g$ in
Figure~\ref{fig:fibtrace1-bicat}, it suffices to compare the parts of
Figures~\ref{fig:fibtrace1-bicat} and~\ref{fig:fibtrace1-upper-b}
occurring above $\eta$.  Moreover, we may also ignore the $Q$ node and
the $(\pi_A)_!$ node on the far left, since neither plays any role in
either of these diagrams.

Now, in Figure~\ref{fig:fibtrace1-usub-1} we have copied the part of
Figure~\ref{fig:fibtrace1-bicat} above $\eta$, with $Q$ and $(\pi_A)_!$
omitted and the definition of the Frobenius morphism substituted.  In
Figure~\ref{fig:fibtrace1-usub-2} we simply slide, and then to get to
Figure~\ref{fig:fibtrace1-usub-3} we apply the broken zigzag identity
(Figure~\ref{fig:curlycue-mate}).

In Figure~\ref{fig:fibtrace1-usub-4} we replace the isomorphism at the
very top by a composite of two other isomorphisms, which is equal to it
by pseudofunctor coherence.  Figure~\ref{fig:fibtrace1-usub-5} is a
simple slide, and then in Figure~\ref{fig:fibtrace1-usub-6} we replace
the composite of two associativity isomorphisms by the composite of
three (the ``pentagon identity,'' which also holds by pseudofunctor
coherence).

In Figure~\ref{fig:fibtrace1-usub-7} we slide the lower of these
associativities past a unit, and finally in
Figure~\ref{fig:fibtrace1-usub-8} we compose two isomorphisms at the
bottom to obtain a different one (again by pseudofunctor coherence).
The result is exactly the part of Figure~\ref{fig:fibtrace1-upper-b}
above $\eta$, with $Q$ and $(\pi_A)_!$ omitted.
\end{proof}

\begin{proof}[of {\normalfont\autoref{thm:fibtrace}\ref{item:fibtrace2}}]
  We define $\widetilde{f}$ to be the composite
  \[ (\Delta_A)_! Q \odot \widehat{M} \xto{\cong}
  Q \otimes M \xto{f}
  M \otimes P \to
  M \otimes (\pi_A)^* (\pi_A)_! P \xto{\cong}
  \widehat{M} \odot (\pi_A)_! P.
  \]
  This is pictured in Figure~\ref{fig:ftilde}.
  We now claim that the following square commutes.
  \begin{equation}
    \xymatrix{ Q\otimes M \ar[r] \ar[d]_f &
      (\Delta_A)^* (\Delta_A)_! Q \otimes M\ar[d]^{\overline{\widetilde{f}}}\\
      M\otimes P \ar[r] & M\otimes (\pi_A)^* (\pi_A)_! P}\label{eq:ftilde-square}
  \end{equation}
  To show this, we substitute $\widetilde{f}$ from
  Figure~\ref{fig:ftilde} as $g$ in Figure~\ref{fig:fibtrace1-gbar}
  (replacing $Q$ by $\Delta_!Q$ and $P$ by $\pi_! P$).  Canceling the
  isomorphism $\Delta^* (\id\times \pi)^* \cong \Id$ with its inverse,
  we see that what is left at the bottom of the resulting diagram is
  exactly the left-bottom composite of~\eqref{eq:ftilde-square}.
  Therefore, it suffices to show that what remains above this becomes
  the identity when composed with the top arrow
  in~\eqref{eq:ftilde-square}.  By inverting the first two transitions
  in Figure~\ref{fig:ftilde}, this is equivalent to the equality shown
  in Figure~\ref{fig:ftilde2}.  However, this follows by a simple
  sliding and the broken zigzag identity from
  Figure~\ref{fig:curlycue-mate}.
  \begin{figure}
  \centering
  \begin{tikzpicture}[strings]
    \def\yy{-1}
    \node[fiber] (m\yy) at (-.5,\yy,0) {$M$};
    \node[fiber] (q\yy) at (0,\yy,2) {$Q$};
    \node[pfeflat] (d3\yy) at (4,\yy,2) {};
    \node[pbeflat] (d4\yy) at (5,\yy,1) {};
    \node[pfeflat] (d5\yy) at (6,\yy, 1){};
    \draw (m\yy) -- (d4\yy) -- (d5\yy);
    \draw (q\yy)--(d3\yy) -- (d4\yy);
     \draw (d3\yy)--(7, \yy,3);

    \def\yy{0}
    \node[fiber] (m\yy) at (-.5,\yy,0) {$M$};
    \node[fiber] (q\yy) at (0,\yy,2) {$Q$};
    \node[pbeflat] (d3\yy) at (4,\yy,1.5) {};
    \node[pfeflat] (d4\yy) at (5,\yy,1.5) {};
     \node[pfeflat] (d5\yy) at (6,\yy, 1){};
     \draw (m\yy) -- (d3\yy);
     \draw (d4\yy)  -- (d5\yy);
    \draw (q\yy)--(d3\yy) -- (d4\yy) --(7, \yy,3);

    \def\yy{1}
    \node[fiber] (m\yy) at (-.5,\yy,0) {$M$};
    \node[fiber] (q\yy) at (0,\yy,2) {$Q$};
    \node[pbeflat] (d4\yy) at (4,\yy,1) {};
    \draw (m\yy) -- (d4\yy) --(7, \yy,1);
    \draw (q\yy) -- (d4\yy);

    \def\yy{3}
    \node[fiber] (p\yy) at (-.5,\yy,0) {$P$};
    \node[fiber] (m\yy) at (0,\yy,2) {$M$};
    \node[pbeflat] (d4\yy) at (4,\yy,1) {};
    \draw (m\yy) -- (d4\yy) --(7, \yy,1);
    \draw (p\yy) -- (d4\yy);

    \node[outervert] (f) at (-.25,2.1,1) {$f$};
    \draw[outer] (m1) -- (f);
    \draw[outer] (q1) -- (f);
    \draw[outer] (f) -- (p3);
    \draw[outer] (f) -- (m3);

    \def\yy{4}
    \node[fiber] (p\yy) at (-.5,\yy,0) {$P$};
    \node[fiber] (m\yy) at (0,\yy,2) {$M$};
    \node[pfeflat] (d2\yy) at (2,\yy,0) {};
    \node[pbeflat] (d3\yy) at (3,\yy,0) {};
    \node[pbeflat] (d4\yy) at (4,\yy,1) {};
    \draw (m\yy) -- (d4\yy) --(7, \yy,1);
    \draw (p\yy) -- (d2\yy);
     \draw(d3\yy)--(d4\yy);

    \draw[outer] (q-1) -- (q0) -- (q1);

    \begin{scope}[yshift=.5cm]

    \def\yy{5}
    \node[fiber] (p\yy) at (-.5,\yy,0) {$P$};
    \node[fiber] (m\yy) at (0,\yy,2) {$M$};
    \node[pfeflat] (d2\yy) at (2,\yy,0) {};

    \draw (m\yy) -- (7, \yy,2);
    \draw (p\yy) -- (d2\yy);

    \draw[inner] (d30) -- (d41)-- (f) -- (d43)--(d44) to[out=-90,in=-90,looseness=2] node[inneriso] {} (d34);
    \draw[inner] (d5-1) -- (d50) to[out=-90,in=-90,looseness=2] node[inneriso] {} (d40);
    \draw (d40) -- (d3-1);
    \draw[inner] (d4-1) -- (d30);
    \node[innerbc] at (intersection of d3-1--d40 and d4-1--d30) {};
    \draw[inner](d24)--(d25);
    \draw[outer] (m-1) -- (m0) -- (m1);
    \draw[outer] (m3) -- (m4)--(m5);
    \draw[outer] (p3)--(p4)--(p5);

    \draw[inner] (d34) to[out=90,in=90,looseness=1.5] node[innerunit] {} (d24);

    \end{scope}
  \end{tikzpicture}
  \caption{The map $\tilde{f}$ in \autoref{thm:fibtrace}\ref{item:fibtrace2}}
  \label{fig:ftilde}
\end{figure}
\begin{figure}
  \centering
  \begin{tikzpicture}[strings]
    \def\yy{0}
    \node[fiber] (m\yy) at (-.5,\yy,0) {$M$};
    \node[fiber] (q\yy) at (0,\yy,2) {$Q$};
    \node[pbeflat] (d3\yy) at (3,\yy,1) {};
    \draw (m\yy) -- (d3\yy) -- (6,\yy,1);
    \draw (q\yy) -- (d3\yy);

    \def\yy{1}
    \node[fiber] (m\yy) at (-.5,\yy,0) {$M$};
    \node[fiber] (q\yy) at (0,\yy,2) {$Q$};
    \node[pbeflat] (d3\yy) at (3,\yy,1) {};
    \node[pfeflat] (d4\yy) at (4,\yy,1) {};
    \node[pfeflat] (d5\yy) at (5,\yy,0) {};
    \draw (m\yy) -- (d3\yy) -- (d4\yy) -- (6,\yy,1);
    \draw (q\yy) -- (d3\yy);
    \draw (d4\yy) -- (d5\yy);

    \def\yy{2}
    \node[fiber] (m\yy) at (-.5,\yy,0) {$M$};
    \node[fiber] (q\yy) at (0,\yy,2) {$Q$};
    \node[pfeflat] (d1\yy) at (1,\yy,2) {};
    \node[pbeflat] (d2\yy) at (2,\yy,2) {};
    \node[pbeflat] (d3\yy) at (3,\yy,1) {};
    \node[pfeflat] (d4\yy) at (4,\yy,1) {};
    \node[pfeflat] (d5\yy) at (5,\yy,0) {};
    \draw (m\yy) -- (d3\yy) -- (d4\yy) -- (6,\yy,1);
    \draw (q\yy) -- (d1\yy);
    \draw (d2\yy) -- (d3\yy);
    \draw (d1\yy) to[bend right] (d2\yy);
    \draw (d1\yy) to[bend left] (d2\yy);
    \draw (d4\yy) -- (d5\yy);

    \def\yy{3}
    \node[fiber] (m\yy) at (-.5,\yy,0) {$M$};
    \node[fiber] (q\yy) at (0,\yy,2) {$Q$};
    \node[pfeflat] (d1\yy) at (1,\yy,2) {};
    \node[pbeflat] (d2\yy) at (2,\yy,0) {};
    \node[pbeflat] (d3\yy) at (3,\yy,1) {};
    \node[pfeflat] (d4\yy) at (4,\yy,1) {};
    \node[pfeflat] (d5\yy) at (5,\yy,0) {};
    \draw (m\yy) -- (d2\yy);
    \draw (q\yy) -- (d1\yy);
    \draw (d2\yy) -- (d3\yy) -- (d4\yy) -- (6,\yy,1);
    \draw (d1\yy) -- (d2\yy);
    \draw (d1\yy) -- (d3\yy);
    \draw (d4\yy) -- (d5\yy);

    \def\yy{4}
    \node[fiber] (m\yy) at (-.5,\yy,0) {$M$};
    \node[fiber] (q\yy) at (0,\yy,2) {$Q$};
    \node[pfeflat] (d1\yy) at (1,\yy,2) {};
    \node[pbeflat] (d2\yy) at (2,\yy,0) {};
    \node[pfeflat] (d5\yy) at (5,\yy,0) {};
    \draw (m\yy) -- (d2\yy) -- (d5\yy);
    \draw (q\yy) -- (d1\yy) -- (6,\yy,2);
    
    \begin{scope}[inner]
      \draw (d30) -- (d31) -- (d32) -- (d23) -- (d24);
      \draw (d33) -- (d22);
      \draw (d22) to[out=90,in=90,looseness=3] node[innerunit] {} (d12);
      \draw (d12) -- (d13) -- (d14);
      \draw (d33) to[out=-90,in=-90,looseness=2] node[innercounit] {} (d43);
      \draw (d43) -- (d42) -- (d41);
      \draw (d41) to[out=90,in=90,looseness=3] node[inneriso] {} (d51);
      \draw (d51) -- (d52) -- (d53) -- (d54);
      \node[inneriso] at (intersection of d22--d33 and d32--d23) {};
    \end{scope}
    \begin{scope}[outer]
      \draw (m0) -- (m1) -- (m2) -- (m3) -- (m4);
      \draw (q0) -- (q1) -- (q2) -- (q3) -- (q4);
    \end{scope}
  \end{tikzpicture}
  \qquad\raisebox{3.5cm}{$=$}\qquad
  \begin{tikzpicture}[strings]
    \def\yy{0}
    \node[fiber] (m\yy) at (-.5,\yy,0) {$M$};
    \node[fiber] (q\yy) at (0,\yy,2) {$Q$};
    \node[pbeflat] (d5\yy) at (5,\yy,1) {};
    \draw (m\yy) -- (d5\yy) -- (6,\yy,1);
    \draw (q\yy) -- (d5\yy);

    \def\yy{1}
    \node[fiber] (m\yy) at (-.5,\yy,0) {$M$};
    \node[fiber] (q\yy) at (0,\yy,2) {$Q$};
    \node[pfeflat] (d1\yy) at (1,\yy,2) {};
    \node[pbeflat] (d2\yy) at (2,\yy,2) {};
    \node[pbeflat] (d5\yy) at (5,\yy,1) {};
    \draw (m\yy) -- (d5\yy) -- (6,\yy,1);
    \draw (d1\yy) to[bend right] (d2\yy);
    \draw (d1\yy) to[bend left] (d2\yy);
    \draw (q\yy) -- (d1\yy);
    \draw (d2\yy) -- (d5\yy);

    \def\yy{2}
    \node[fiber] (m\yy) at (-.5,\yy,0) {$M$};
    \node[fiber] (q\yy) at (0,\yy,2) {$Q$};
    \node[pfeflat] (d1\yy) at (1,\yy,2) {};
    \node[pbeflat] (d2\yy) at (2,\yy,0) {};
    \node[pbeflat] (d5\yy) at (5,\yy,1) {};
    \draw (m\yy) -- (d2\yy) -- (d5\yy) -- (6,\yy,1);
    \draw (d1\yy) -- (d2\yy);
    \draw (q\yy) -- (d1\yy);
    \draw (d1\yy) -- (d5\yy);

    \def\yy{3}
    \node[fiber] (m\yy) at (-.5,\yy,0) {$M$};
    \node[fiber] (q\yy) at (0,\yy,2) {$Q$};
    \node[pfeflat] (d1\yy) at (1,\yy,2) {};
    \node[pbeflat] (d2\yy) at (2,\yy,0) {};
    \node[pfeflat] (d3\yy) at (3,\yy,0) {};
    \node[pbeflat] (d4\yy) at (4,\yy,0) {};
    \node[pbeflat] (d5\yy) at (5,\yy,1) {};
    \draw (m\yy) -- (d2\yy) -- (d3\yy);
    \draw (d4\yy) -- (d5\yy);
    \draw (d1\yy) -- (d2\yy);
    \draw (q\yy) -- (d1\yy);
    \draw (d1\yy) -- (d5\yy) -- (6,\yy,1);

    \def\yy{4}
    \node[fiber] (m\yy) at (-.5,\yy,0) {$M$};
    \node[fiber] (q\yy) at (0,\yy,2) {$Q$};
    \node[pfeflat] (d1\yy) at (1,\yy,2) {};
    \node[pbeflat] (d2\yy) at (2,\yy,0) {};
    \node[pfeflat] (d3\yy) at (3,\yy,0) {};
    \draw (m\yy) -- (d2\yy) -- (d3\yy);
    \draw (d1\yy) -- (d2\yy);
    \draw (q\yy) -- (d1\yy);
    \draw (d1\yy) -- (6,\yy,1);

    \begin{scope}[inner]
      \draw (d50) -- (d51) -- (d22);
      \draw (d21) -- (d52) -- (d53);
      \node[inneriso] at (intersection of d22--d51 and d21--d52) {};
      \draw (d53) to[out=-90,in=-90,looseness=2] node[inneriso] {} (d43);
      \draw (d43) to[out=90,in=90,looseness=2] node[innerunit] {} (d33);
      \draw (d33) -- (d43);
      \draw (d21) to[out=90,in=90,looseness=3] node[innerunit] {} (d11);
      \draw (d11) -- (d12) -- (d13) -- (d14);
      \draw (d22) -- (d23) -- (d24);
      \draw (d33) -- (d34);
    \end{scope}
    \begin{scope}[outer]
      \draw (m0) -- (m1) -- (m2) -- (m3) -- (m4);
      \draw (q0) -- (q1) -- (q2) -- (q3) -- (q4);
    \end{scope}
  \end{tikzpicture}
  \caption{The sufficient condition for \autoref{thm:fibtrace}\ref{item:fibtrace2}}
  \label{fig:ftilde2}
\end{figure}

Now, from the conclusion of~\ref{item:fibtrace1} and the naturality of
symmetric monoidal traces, we can conclude that the following diagram
commutes.
  \[\xymatrix{ (\pi_A)_!Q \ar[r] \ar[d]_{(\pi_A)_! \tr(f)} &
    (\pi_A)_!(\Delta_A)^* (\Delta_A)_! Q \ar[d]_{(\pi_A)_!\tr(\overline{\widetilde{f}})}
    \ar[dr]^{\tr(\widetilde{f})}
    \\
    (\pi_A)_! P \ar[r] & (\pi_A)_! (\pi_A)^* (\pi_A)_! P \ar[r] & (\pi_A)_! P }
  \]
  But the composite along the bottom of this diagram is the identity, by
  a triangle law for the adjunction $(\pi_A)_! \dashv (\pi_A)^*$.  This
  yields the desired result.
\end{proof}

\begin{proof}[of {\normalfont\autoref{thm:fibtrace}\ref{item:fibtrace3}}]
  The bijection between the two types of morphism is immediate, since
  the respective domains and codomains are isomorphic.  For a morphism
  $f\colon Q\otimes M \to M\otimes (\pi_A)^*P$, the corresponding
  $\widehat{f}\colon (\Delta_A)_! Q \odot \widehat{M} \to \widehat{M}
  \odot P$ is shown on the left side of Figure~\ref{fig:fhat}.
  Moreover, given $f\colon Q\otimes M \to M\otimes (\pi_A)^*P$, we can
  also construct $\widetilde{f}$ as in part~\ref{item:fibtrace2}
  (replacing $P$ by $(\pi_A)^*P$), and we claim that $\widehat{f}$ is
  the composite
  \begin{equation}
    \xymatrix{(\Delta_A)_! Q \odot \widehat{M} \ar[r]^-{\widetilde{f}} &
      \widehat{M} \odot (\pi_A)_!(\pi_A)^* P \ar[r]^-{\id\odot \ep} &
      \widehat{M} \odot P.}\label{eq:fhat-ftilde}
  \end{equation}
  This composite is shown on the right side of Figure~\ref{fig:fhat}; a
  slide and a triangle identity suffice to prove the equality.  Now the
  desired result follows from the conclusion of
  part~\ref{item:fibtrace2} together with the naturality of
  bicategorical traces~\cite[Prop.~7.1]{PS2}.
\end{proof}
\begin{figure}
  \centering
  \begin{tikzpicture}[strings]
    \def\yy{0}
    \node[fiber] (m\yy) at (-1,\yy,0) {$M$};
    \node[fiber] (q\yy) at (0,\yy,2) {$Q$};
    \node[pfeflat] (d1\yy) at (1,\yy,2) {};
    \node[pbeflat] (d2\yy) at (2,\yy,1) {};
    \node[pfeflat] (d3\yy) at (3,\yy,1) {};
    \draw (m\yy) -- (d2\yy) -- (d3\yy);
    \draw (q\yy) -- (d1\yy) -- (d2\yy);
    \draw (d1\yy) -- (4,\yy,2);

    \def\yy{1}
    \node[fiber] (m\yy) at (-1,\yy,0) {$M$};
    \node[fiber] (q\yy) at (0,\yy,2) {$Q$};
    \node[pbeflat] (d1\yy) at (1,\yy,1) {};
    \node[pfeflat] (d2\yy) at (2,\yy,1) {};
    \node[pfeflat] (d3\yy) at (3,\yy,0) {};
    \draw (m\yy) -- (d1\yy) -- (d2\yy) -- (d3\yy);
    \draw (q\yy) -- (d1\yy);
    \draw (d2\yy) -- (4,\yy,1);

    \def\yy{2}
    \node[fiber] (m\yy) at (-1,\yy,0) {$M$};
    \node[fiber] (q\yy) at (0,\yy,2) {$Q$};
    \node[pbeflat] (d1\yy) at (1,\yy,1) {};
    \draw (m\yy) -- (d1\yy) -- (4,\yy,1);
    \draw (q\yy) -- (d1\yy);

    \def\yy{4}
    \node[fiber] (p\yy) at (-1,\yy,0) {$P$};
    \node[fiber] (m\yy) at (0,\yy,2) {$M$};
    \node[pbeflat] (d1\yy) at (1,\yy,0) {};
    \node[pbeflat] (d2\yy) at (2,\yy,1) {};
    \draw (d1\yy) -- (d2\yy) -- (4,\yy,1);
    \draw (m\yy) -- (d2\yy);

    \def\yy{5}
    \node[fiber] (p\yy) at (-1,\yy,0) {$P$};
    \node[fiber] (m\yy) at (0,\yy,2) {$M$};
    \draw (m\yy) -- (4,\yy,2);

    \begin{scope}[outer]
      \node[outervert] (f) at (-.3,3,1) {$f$};
      \draw (m0) -- (m1) -- (m2) -- (f) -- (m4) -- (m5);
      \draw (q0) -- (q1) -- (q2) -- (f) -- (p4) -- (p5);
    \end{scope}
    \begin{scope}[inner]
      \draw (d20) -- (d11) -- (d12) -- (f);
      \draw (d10) -- (d21);
      \node[innerbc] at (intersection of d10--d21 and d20--d11) {};
      \draw (d21) to[out=-90,in=-90,looseness=3] node[inneriso] {} (d31);
      \draw (d30) -- (d31);
      \draw (f) -- (d14);
      \draw (f) -- (d24);
      \draw (d14) to[out=-90,in=-90,looseness=3] node[inneriso] {} (d24);
    \end{scope}
    \node[white] at (0,7,0) {$M$};
  \end{tikzpicture}
  \qquad\raisebox{6cm}{$=$}\qquad
  \begin{tikzpicture}[strings]
    \def\yy{-1}
    \node[fiber] (m\yy) at (-.5,\yy,0) {$M$};
    \node[fiber] (q\yy) at (0,\yy,2) {$Q$};
    \node[pfeflat] (d3\yy) at (4,\yy,2) {};
    \node[pbeflat] (d4\yy) at (5,\yy,1) {};
    \node[pfeflat] (d5\yy) at (6,\yy, 1){};
    \draw (m\yy) -- (d4\yy) -- (d5\yy);
    \draw (q\yy)--(d3\yy) -- (d4\yy);
     \draw (d3\yy)--(7, \yy,3);

    \def\yy{0}
    \node[fiber] (m\yy) at (-.5,\yy,0) {$M$};
    \node[fiber] (q\yy) at (0,\yy,2) {$Q$};
    \node[pbeflat] (d3\yy) at (4,\yy,1.5) {};
    \node[pfeflat] (d4\yy) at (5,\yy,1.5) {};
     \node[pfeflat] (d5\yy) at (6,\yy, 1){};
     \draw (m\yy) -- (d3\yy);
     \draw (d4\yy)  -- (d5\yy);
    \draw (q\yy)--(d3\yy) -- (d4\yy) --(7, \yy,3);

    \def\yy{1}
    \node[fiber] (m\yy) at (-.5,\yy,0) {$M$};
    \node[fiber] (q\yy) at (0,\yy,2) {$Q$};
    \node[pbeflat] (d4\yy) at (4,\yy,1) {};
    \draw (m\yy) -- (d4\yy) --(7, \yy,1);
    \draw (q\yy) -- (d4\yy);

    \def\yy{3}
    \node[fiber] (p\yy) at (-.5,\yy,0) {$P$};
    \node[pbeflat] (d1\yy) at (1,\yy,0) {};
    \node[fiber] (m\yy) at (0,\yy,2) {$M$};
    \node[pbeflat] (d4\yy) at (4,\yy,1) {};
    \draw (m\yy) -- (d4\yy) --(7, \yy,1);
    \draw (d1\yy) -- (d4\yy);

    \node[outervert] (f) at (-.25,2.1,1) {$f$};
    \draw[outer] (m1) -- (f);
    \draw[outer] (q1) -- (f);
    \draw[inner] (f) -- (d13);
    \draw[outer] (f) -- (m3);

    \def\yy{4}
    \node[fiber] (p\yy) at (-.5,\yy,0) {$P$};
    \node[pbeflat] (d1\yy) at (1,\yy,0) {};
    \node[fiber] (m\yy) at (0,\yy,1.5) {$M$};
    \node[pfeflat] (d2\yy) at (2,\yy,0) {};
    \node[pbeflat] (d3\yy) at (3,\yy,0) {};
    \node[pbeflat] (d4\yy) at (4,\yy,1) {};
    \draw (m\yy) -- (d4\yy) --(7, \yy,1);
    \draw (d1\yy) -- (d2\yy);
     \draw(d3\yy)--(d4\yy);

    \draw[outer] (q-1) -- (q0) -- (q1);

    \begin{scope}[yshift=.5cm]

    \def\yy{5}
    \node[fiber] (p\yy) at (-.5,\yy,0.5) {$P$};
    \node[pbeflat] (d1\yy) at (1,\yy,0.5) {};
    \node[fiber] (m\yy) at (0,\yy,2) {$M$};
    \node[pfeflat] (d2\yy) at (2,\yy,0.5) {};

    \draw (m\yy) -- (7, \yy,2);
    \draw (d1\yy) -- (d2\yy);

    \def\yy{6}
    \node[fiber] (p\yy) at (-.5,\yy,0.5) {$P$};
    \node[fiber] (m\yy) at (0,\yy,2) {$M$};
    \draw (m\yy) -- (7, \yy,2);

    \draw[inner] (d30) -- (d41)-- (f) -- (d43)--(d44)
    to[out=-90,in=-90,looseness=2] node[inneriso] {} (d34);
    \draw[inner] (d5-1) -- (d50) to[out=-90,in=-90,looseness=2] node[inneriso] {} (d40);
    \draw (d40) -- (d3-1);
    \draw[inner] (d4-1) -- (d30);
    \node[innerbc] at (intersection of d3-1--d40 and d4-1--d30) {};
    \draw[inner](d24)--(d25);
    \draw[outer] (m-1) -- (m0) -- (m1);
    \draw[outer] (m3) -- (m4)--(m5) -- (m6);
    \draw[inner] (d13)--(d14)--(d15);
    \draw[outer](f) -- (p3)--(p4)--(p5) -- (p6);
    \draw[inner] (d15) to[out=-90,in=-90,looseness=4] node[innercounit] {} (d25);

    \draw[inner] (d34) to[out=90,in=90,looseness=1.5] node[innerunit] {} (d24);

    \end{scope}
  \end{tikzpicture}

  \caption{The map $\widehat{f}$ in \autoref{thm:fibtrace}\ref{item:fibtrace3}}
  \label{fig:fhat}
\end{figure}

\section{Proofs for total duality and trace}
\label{sec:tdproofs}

In contrast to the fiberwise case, our general comparison theorems for
total duality traces were completely formal, not requiring any string
diagram calculations.  However, the two most interesting applications,
namely Corollaries \ref{cor:totaltr3} and \ref{thm:totaltr-transfer},
required identifying the composites $\xi\circ f$ and $\zeta\circ f$ as
$\Sigma(f)$ and $\Sigma(\Delta_A\circ f)$, respectively, and for this we
do need some calculation.

Recall that we have an endomorphism $\phi\colon A\to A$ of $A\in \bS$
and that we defined $f$ to be the isomorphism
\begin{equation}\label{eq:f-over-f1}
  I_A \cong \phi^* I_A
\end{equation}
regarded as a morphism $\widecheck{I_A} \to \widecheck{I_A}\odot A_\phi$.
We defined $\xi$ be the morphism
\[ A_\phi \odot \widehat{I_A} \xto{\cong} \widehat{\phi_! I_A}
\xto{\cong} \widehat{\phi_! \phi^* I_A} \too \widehat{I_A}.
\]
By definition, then, the composite $\xi\circ f$ becomes
\begin{multline}
(\pi_A)_! I_A \xto{\cong} \widecheck{I_A} \odot \widehat{I_A}
\xto{\cong}
\widecheck{\phi^* I_A} \odot \widehat{I_A}
\xto{\cong} \widecheck{I_A} \odot A_\phi \odot \widehat{I_A}
\xto{\cong} \widecheck{I_A} \odot \widehat{\phi_! I_A}\\
\xto{\cong} \widecheck{I_A} \odot \widehat{\phi_! \phi^* I_A}
\too \widecheck{I_A} \odot \widehat{I_A}
\xto{\cong} (\pi_A)_! I_A.\label{eq:bco-sliding-0}
\end{multline}

To give a simpler description of this composite we start by considering
the composite of only the third and fourth morphisms (the two on the
second line of~\eqref{eq:bco-sliding-0}).  These are instances of
\autoref{thm:bco-action}.  Tracing through their definitions yields a
fairly long composite of isomorphisms, but fortunately it is equal to
something simpler, as an instance of the following lemma.

\begin{lem}\label{thm:bco-sliding}
  Let $M\in\sC^B$, $N\in\sC^A$, and let $\phi\colon A\to B$ be a morphism in \bS.
  Then the composite
  \begin{equation}
    \widecheck{\phi^* M} \odot \widehat{N} \xto{\cong}
    \widecheck{M} \odot B_\phi \odot \widehat{N} \xto{\cong}
    \widecheck{M} \odot \widehat{\phi_! N}.\label{eq:bco-sliding}
  \end{equation}
  is equal to
   \begin{align*}
     \widecheck{\phi^* M} \odot \widehat{N}
     &= (\pi_A)_! (\Delta_A)^* (\phi^* M \boxtimes N)\\
     & \cong (\pi_B)_! \phi_! (\Delta_A)^* (\phi^* M \boxtimes N)\\
     & \cong (\pi_B)_! (\Delta_B)^* (M \boxtimes \phi_! N)\\
     & = \widecheck{M} \odot \widehat{\phi_! N}.
   \end{align*}
\end{lem}
\begin{proof}
  A straightforward, though tedious, verification, using the coherence
  of Beck-Chevalley morphisms (the techniques for computation with mates
  described in~\cite{ks:r2cats} are useful).
\end{proof}

The composite displayed in \autoref{thm:bco-sliding} is shown
graphically in Figure~\ref{fig:bco-sliding}.  Note the occurrence of the
``sliding and splitting'' Beck-Chevalley isomorphism (recall
Figure~\ref{fig:slidpb}).
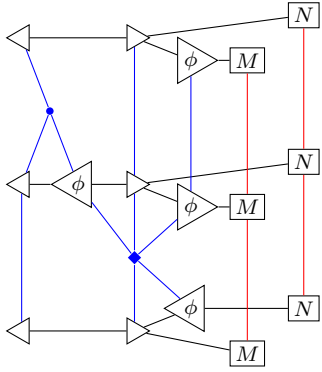
\begin{figure}
  \centering
  \begin{tikzpicture}[strings,y={(0,1.3)}]
    \def\yy{0}
    \node[fiber] (m\yy) at (0,\yy,1) {$M$};
    \node[fiber] (n\yy) at (-1,\yy,-1) {$N$};
    \node[pbsmflat] (phi1\yy) at (1,\yy,1) {$\phi$};
    \node[pbeflat] (d2\yy) at (2,\yy,0) {};
    \node[pfeflat] (d4\yy) at (4,\yy,0) {};
    \draw (m\yy) -- (phi1\yy) -- (d2\yy) -- (d4\yy);
    \draw (n\yy) -- (d2\yy);
    
    \def\yy{1}
    \node[fiber] (m\yy) at (0,\yy,1) {$M$};
    \node[fiber] (n\yy) at (-1,\yy,-1) {$N$};
    \node[pbsmflat] (phi1\yy) at (1,\yy,1) {$\phi$};
    \node[pbeflat] (d2\yy) at (2,\yy,0) {};
    \node[pfsmflat] (phi3\yy) at (3,\yy,0) {$\phi$};
    \node[pfeflat] (d4\yy) at (4,\yy,0) {};
    \draw (m\yy) -- (phi1\yy) -- (d2\yy) -- (phi3\yy) -- (d4\yy);
    \draw (n\yy) -- (d2\yy);

    \def\yy{2}
    \node[fiber] (m\yy) at (0,\yy,1) {$M$};
    \node[fiber] (n\yy) at (-1,\yy,-1) {$N$};
    \node[pfsmflat] (phi1\yy) at (1,\yy,-1) {$\phi$};
    \node[pbeflat] (d2\yy) at (2,\yy,0) {};
    \node[pfeflat] (d4\yy) at (4,\yy,0) {};
    \draw (m\yy) -- (d2\yy) -- (d4\yy);
    \draw (n\yy) -- (phi1\yy) -- (d2\yy);

    \node[inneriso] (ps) at (3.5,.5,0) {};
    \node[innerbc] (bc) at (2,1.5,0) {};
    \begin{scope}[inner]
      \draw (d40) -- (ps) -- (d41) -- (d42);
      \draw (ps) -- (phi31) -- (bc);
      \draw (d20) -- (d21) -- (bc) -- (d22);
      \draw (phi10) -- (phi11) -- (bc) -- (phi12);
    \end{scope}
    \begin{scope}[outer]
      \draw (m0) -- (m1) -- (m2);
      \draw (n0) -- (n1) -- (n2);
    \end{scope}
  \end{tikzpicture}
  \caption{The simplified version of~\eqref{eq:bco-sliding}}
  \label{fig:bco-sliding}
\end{figure}

\begin{prop}\label{prop:xi_circ_f}
The composite $\xi\circ f$ is $\Sigma(\phi)$.
\end{prop} 
\begin{proof}
  Using \autoref{thm:bco-sliding}, and unfolding the definition of the
  ``sliding and splitting'' Beck-Chevalley morphism,
  \eqref{eq:bco-sliding-0} becomes the morphism shown in
  Figure~\ref{fig:xif1}.  In Figure~\ref{fig:xif2} we simply rearrange
  the order of all the morphisms, sliding the counit $\phi_! \phi^* \to
  \Id$ on the far left down to the bottom, the isomorphism $\phi^* \pi^*
  \cong \pi^*$ on the far right up to the top, and rearranging things in
  the middle so that the other unit and counit meet.

  In Figure~\ref{fig:xif3} we cancel the internal unit and counit using
  a triangle identity.  Now we have the isomorphisms $\phi^* \pi^* \cong
  \pi^*$ and $\pi_! \phi_! \cong \pi_!$ at the top, a counit $\phi_!
  \phi^* \to \Id$ at the bottom, and in the middle a composite of
  pseudofunctoriality isomorphisms for reindexing functors, which starts
  and ends at $\pi_! \phi_! \phi^* \pi^*$.  This middle composite is
  extracted in Figure~\ref{fig:xifm}, after peeling off the $\pi_!
  \phi_!$ and $\pi^*$ (which are unchanged throughout).  By coherence
  for pseudofunctors, this composite is equal to the identity, so that
  Figure~\ref{fig:xif3} is equal to Figure~\ref{fig:xif4}.  But this is
  exactly $\Sigma(\phi)$ (see~\eqref{eq:sigma-phi}).
\end{proof}
\afterpage{\begin{landscape}
\begin{figure}
  \centering
  \subfigure[Step 1]{\label{fig:xif1}
    \begin{tikzpicture}[strings,y={(0,1.1)}, yscale =.8]
      \def\yy{0}
      \node[pbeflat] (d-2\yy) at (-2,\yy,0) {};
      \node[pfeflat] (d7\yy) at (7,\yy,0) {};
      \draw (d-2\yy) -- (d7\yy);

      \def\yy{1}
      \node[pbeflat] (d-2\yy) at (-2,\yy,-1) {};
      \node[pbeflat] (d-1\yy) at (-1,\yy,1) {};
      \node[pbeflat] (d4\yy) at (4,\yy,0) {};
      \node[pfeflat] (d7\yy) at (7,\yy,0) {};
      \draw (d-2\yy) -- (d4\yy) -- (d7\yy);
      \draw (d-1\yy) -- (d4\yy);
      
      \def\yy{2}
      \node[pbeflat] (d-2\yy) at (-2,\yy,-1) {};
      \node[pbeflat] (d-1\yy) at (-1,\yy,1) {};
      \node[pbsmflat] (phi3\yy) at (3,\yy,1) {$\phi$};
      \node[pbeflat] (d4\yy) at (4,\yy,0) {};
      \node[pfeflat] (d7\yy) at (7,\yy,0) {};
      \draw (d-2\yy) -- (d4\yy) -- (d7\yy);
      \draw (d-1\yy) -- (phi3\yy) -- (d4\yy);

      \def\yy{3}
      \node[pbeflat] (d-2\yy) at (-2,\yy,-1) {};
      \node[pbeflat] (d-1\yy) at (-1,\yy,1) {};
      \node[pbsmflat] (phi3\yy) at (3,\yy,1) {$\phi$};
      \node[pbeflat] (d4\yy) at (4,\yy,0) {};
      \node[pfsmflat] (phi6\yy) at (6,\yy,0) {$\phi$};
      \node[pfeflat] (d7\yy) at (7,\yy,0) {};
      \draw (d-2\yy) -- (d4\yy) -- (phi6\yy) -- (d7\yy);
      \draw (d-1\yy) -- (phi3\yy) -- (d4\yy);

      \def\yy{4}
      \node[pbeflat] (d-2\yy) at (-2,\yy,-1) {};
      \node[pbeflat] (d-1\yy) at (-1,\yy,1) {};
      \node[pfsmflat] (phi1\yy) at (1,\yy,-1) {$\phi$};
      \node[pbsmflat] (phi2\yy) at (2,\yy,-1) {$\phi$};
      \node[pbsmflat] (phi3\yy) at (3,\yy,1) {$\phi$};
      \node[pbeflat] (d4\yy) at (4,\yy,0) {};
      \node[pfsmflat] (phi6\yy) at (6,\yy,0) {$\phi$};
      \node[pfeflat] (d7\yy) at (7,\yy,0) {};
      \draw (d-2\yy) -- (phi1\yy) -- (phi2\yy) -- (d4\yy) -- (phi6\yy) -- (d7\yy);
      \draw (d-1\yy) -- (phi3\yy) -- (d4\yy);

      \def\yy{5}
      \node[pbeflat] (d-2\yy) at (-2,\yy,-1) {};
      \node[pbeflat] (d-1\yy) at (-1,\yy,1) {};
      \node[pfsmflat] (phi1\yy) at (1,\yy,-1) {$\phi$};
      \node[pbeflat] (d4\yy) at (4,\yy,0) {};
      \node[pbsmflat] (phi5\yy) at (5,\yy,0) {$\phi$};
      \node[pfsmflat] (phi6\yy) at (6,\yy,0) {$\phi$};
      \node[pfeflat] (d7\yy) at (7,\yy,0) {};
      \draw (d-2\yy) -- (phi1\yy) -- (d4\yy) -- (phi5\yy) -- (phi6\yy) -- (d7\yy);
      \draw (d-1\yy) -- (d4\yy);

      \def\yy{6}
      \node[pbeflat] (d-2\yy) at (-2,\yy,-1) {};
      \node[pbeflat] (d-1\yy) at (-1,\yy,1) {};
      \node[pfsmflat] (phi1\yy) at (1,\yy,-1) {$\phi$};
      \node[pbeflat] (d4\yy) at (4,\yy,0) {};
      \node[pfeflat] (d7\yy) at (7,\yy,0) {};
      \draw (d-2\yy) -- (phi1\yy) -- (d4\yy) -- (d7\yy);
      \draw (d-1\yy) -- (d4\yy);

      \def\yy{7}
      \node[pbeflat] (d-2\yy) at (-2,\yy,-1) {};
      \node[pbeflat] (d-1\yy) at (-1,\yy,1) {};
      \node[pbsmflat] (phi0\yy) at (0,\yy,-1) {$\phi$};
      \node[pfsmflat] (phi1\yy) at (1,\yy,-1) {$\phi$};
      \node[pbeflat] (d4\yy) at (4,\yy,0) {};
      \node[pfeflat] (d7\yy) at (7,\yy,0) {};
      \draw (d-2\yy) -- (phi0\yy) -- (phi1\yy) -- (d4\yy) -- (d7\yy);
      \draw (d-1\yy) -- (d4\yy);

      \def\yy{8}
      \node[pbeflat] (d-2\yy) at (-2,\yy,-1) {};
      \node[pbeflat] (d-1\yy) at (-1,\yy,1) {};
      \node[pbeflat] (d4\yy) at (4,\yy,0) {};
      \node[pfeflat] (d7\yy) at (7,\yy,0) {};
      \draw (d-2\yy) -- (d4\yy) -- (d7\yy);
      \draw (d-1\yy) -- (d4\yy);
      
      \def\yy{9}
      \node[pbeflat] (d-2\yy) at (-2,\yy,0) {};
      \node[pfeflat] (d7\yy) at (7,\yy,0) {};
      \draw (d-2\yy) -- (d7\yy);

      \node[inneriso] (iso2) at (1,1.5,0) {};
      \node[inneriso] (iso3) at (6.5,2.5,0) {};
      \node[inneriso] (iso5) at (4,4.5,0) {};
      \node[inneriso] (iso7) at (-1.5,6.5,0) {};
      \begin{scope}[inner]
        \draw (d-20) -- (d-21) -- (d-22) -- (d-23) -- (d-24) -- (d-25) --
        (d-26) -- (iso7) -- (d-27) -- (d-28) -- (d-29);
        \draw (d70) -- (d71) -- (d72) -- (iso3) -- (d73) -- (d74) --
        (d75) -- (d76) -- (d77) -- (d78) -- (d79);
        \draw (d-11) to[out=90,in=90,looseness=1] node[inneriso] {} (d41);
        \draw (d41) -- (d42) -- (d43) -- (d44) -- (iso5) -- (d45) --
        (d46) -- (d47) -- (d48);
        \draw (d-11) -- (iso2) -- (d-12) -- (d-13) -- (d-14) --
        (d-15) -- (d-16) -- (d-17) -- (d-18);
        \draw (d-18) to[out=-90,in=-90,looseness=1] node[inneriso] {} (d48);
        \draw (iso2) -- (phi32) -- (phi33) -- (phi34) -- (iso5) -- (phi55);
        \draw (phi55) to[out=-90,in=-90,looseness=3] node[innercounit] {} (phi65);
        \draw (iso3) -- (phi63) -- (phi64) -- (phi65);
        \draw (iso5) to[out=0,in=-90] (phi24);
        \draw (phi24) to[out=90,in=90,looseness=3] node[innerunit] {} (phi14);
        \draw (phi14) -- (phi15) -- (phi16) -- (phi17);
        \draw (phi17) to[out=-90,in=-90,looseness=4] node[innercounit] {} (phi07);
        \draw (phi07) -- (iso7);
      \end{scope}
    \end{tikzpicture}}
  \hspace{1cm}
  \subfigure[Step 2]{\label{fig:xif2}
    \begin{tikzpicture}[strings,y={(0,1.1)}, yscale=.8]
      \def\yy{0}
      \node[pbeflat] (d-2\yy) at (-2,\yy,0) {};
      \node[pfeflat] (d7\yy) at (7,\yy,0) {};
      \draw (d-2\yy) -- (d7\yy);

      \def\yy{1}
      \node[pbeflat] (d-2\yy) at (-2,\yy,0) {};
      \node[pbsmflat] (phi0\yy) at (0,\yy,0) {$\phi$};
      \node[pfeflat] (d7\yy) at (7,\yy,0) {};
      \draw (d-2\yy) -- (phi0\yy)-- (d7\yy);
      
      \def\yy{2}
      \node[pbeflat] (d-2\yy) at (-2,\yy,0) {};
      \node[pbsmflat] (phi0\yy) at (0,\yy,0) {$\phi$};
      \node[pfsmflat] (phi6\yy) at (6,\yy,0) {$\phi$};
      \node[pfeflat] (d7\yy) at (7,\yy,0) {};
      \draw (d-2\yy)  -- (phi0\yy) -- (phi6\yy) -- (d7\yy);

      \def\yy{3}
      \node[pbeflat] (d-2\yy) at (-2,\yy,-1) {};
      \node[pbeflat] (d-1\yy) at (-1,\yy,1) {};
      \node[pbsmflat] (phi0\yy) at (0,\yy,-1) {$\phi$};
      \node[pbeflat] (d4\yy) at (4,\yy,0) {};
      \node[pfsmflat] (phi6\yy) at (6,\yy,0) {$\phi$};
      \node[pfeflat] (d7\yy) at (7,\yy,0) {};
      \draw (d-2\yy) -- (phi0\yy) -- (d4\yy) -- (phi6\yy) -- (d7\yy);
      \draw (d-1\yy)  -- (d4\yy);

      \def\yy{4}
      \node[pbeflat] (d-2\yy) at (-2,\yy,-1) {};
      \node[pbeflat] (d-1\yy) at (-1,\yy,1) {};
      \node[pbsmflat] (phi0\yy) at (0,\yy,-1) {$\phi$};
      \node[pbsmflat] (phi3\yy) at (3,\yy,1) {$\phi$};
      \node[pbeflat] (d4\yy) at (4,\yy,0) {};
      \node[pfsmflat] (phi6\yy) at (6,\yy,0) {$\phi$};
      \node[pfeflat] (d7\yy) at (7,\yy,0) {};
      \draw (d-2\yy) -- (phi0\yy) -- (d4\yy) -- (phi6\yy) -- (d7\yy);
      \draw (d-1\yy) -- (phi3\yy) -- (d4\yy);

      \def\yy{5}
      \node[pbeflat] (d-2\yy) at (-2,\yy,-1) {};
      \node[pbeflat] (d-1\yy) at (-1,\yy,1) {};
      \node[pbsmflat] (phi0\yy) at (0,\yy,-1) {$\phi$};
      \node[pfsmflat] (phi1\yy) at (1,\yy,-1) {$\phi$};
      \node[pbsmflat] (phi2\yy) at (2,\yy,-1) {$\phi$};
      \node[pbsmflat] (phi3\yy) at (3,\yy,1) {$\phi$};
      \node[pbeflat] (d4\yy) at (4,\yy,0) {};
      \node[pfsmflat] (phi6\yy) at (6,\yy,0) {$\phi$};
      \node[pfeflat] (d7\yy) at (7,\yy,0) {};
      \draw (d-2\yy) -- (phi0\yy) -- (phi1\yy) -- (phi2\yy) --
      (d4\yy) -- (phi6\yy) -- (d7\yy);
      \draw (d-1\yy) -- (phi3\yy) -- (d4\yy);

      \def\yy{6}
      \node[pbeflat] (d-2\yy) at (-2,\yy,-1) {};
      \node[pbeflat] (d-1\yy) at (-1,\yy,1) {};
      \node[pbsmflat] (phi2\yy) at (2,\yy,-1) {$\phi$};
      \node[pbsmflat] (phi3\yy) at (3,\yy,1) {$\phi$};
      \node[pbeflat] (d4\yy) at (4,\yy,0) {};
      \node[pfsmflat] (phi6\yy) at (6,\yy,0) {$\phi$};
      \node[pfeflat] (d7\yy) at (7,\yy,0) {};
      \draw (d-2\yy) -- (phi2\yy)-- (d4\yy) --  (phi6\yy) -- (d7\yy);
      \draw (d-1\yy) -- (phi3\yy) -- (d4\yy);

      \def\yy{7}
      \node[pbeflat] (d-2\yy) at (-2,\yy,-1) {};
      \node[pbeflat] (d-1\yy) at (-1,\yy,1) {};
      \node[pbeflat] (d4\yy) at (4,\yy,0) {};
      \node[pbsmflat] (phi5\yy) at (5,\yy,0) {$\phi$};
      \node[pfsmflat] (phi6\yy) at (6,\yy,0) {$\phi$};
      \node[pfeflat] (d7\yy) at (7,\yy,0) {};
      \draw (d-2\yy) -- (d4\yy) -- (phi5\yy) -- (phi6\yy) -- (d7\yy);
      \draw (d-1\yy) -- (d4\yy);

      \def\yy{8}
      \node[pbeflat] (d-2\yy) at (-2,\yy,0) {};
      \node[pbsmflat] (phi5\yy) at (5,\yy,0) {$\phi$};
      \node[pfsmflat] (phi6\yy) at (6,\yy,0) {$\phi$};
      \node[pfeflat] (d7\yy) at (7,\yy,0) {};
      \draw (d-2\yy) -- (phi5\yy) -- (phi6\yy) -- (d7\yy);
      
      \def\yy{9}
      \node[pbeflat] (d-2\yy) at (-2,\yy,0) {};
      \node[pfeflat] (d7\yy) at (7,\yy,0) {};
      \draw (d-2\yy) -- (d7\yy);

      \node[inneriso] (iso2) at (-.5,3.5,0) {};
      \node[inneriso] (iso3) at (6.5,1.5,0) {};
      \node[inneriso] (iso5) at (4,6.5,0) {};
      \node[inneriso] (iso7) at (-1.5,0.5,0) {};
      \begin{scope}[inner]
        \draw (d-20)-- (iso7)  -- (d-21) -- (d-22) -- (d-23) --
        (d-24) -- (d-25) -- (d-26) --(d-27) -- (d-28) -- (d-29);
        \draw (d70) -- (d71)  -- (iso3) -- (d72)-- (d73) -- (d74) --
        (d75) -- (d76) -- (d77) -- (d78) -- (d79);
        \draw (d-13) to[out=90,in=90,looseness=1] node[inneriso] {} (d43);
        \draw (d43) -- (d44) -- (d45) -- (d46) -- (iso5) -- (d47);
        \draw  (d-13)  -- (iso2)  -- (d-14) -- (d-15) -- (d-16) -- (d-17);
        \draw (d-17) to[out=-90,in=-90,looseness=1] node[inneriso] {} (d47);
        \draw (iso2) -- (phi34) -- (phi35)-- (phi36) --(iso5) -- (phi57) -- (phi58);
        \draw (phi58) to[out=-90,in=-90,looseness=3] node[innercounit] {} (phi68);
        \draw (iso3) -- (phi62) -- (phi63) -- (phi64) -- (phi65) --
        (phi66) -- (phi67) -- (phi68);
        \draw (iso5) to[out=0,in=-90] (phi26);
        \draw (phi25)--(phi26);
        \draw (phi25) to[out=90,in=90,looseness=3] node[innerunit] {} (phi15);
        \draw (phi15) to[out=-90,in=-90,looseness=4] node[innercounit] {} (phi05);
        \draw (phi05) -- (phi04) -- (phi03) -- (phi02)  -- (phi01)-- (iso7);
      \end{scope}
    \end{tikzpicture}}
  \caption{The composite $\xi\circ f$, steps 1--2}
  \label{fig:xif_1}
\end{figure}

\begin{figure}
  \centering
  \subfigure[Step 3]{\label{fig:xif3}
    \begin{tikzpicture}[strings,y={(0,1.1)}, yscale=.8]
      \def\yy{0}
      \node[pbeflat] (d-2\yy) at (-2,\yy,0) {};
      \node[pfeflat] (d7\yy) at (7,\yy,0) {};
      \draw (d-2\yy) -- (d7\yy);

      \def\yy{1}
      \node[pbeflat] (d-2\yy) at (-2,\yy,0) {};
      \node[pbsmflat] (phi0\yy) at (0,\yy,0) {$\phi$};
      \node[pfeflat] (d7\yy) at (7,\yy,0) {};
      \draw (d-2\yy) -- (phi0\yy)-- (d7\yy);
      
      \def\yy{2}
      \node[pbeflat] (d-2\yy) at (-2,\yy,0) {};
      \node[pbsmflat] (phi0\yy) at (0,\yy,0) {$\phi$};
      \node[pfsmflat] (phi6\yy) at (6,\yy,0) {$\phi$};
      \node[pfeflat] (d7\yy) at (7,\yy,0) {};
      \draw (d-2\yy)  -- (phi0\yy) -- (phi6\yy) -- (d7\yy);

      \def\yy{3}
      \node[pbeflat] (d-2\yy) at (-2,\yy,-1) {};
      \node[pbeflat] (d-1\yy) at (-1,\yy,1) {};
      \node[pbsmflat] (phi0\yy) at (0,\yy,-1) {$\phi$};
      \node[pbeflat] (d4\yy) at (4,\yy,0) {};
      \node[pfsmflat] (phi6\yy) at (6,\yy,0) {$\phi$};
      \node[pfeflat] (d7\yy) at (7,\yy,0) {};
      \draw (d-2\yy) -- (phi0\yy) -- (d4\yy) -- (phi6\yy) -- (d7\yy);
      \draw (d-1\yy)  -- (d4\yy);

      \def\yy{4}
      \node[pbeflat] (d-2\yy) at (-2,\yy,-1) {};
      \node[pbeflat] (d-1\yy) at (-1,\yy,1) {};
      \node[pbsmflat] (phi0\yy) at (0,\yy,-1) {$\phi$};
      \node[pbsmflat] (phi3\yy) at (3,\yy,1) {$\phi$};
      \node[pbeflat] (d4\yy) at (4,\yy,0) {};
      \node[pfsmflat] (phi6\yy) at (6,\yy,0) {$\phi$};
      \node[pfeflat] (d7\yy) at (7,\yy,0) {};
      \draw (d-2\yy) -- (phi0\yy) -- (d4\yy) -- (phi6\yy) -- (d7\yy);
      \draw (d-1\yy) -- (phi3\yy) -- (d4\yy);


      \def\yy{6}
      \node[pbeflat] (d-2\yy) at (-2,\yy,-1) {};
      \node[pbeflat] (d-1\yy) at (-1,\yy,1) {};
      \node[pbsmflat] (phi0\yy) at (2,\yy,-1) {$\phi$};
      \node[pbsmflat] (phi3\yy) at (3,\yy,1) {$\phi$};
      \node[pbeflat] (d4\yy) at (4,\yy,0) {};
      \node[pfsmflat] (phi6\yy) at (6,\yy,0) {$\phi$};
      \node[pfeflat] (d7\yy) at (7,\yy,0) {};
      \draw (d-2\yy) -- (phi0\yy)-- (d4\yy) --  (phi6\yy) -- (d7\yy);
      \draw (d-1\yy) -- (phi3\yy) -- (d4\yy);

      \def\yy{7}
      \node[pbeflat] (d-2\yy) at (-2,\yy,-1) {};
      \node[pbeflat] (d-1\yy) at (-1,\yy,1) {};
      \node[pbeflat] (d4\yy) at (4,\yy,0) {};
      \node[pbsmflat] (phi5\yy) at (5,\yy,0) {$\phi$};
      \node[pfsmflat] (phi6\yy) at (6,\yy,0) {$\phi$};
      \node[pfeflat] (d7\yy) at (7,\yy,0) {};
      \draw (d-2\yy) -- (d4\yy) -- (phi5\yy) -- (phi6\yy) -- (d7\yy);
      \draw (d-1\yy) -- (d4\yy);

      \def\yy{8}
      \node[pbeflat] (d-2\yy) at (-2,\yy,0) {};
      \node[pbsmflat] (phi5\yy) at (5,\yy,0) {$\phi$};
      \node[pfsmflat] (phi6\yy) at (6,\yy,0) {$\phi$};
      \node[pfeflat] (d7\yy) at (7,\yy,0) {};
      \draw (d-2\yy) -- (phi5\yy) -- (phi6\yy) -- (d7\yy);
      
      \def\yy{9}
      \node[pbeflat] (d-2\yy) at (-2,\yy,0) {};
      \node[pfeflat] (d7\yy) at (7,\yy,0) {};
      \draw (d-2\yy) -- (d7\yy);

      \node[inneriso] (iso2) at (-.5,3.5,0) {};
      \node[inneriso] (iso3) at (6.5,1.5,0) {};
      \node[inneriso] (iso5) at (4,6.5,0) {};
      \node[inneriso] (iso7) at (-1.5,0.5,0) {};
      \begin{scope}[inner]
        \draw (d-20)-- (iso7)  -- (d-21) -- (d-22) -- (d-23) --  (d-24) 
        -- (d-26) --(d-27) -- (d-28) -- (d-29);
        \draw (d70) -- (d71)  -- (iso3) -- (d72)-- (d73) -- (d74) 
        -- (d76) -- (d77) -- (d78) -- (d79);
        \draw (d-13) to[out=90,in=90,looseness=1] node[inneriso] {} (d43);
        \draw (d43) -- (d44) 
        -- (d46) -- (iso5) -- (d47);
        \draw  (d-13)  -- (iso2)  -- (d-14) 
        -- (d-16) -- (d-17);
        \draw (d-17) to[out=-90,in=-90,looseness=1] node[inneriso] {} (d47);
        \draw (iso2) -- (phi34) 
        -- (phi36) --(iso5) -- (phi57) -- (phi58);
        \draw (phi58) to[out=-90,in=-90,looseness=3] node[innercounit] {} (phi68);
        \draw (iso3) -- (phi62) -- (phi63) -- (phi64) 
        -- (phi66) -- (phi67) -- (phi68);
        \draw (iso5) to[out=0,in=-90] (phi06);
        \draw (phi06) -- (phi04) -- (phi03) -- (phi02)  -- (phi01)-- (iso7);
      \end{scope}
    \end{tikzpicture}}
  \hspace{1cm}
%
%
%
  \subfigure[The pseudofunctoriality part]{\label{fig:xifm}
    \begin{tikzpicture}[strings,y={(0,1.1)}, scale=.9]
      
      \def\yy{2}
      \node (d-2\yy) at (-2,\yy,-1) {};
      \node[pbsmflat] (phi0\yy) at (0,\yy,-1) {$\phi$};
      \node (phi6\yy) at (6,\yy,0){};
      \draw (d-2\yy)  -- (phi0\yy) -- (phi6\yy);

      \def\yy{3}
      \node(d-2\yy) at (-2,\yy,-1) {};
      \node[pbeflat] (d-1\yy) at (-1,\yy,1) {};
      \node[pbsmflat] (phi0\yy) at (0,\yy,-1) {$\phi$};
      \node[pbeflat] (d4\yy) at (4,\yy,0) {};
      \node (phi6\yy) at (6,\yy,0) {};
      \draw (d-2\yy) -- (phi0\yy) -- (d4\yy) -- (phi6\yy);
      \draw (d-1\yy)  -- (d4\yy);

      \def\yy{4}
      \node(d-2\yy) at (-2,\yy,-1) {};
      \node[pbeflat] (d-1\yy) at (-1,\yy,1) {};
      \node[pbsmflat] (phi0\yy) at (0,\yy,-1) {$\phi$};
      \node[pbsmflat] (phi3\yy) at (3,\yy,1) {$\phi$};
      \node[pbeflat] (d4\yy) at (4,\yy,0) {};
      \node (phi6\yy) at (6,\yy,0) {};
      \draw (d-2\yy) -- (phi0\yy) -- (d4\yy) -- (phi6\yy);
      \draw (d-1\yy) -- (phi3\yy) -- (d4\yy);


      \def\yy{6}
      \node (d-2\yy) at (-2,\yy,-1) {};
      \node[pbeflat] (d-1\yy) at (-1,\yy,1) {};
      \node[pbsmflat] (phi0\yy) at (2,\yy,-1) {$\phi$};
      \node[pbsmflat] (phi3\yy) at (3,\yy,1) {$\phi$};
      \node[pbeflat] (d4\yy) at (4,\yy,0) {};
      \node (phi6\yy) at (6,\yy,0) {};
      \draw (d-2\yy) -- (phi0\yy)-- (d4\yy) --  (phi6\yy);
      \draw (d-1\yy) -- (phi3\yy) -- (d4\yy);

      \def\yy{7}
      \node (d-2\yy) at (-2,\yy,-1) {};
      \node[pbeflat] (d-1\yy) at (-1,\yy,1) {};
      \node[pbeflat] (d4\yy) at (4,\yy,0) {};
      \node[pbsmflat] (phi5\yy) at (5,\yy,0) {$\phi$};
      \node (phi6\yy) at (6,\yy,0) {};
      \draw (d-2\yy) -- (d4\yy) -- (phi5\yy) -- (phi6\yy);
      \draw (d-1\yy) -- (d4\yy);

      \def\yy{8}
      \node (d-2\yy) at (-2,\yy,-1) {};
      \node[pbsmflat] (phi5\yy) at (5,\yy,0) {$\phi$};
      \node (phi6\yy) at (6,\yy,0) {};
      \draw (d-2\yy) -- (phi5\yy) -- (phi6\yy);

      \node[inneriso] (iso2) at (-.5,3.5,0) {};
      \node[inneriso] (iso5) at (4,6.5,0) {};
      \begin{scope}[inner]
        \draw (d-13) to[out=90,in=90,looseness=1] node[inneriso] {} (d43);
        \draw (d43) -- (d44) 
        -- (d46) -- (iso5) -- (d47);
        \draw  (d-13)  -- (iso2)  -- (d-14) 
        -- (d-16) -- (d-17);
        \draw (d-17) to[out=-90,in=-90,looseness=.8] node[inneriso] {} (d47);
        \draw (iso2) -- (phi34) 
        -- (phi36) --(iso5) -- (phi57) -- (phi58);
        \draw (iso5) to[out=0,in=-90] (phi06);
        \draw (phi06) -- (phi04) -- (phi03) -- (phi02) ;
      \end{scope}
    \end{tikzpicture}}
  \hspace{1cm}
  \subfigure[Step 4]{\label{fig:xif4}
    \begin{tikzpicture}[strings]
      \def\yy{0}
      \node[pbeflat] (d0\yy) at (0,\yy) {};
      \node[pfeflat] (d3\yy) at (3,\yy) {};
      \draw (d0\yy) -- (d3\yy);

      \def\yy{1}
      \node[pbeflat] (d0\yy) at (0,\yy) {};
      \node[pbsmflat] (phi1\yy) at (1,\yy) {$\phi$};
      \node[pfsmflat] (phi2\yy) at (2,\yy) {$\phi$};
      \node[pfeflat] (d3\yy) at (3,\yy) {};
      \draw (d0\yy) -- (phi1\yy) -- (phi2\yy) -- (d3\yy);

      \def\yy{2}
      \node[pbeflat] (d0\yy) at (0,\yy) {};
      \node[pfeflat] (d3\yy) at (3,\yy) {};
      \draw (d0\yy) -- (d3\yy);

      \node[inneriso] (iso11) at (.5,.5,0) {};
      \node[inneriso] (iso31) at (2.5,.5,0) {};
      \begin{scope}[inner]
        \draw (d00) -- (iso11) -- (d01) -- (d02);
        \draw (d30) -- (iso31) -- (d31) -- (d32);
        \draw (iso11) -- (phi11);
        \draw (iso31) -- (phi21);
        \draw (phi11) to[out=-90,in=-90,looseness=3] node[innercounit] {} (phi21);
      \end{scope}
    \end{tikzpicture}}
  \caption{The composite $\xi\circ f$, steps 3--4}
  \label{fig:xif_2}
\end{figure}
\end{landscape}
}

This completes the proof of \autoref{cor:totaltr3}; we now move on to
\autoref{thm:totaltr-transfer}.  In this case, we defined $\zeta$ to be
the composite
\[ A_\phi \odot \widehat{I_A}
\xto{\xi} I_A
\too (\pi_A)^* (\pi_A)_! I_A
\xto{\cong} \widehat{I_A} \odot \Sigma(A)
\]
where $\xi$ is as above.

\begin{prop}\label{thm:zeta-circ-f}
  The composite $\xi\circ f$ is $\Sigma(\Delta_A\circ \phi)$.
\end{prop}
\begin{proof}
  Since $\zeta$ factors through $\xi$,  $\zeta\circ f$ factors
  through $\xi\circ f$.  Applying \autoref{prop:xi_circ_f}, we see that
  $\zeta\circ f$ is equal to the composite
  \[ \Sigma(A) \xto{\Sigma(\phi)} \Sigma(A) = (\pi_A)_! I_A
  \too (\pi_A)_! (\pi_A)^* (\pi_A)_! I_A
  \xto{\cong} \Sigma(A) \otimes \Sigma(A).
  \]
  Thus, it will suffice to show that the composite
  \begin{equation}
    (\pi_A)_! I_A
    \too (\pi_A)_! (\pi_A)^* (\pi_A)_! I_A
    \xto{\cong} \Sigma(A) \otimes \Sigma(A).\label{eq:half-zeta}
  \end{equation}
  is equal to $\Sigma(\Delta_A)$.  Now, by definition,
  $\Sigma(\Delta_A)$ is equal to the composite in
  Figure~\ref{fig:zeta1}.  We can rewrite the initial
  pseudofunctoriality isomorphisms to obtain Figure~\ref{fig:zeta2}, and
  then apply the broken zigzag identity (Figure~\ref{fig:curlycue-mate})
  to obtain Figure~\ref{fig:zeta3}.  In Figure~\ref{fig:zeta4} we
  isotope sideways, and then in Figure~\ref{fig:zeta5} we slide the
  adjunction unit to the top.  Finally, we can cancel the isomorphisms
  in the middle by pseudofunctor coherence, obtaining
  Figure~\ref{fig:zeta6}, which is exactly~\eqref{eq:half-zeta}.
\end{proof}

\begin{figure}
  \centering
  \subfigure[Step 1]{\label{fig:zeta1}
    \begin{tikzpicture}[strings]
      \def\yy{0}
      \node[pbeflat] (d1\yy) at (1,\yy,0) {};
      \node[pfeflat] (d4\yy) at (4,\yy,0) {};
      \draw (d1\yy) -- (d4\yy);

      \def\yy{1}
      \node[pbeflat] (d0\yy) at (0,\yy,2) {};
      \node[pbeflat] (d1\yy) at (1,\yy,0) {};
      \node[pbeflat] (d2\yy) at (2,\yy,1) {};
      \node[pfeflat] (d3\yy) at (3,\yy,1) {};
      \node[pfeflat] (d4\yy) at (4,\yy,2) {};
      \node[pfeflat] (d5\yy) at (5,\yy,0) {};
      \draw (d0\yy) -- (d2\yy);
      \draw (d1\yy) -- (d2\yy) -- (d3\yy) -- (d4\yy);
      \draw (d3\yy) -- (d5\yy);

      \def\yy{2}
      \node[pbeflat] (d0\yy) at (0,\yy,1) {};
      \node[pbeflat] (d1\yy) at (1,\yy,0) {};
      \node[pfeflat] (d4\yy) at (4,\yy,1) {};
      \node[pfeflat] (d5\yy) at (5,\yy,0) {};
      \draw (d0\yy) -- (d4\yy);
      \draw (d1\yy) -- (d5\yy);
      
      \begin{scope}[inner]
        \node[inneriso] (iso1) at (1,.5,0) {};
        \draw (d10) -- (iso1) -- (d11) -- (d12);
        \draw (iso1) -- (d01) -- (d02);
        \draw (iso1) -- (d21);
        \node[inneriso] (iso4) at (4,.5,0) {};
        \draw (d40) -- (iso4) -- (d41) -- (d42);
        \draw (iso4) -- (d51) -- (d52);
        \draw (iso4) -- (d31);
        \draw (d31) to[out=-90,in=-90,looseness=2] node[innercounit] {} (d21);
      \end{scope}
    \end{tikzpicture}}
  \hspace{1cm}
  \subfigure[Step 2]{\label{fig:zeta2}
    \begin{tikzpicture}[strings]
      \def\yy{0}
      \node[pbeflat] (d0\yy) at (0,\yy,0) {};
      \node[pfeflat] (d5\yy) at (5,\yy,0) {};
      \draw (d0\yy) -- (d5\yy);

      \def\yy{1}
      \node[pbeflat] (d0\yy) at (0,\yy,2) {};
      \node[pbeflat] (d1\yy) at (1,\yy,0) {};
      \node[pbeflat] (d2\yy) at (2,\yy,1) {};
      \node[pfeflat] (d3\yy) at (3,\yy,1) {};
      \node[pfeflat] (d4\yy) at (4,\yy,2) {};
      \node[pfeflat] (d5\yy) at (5,\yy,0) {};
      \draw (d0\yy) -- (d2\yy);
      \draw (d1\yy) -- (d2\yy) -- (d3\yy) -- (d4\yy);
      \draw (d3\yy) -- (d5\yy);

      \def\yy{2}
      \node[pbeflat] (d0\yy) at (0,\yy,1) {};
      \node[pbeflat] (d1\yy) at (1,\yy,0) {};
      \node[pfeflat] (d4\yy) at (4,\yy,1) {};
      \node[pfeflat] (d5\yy) at (5,\yy,0) {};
      \draw (d0\yy) -- (d4\yy);
      \draw (d1\yy) -- (d5\yy);
      
      \begin{scope}[inner]
        \draw (d00) -- (d01) -- (d02);
        \draw (d50) -- (d51) -- (d52);
        \draw (d11) -- (d12);
        \draw (d41) -- (d42);
        \draw (d11) to[out=90,in=90,looseness=3] node[inneriso] {} (d21);
        \draw (d31) to[out=90,in=90,looseness=3] node[inneriso] {} (d41);
        \draw (d31) to[out=-90,in=-90,looseness=2] node[innercounit] {} (d21);
      \end{scope}
    \end{tikzpicture}}
  \\
  \subfigure[Step 3]{\label{fig:zeta3}
    \begin{tikzpicture}[strings]
      \def\yy{0}
      \node[pbeflat] (d0\yy) at (0,\yy,0) {};
      \node[pfeflat] (d5\yy) at (5,\yy,0) {};
      \draw (d0\yy) -- (d5\yy);

      \def\yy{1}
      \node[pbeflat] (d0\yy) at (0,\yy,2) {};
      \node[pbeflat] (d1\yy) at (1,\yy,0) {};
      \node[pbeflat] (d4\yy) at (4,\yy,1) {};
      \node[pfeflat] (d5\yy) at (5,\yy,1) {};
      \draw (d0\yy) -- (d4\yy);
      \draw (d1\yy) -- (d4\yy) -- (d5\yy);

      \def\yy{2}
      \node[pbeflat] (d0\yy) at (0,\yy,2) {};
      \node[pbeflat] (d1\yy) at (1,\yy,0) {};
      \node[pfeflat] (d2\yy) at (2,\yy,2) {};
      \node[pbeflat] (d3\yy) at (3,\yy,2) {};
      \node[pbeflat] (d4\yy) at (4,\yy,1) {};
      \node[pfeflat] (d5\yy) at (5,\yy,1) {};
      \draw (d0\yy) -- (d2\yy);
      \draw (d3\yy) -- (d4\yy);
      \draw (d1\yy) -- (d4\yy) -- (d5\yy);

      \def\yy{3}
      \node[pbeflat] (d0\yy) at (0,\yy,1) {};
      \node[pbeflat] (d1\yy) at (1,\yy,0) {};
      \node[pfeflat] (d2\yy) at (2,\yy,1) {};
      \node[pfeflat] (d5\yy) at (5,\yy,0) {};
      \draw (d0\yy) -- (d2\yy);
      \draw (d1\yy) -- (d5\yy);
      
      \begin{scope}[inner]
        \draw (d00) -- (d01) -- (d02) -- (d03);
        \draw (d50) -- (d51) -- (d52) -- (d53);
        \draw (d11) -- (d12) -- (d13);
        \draw (d41) -- (d42);
        \draw (d11) to[out=90,in=90,looseness=1] node[inneriso] {} (d41);
        \draw (d22) to[out=90,in=90,looseness=3] node[innerunit] {} (d32);
        \draw (d32) to[out=-90,in=-90,looseness=2] node[inneriso] {} (d42);
        \draw (d22) -- (d23);
      \end{scope}
    \end{tikzpicture}}
  \hspace{1cm}
  \subfigure[Step 4]{\label{fig:zeta4}
    \begin{tikzpicture}[strings]
      \def\yy{0}
      \node[pbeflat] (d0\yy) at (0,\yy,0) {};
      \node[pfeflat] (d5\yy) at (5,\yy,0) {};
      \draw (d0\yy) -- (d5\yy);

      \def\yy{1}
      \node[pbeflat] (d0\yy) at (0,\yy,2) {};
      \node[pbeflat] (d1\yy) at (3,\yy,0) {};
      \node[pbeflat] (d4\yy) at (4,\yy,1) {};
      \node[pfeflat] (d5\yy) at (5,\yy,1) {};
      \draw (d0\yy) -- (d4\yy);
      \draw (d1\yy) -- (d4\yy) -- (d5\yy);

      \def\yy{2}
      \node[pbeflat] (d0\yy) at (0,\yy,2) {};
      \node[pbeflat] (d1\yy) at (3,\yy,0) {};
      \node[pfeflat] (d2\yy) at (1,\yy,2) {};
      \node[pbeflat] (d3\yy) at (2,\yy,2) {};
      \node[pbeflat] (d4\yy) at (4,\yy,1) {};
      \node[pfeflat] (d5\yy) at (5,\yy,1) {};
      \draw (d0\yy) -- (d2\yy);
      \draw (d3\yy) -- (d4\yy);
      \draw (d1\yy) -- (d4\yy) -- (d5\yy);

      \def\yy{3}
      \node[pbeflat] (d0\yy) at (0,\yy,1) {};
      \node[pbeflat] (d1\yy) at (3,\yy,1) {};
      \node[pfeflat] (d2\yy) at (1,\yy,1) {};
      \node[pfeflat] (d5\yy) at (5,\yy,1) {};
      \draw (d0\yy) -- (d2\yy);
      \draw (d1\yy) -- (d5\yy);

      \def\yy{4}
      \node[pbeflat] (d0\yy) at (0,\yy,1) {};
      \node[pbeflat] (d1\yy) at (1,\yy,0) {};
      \node[pfeflat] (d2\yy) at (4,\yy,1) {};
      \node[pfeflat] (d5\yy) at (5,\yy,0) {};
      \draw (d0\yy) -- (d2\yy);
      \draw (d1\yy) -- (d5\yy);
      
      \begin{scope}[inner]
        \draw (d00) -- (d01) -- (d02) -- (d03) -- (d04);
        \draw (d50) -- (d51) -- (d52) -- (d53) -- (d54);
        \draw (d11) -- (d12) -- (d13) -- (d14);
        \draw (d41) -- (d42);
        \draw (d11) to[out=90,in=90,looseness=2] node[inneriso] {} (d41);
        \draw (d22) to[out=90,in=90,looseness=3] node[innerunit] {} (d32);
        \draw (d32) to[out=-90,in=-90,looseness=1] node[inneriso,near start] {} (d42);
        \draw[white,line width=4pt] (d23) -- (d24);
        \draw (d22) -- (d23) -- (d24);
      \end{scope}
    \end{tikzpicture}}
  \\
  \subfigure[Step 5]{\label{fig:zeta5}
    \begin{tikzpicture}[strings]
      \def\yy{0}
      \node[pbeflat] (d0\yy) at (0,\yy,0) {};
      \node[pfeflat] (d5\yy) at (5,\yy,0) {};
      \draw (d0\yy) -- (d5\yy);

      \def\yy{1}
      \node[pbeflat] (d0\yy) at (0,\yy,0) {};
      \node[pfeflat] (d1\yy) at (1,\yy,0) {};
      \node[pbeflat] (d2\yy) at (2,\yy,0) {};
      \node[pfeflat] (d5\yy) at (5,\yy,0) {};
      \draw (d0\yy) -- (d1\yy);
      \draw (d2\yy) -- (d5\yy);

      \def\yy{2}
      \node[pbeflat] (d0\yy) at (0,\yy,1) {};
      \node[pfeflat] (d1\yy) at (1,\yy,1) {};
      \node[pbeflat] (d2\yy) at (2,\yy,2) {};
      \node[pbeflat] (d3\yy) at (3,\yy,0) {};
      \node[pbeflat] (d4\yy) at (4,\yy,1) {};
      \node[pfeflat] (d5\yy) at (5,\yy,1) {};
      \draw (d0\yy) -- (d1\yy);
      \draw (d2\yy) -- (d4\yy) -- (d5\yy);
      \draw (d3\yy) -- (d4\yy);

      \def\yy{3}
      \node[pbeflat] (d0\yy) at (0,\yy,0) {};
      \node[pfeflat] (d1\yy) at (1,\yy,0) {};
      \node[pbeflat] (d3\yy) at (3,\yy,0) {};
      \node[pfeflat] (d5\yy) at (5,\yy,0) {};
      \draw (d0\yy) -- (d1\yy);
      \draw (d3\yy) -- (d5\yy);

      \def\yy{4}
      \node[pbeflat] (d0\yy) at (0,\yy,1) {};
      \node[pbeflat] (d3\yy) at (1,\yy,0) {};
      \node[pfeflat] (d1\yy) at (4,\yy,1) {};
      \node[pfeflat] (d5\yy) at (5,\yy,0) {};
      \draw (d0\yy) -- (d1\yy);
      \draw (d3\yy) -- (d5\yy);

      \begin{scope}[inner]
        \draw (d00) -- (d01) -- (d02) -- (d03) -- (d04);
        \draw (d50) -- (d51) -- (d52) -- (d53) -- (d54);
        \draw (d11) to[out=90,in=90,looseness=3] node[innerunit] {} (d21);
        \draw (d21) -- (d22);
        \draw (d22) to[out=-90,in=-90,looseness=1] node[inneriso,near start] {} (d42);
        \draw (d42) to[out=90,in=90,looseness=2] node[inneriso,near start] {} (d32);
        \draw (d32) -- (d33) -- (d34);
        \draw[white,line width=4pt] (d13) -- (d14);
        \draw (d11) -- (d12) -- (d13) -- (d14);
      \end{scope}
    \end{tikzpicture}}
\hspace{1cm}
  \subfigure[Step 6]{\label{fig:zeta6}
    \begin{tikzpicture}[strings]
      \def\yy{0}
      \node[pbeflat] (d0\yy) at (0,\yy,0) {};
      \node[pfeflat] (d3\yy) at (3,\yy,0) {};
      \draw (d0\yy) -- (d3\yy);

      \def\yy{1}
      \node[pbeflat] (d0\yy) at (0,\yy,0) {};
      \node[pfeflat] (d1\yy) at (1,\yy,0) {};
      \node[pbeflat] (d2\yy) at (2,\yy,0) {};
      \node[pfeflat] (d3\yy) at (3,\yy,0) {};
      \draw (d0\yy) -- (d1\yy);
      \draw (d2\yy) -- (d3\yy);

      \def\yy{2}
      \node[pbeflat] (d0\yy) at (0,\yy,1) {};
      \node[pbeflat] (d2\yy) at (1,\yy,0) {};
      \node[pfeflat] (d1\yy) at (2,\yy,1) {};
      \node[pfeflat] (d3\yy) at (3,\yy,0) {};
      \draw (d0\yy) -- (d1\yy);
      \draw (d2\yy) -- (d3\yy);

      \begin{scope}[inner]
        \draw (d00) -- (d01) -- (d02);
        \draw (d30) -- (d31) -- (d32);
        \draw (d11) to[out=90,in=90,looseness=3] node[innerunit] {} (d21);
        \draw (d21) -- (d22);
        \draw[white,line width=4pt] (d11) -- (d12);
        \draw (d11) -- (d12);
      \end{scope}
    \end{tikzpicture}}
  \caption{The last half of $\zeta$}
  \label{fig:half-zeta}
\end{figure}
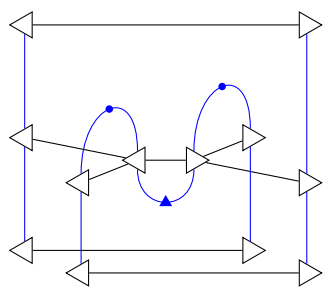
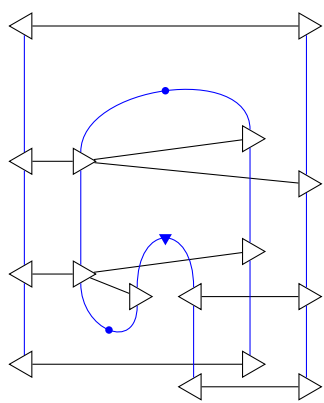
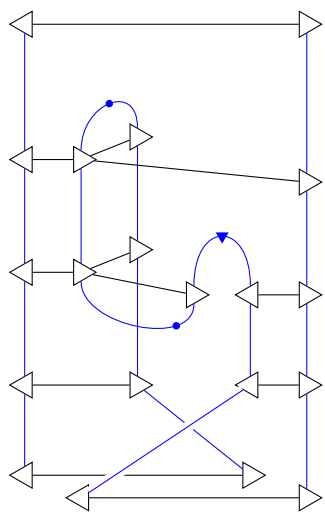
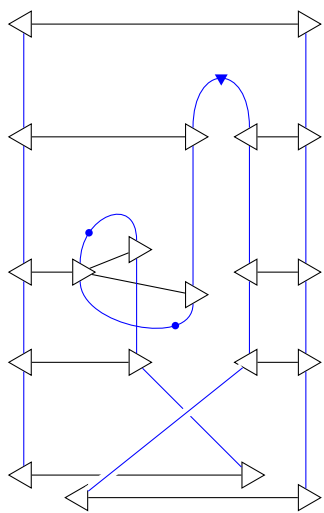
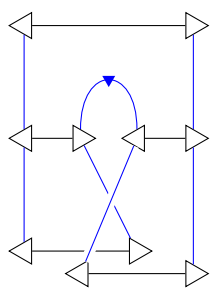

\end{document}